%% file: Article_obstruction.tex
\documentclass{amsart}

\input{boilerplate}

\begin{document}%

\title{Spanning the isogeny class of a power of an elliptic curve}

\author[Kirschmer]{Markus Kirschmer}
\address{Universit\"at Paderborn, Fakult\"at EIM, Institut f\"ur Mathematik, Warburger Str. 100, 33098 Paderborn, Germany}
\email{markus.kirschmer@math.upb.de}

\author[Narbonne]{Fabien Narbonne}
\address{Univ Rennes, CNRS, IRMAR - UMR 6625, F-35000
 Rennes, %
  France. %
}
\email{fabien.narbonne@univ-rennes1.fr}

\author[Ritzenthaler]{Christophe Ritzenthaler}
\address{Univ Rennes, CNRS, IRMAR - UMR 6625, F-35000
 Rennes, %
  France. }
\email{christophe.ritzenthaler@univ-rennes1.fr}

\author[Robert]{Damien Robert}
\address{
INRIA Bordeaux–Sud-Ouest,
200 avenue de la Vieille Tour,
33405 Talence Cedex, France and
Institut de Math\'ematiques de Bordeaux,
351 cours de la liberation,
33405 Talence cedex, France}
\email{damien.robert@inria.fr}



\date{April 2020}

\begin{abstract}
\noindent
 Let $E$ be an ordinary elliptic curve over a finite field and $g$ be a positive integer. Under some technical assumptions, we give an algorithm to span the  isomorphism classes of principally polarized abelian varieties in the isogeny class of $E^g$. The varieties are first described as hermitian lattices over (not necessarily maximal) quadratic orders and then geometrically in terms of their algebraic theta null point. We also show how to  algebraically compute Siegel modular forms of even weight given as polynomials in the theta constants by a careful choice of an affine lift of the theta null point. We then use these results to give an algebraic computation of Serre's obstruction for principally polarized abelian threefolds isogenous to $E^3$ and of the Igusa modular form in dimension $4$. We illustrate our algorithms with examples of curves with many rational points over finite fields.

\end{abstract}

\maketitle

\input{sections/introduction}

\input{sections/hermitianlattices}

\input{sections/classificationlattices}

\input{sections/orthogonal}

\input{sections/equivalence}

\input{sections/theta2}

\input{sections/serreobs2}

\input{sections/examplesg4}

\printbibliography
\end{document}

%% file: boilerplate.tex
\usepackage{ifluatex}
\ifluatex
  \usepackage{fontspec}
  \usepackage{luatextra}
  \defaultfontfeatures{Numbers=OldStyle}
  \setromanfont{Minion Pro}
  \usepackage{unicode-math}
  \unimathsetup{math-style=ISO}
\else
  \usepackage[utf8]{inputenc}
  \usepackage[T1]{fontenc}
  \usepackage{amssymb}
\fi
\usepackage{fullpage}
\usepackage[x11names,svgnames,dvipsnames]{xcolor}
\usepackage{hyperref,bookmark}
\makeatletter
\g@addto@macro{\UrlBreaks}{\UrlOrds}
\makeatother

\def\biblatexbackend{bibtex}
\ifluatex
\def\biblatexbackend{biber}
\fi
\usepackage[%
style=alphabetic,hyperref=true,backref=true,sorting=nyt,
block=space,firstinits=true,maxcitenames=3,mincitenames=3,maxalphanames=3,minalphanames=3,maxbibnames=99,
backend=\biblatexbackend
]{biblatex}
\bibliography{biblio}

\let\oldmathcal\mathcal
\usepackage{mathtools,amscd,amsthm}
\usepackage{mathrsfs,calrsfs,dsfont,mathbbol}
\usepackage[all]{xy}
\usepackage{lipsum}
\usepackage[framemethod=TikZ]{mdframed}
\usepackage[capitalise]{cleveref}
\usepackage{csquotes}
\usepackage{relsize}

\definecolor{webbrown}{rgb}{.6,0,0}

\providecommand\linkcolor{MidnightBlue}
\providecommand\urlcolor{webbrown}
\providecommand\citecolor{green}
\providecommand\anchorcolor{red}
\hypersetup{
  colorlinks,linkcolor=\linkcolor,anchorcolor=\anchorcolor,citecolor=\citecolor,
  urlcolor=\urlcolor,
  breaklinks,pdfborder={0 0 0},linktoc=page,
  pdfhighlight=/O, pdfpagemode=UseOutlines,
  unicode=true
}
\bookmarksetup{numbered=true,open=true,openlevel=1,view={FitH 0}}
\newtagform{color}{(\colorlet{default}{.}\color{\linkcolor}}{\color{default})}
\usetagform{color}

\newtheorem{theorem}{Theorem}[section]
\newtheorem{lemma}[theorem]{Lemma}

\newtheorem{corollary}[theorem]{Corollary}
\newtheorem{proposition}[theorem]{Proposition}

\theoremstyle{definition}
\newtheorem{definition}[theorem]{Definition}
\theoremstyle{remark}
\newtheorem{remark}[theorem]{Remark}
\newtheorem{example}[theorem]{Example}

\usepackage{algorithm,algpseudocode}

\algnewcommand{\IIf}[1]{\State\algorithmicif\ #1\ \algorithmicthen}
\algnewcommand{\EndIIf}{\unskip\ \algorithmicend\ \algorithmicif}

\usetikzlibrary{arrows,chains,matrix,positioning,scopes}%

\DeclareRobustCommand\myparagraph[1]{\paragraph{\textbf{#1}}}

\DeclareMathOperator{\Mat}{Mat}
\DeclareMathOperator{\Sp}{Sp}
\DeclareMathOperator{\Spec}{Spec}

\DeclareMathOperator{\Jac}{Jac}
\DeclareMathOperator{\lcm}{lcm}
\DeclareMathOperator{\GL}{GL}

\DeclareMathOperator{\Hom}{Hom}
\DeclareMathOperator{\Frac}{Frac}
\DeclareMathOperator{\diag}{diag}

\DeclareMathOperator{\Coker}{Coker}
\DeclareMathOperator{\Ab}{Ab}

\DeclareMathOperator{\Gram}{Gram}

\DeclareMathOperator{\SL}{SL}
\DeclareMathOperator{\car}{char}
\DeclareMathOperator{\Tr}{Tr}
\DeclareMathOperator{\Nr}{Nr}
\DeclareMathOperator{\Aut}{Aut}
\DeclareMathOperator{\gen}{gen}
\DeclareMathOperator{\cls}{cls}
\DeclareMathOperator{\GU}{U}
\DeclareMathOperator{\SU}{SU}
\DeclareMathOperator{\Mass}{Mass}

\newcommand{\defby}{\coloneqq}
\newcommand{\trsp}[1]{{}^{t}\!#1}
\newcommand{\OO}{\mathcal{O}}

\newcommand{\fp}{\mathfrak{p}}

\newcommand{\fa}{\mathfrak{a}}
\newcommand{\fb}{\mathfrak{b}}

\newcommand{\QQ}{\mathds{Q}}
\newcommand{\Gm}{\mathds{G}_m}
\newcommand{\FF}{\mathds{F}}
\newcommand{\Fq}{\FF_q}

\newcommand{\HHb}{\mathds{H}}
\newcommand{\ZZ}{\mathds{Z}}
\newcommand{\ZS}{\ZZ^{*2}}
\newcommand{\RR}{\mathds{R}}
\newcommand{\PP}{\mathds{P}}
\newcommand{\Primes}{\mathcal{P}}
\newcommand{\RLat}{R-\text{Mod}_{\text{f.p}}}
\newcommand{\NN}{\mathds{N}}
\newcommand{\CC}{\mathds{C}}

\providecommand{\C}{}
\renewcommand{\C}{\mathcal{C}}

\newcommand{\iso}{\ensuremath{\simeq}}      %

\newcommand{\transp}[1]{\prescript{t}{}{#1}}

\newcommand{\A}{\mathcal{A}}

\newcommand{\E}{\mathcal{E}}
\newcommand{\F}{\mathcal{F}_E}
\providecommand{\G}{}
\renewcommand{\G}{\mathcal{G}}
\newcommand{\LL}{\mathcal{L}}
\newcommand{\End}{\text{End}}
\newcommand{\Id}{\text{Id}}
\newcommand{\Ld}{L^{\#}}
\newcommand{\s}{^{*}}

\newcommand{\Lcal}{\mathcal{L}}
\newcommand{\pol}{\Lcal}

\let\Mpol=\Mcal
\newcommand{\bpol}{\mathscr{M}}
\newcommand{\carac}[2]{%
    \left[\begin{smallmatrix}%
    #1 \\
    #2
    \end{smallmatrix}\right]
}
\newcommand{\thetacarac}[2]{\vartheta\carac{#1}{#2}}
\newcommand{\thetacaraco}[2]{\theta\carac{#1}{#2}}
\let\thetacar\thetacarac

\newcommand{\Zstruct}[1]{\ensuremath{Z({#1})}}
\newcommand{\dZstruct}[1]{\hat{Z}(#1)}

\newcommand{\overln}{\ensuremath{\overline{\ell n}}}
\newcommand{\overn}{\ensuremath{\overline{n}}}
\newcommand{\overl}{\ensuremath{\overline{\ell}}}

\newcommand{\Zln}{\Zstruct{\overln}}
\newcommand{\Zn}{\Zstruct{\overn}}
\newcommand{\Zl}{\Zstruct{\overl}}

\newcommand{\dZn}{\dZstruct{\overn}}
\newcommand{\dZl}{\dZstruct{\overl}}

\newcommand{\overtwo}{\ensuremath{\overline{2}}}
\newcommand{\Ztwo}{\Zstruct{\overtwo}}
\newcommand{\dZtwo}{\dZstruct{\overtwo}}
\newcommand{\Zfour}{\Zstruct{\overline{4}}}
\newcommand{\Hstruct}[1]{\ensuremath{\oldmathcal{H}(#1)}}
\newcommand{\Hn}{\Hstruct{\overn}}
\newcommand{\Hln}{\Hstruct{\overln}}
\newcommand{\Agln}{\mathcal{A}_{g,\ell n}}
\newcommand{\Agn}{\mathcal{A}_{g,n}}
\newcommand{\Abargn}{\mathcal{\overline{A}}_{g,n}}
\newcommand{\Abargln}{\mathcal{\overline{A}}_{g,\ell n}}
\newcommand{\Ag}{\mathcal{A}_{g,1}}
\newcommand{\Agthree}{\mathcal{A}_{3,1}}
\newcommand{\Agfour}{\mathcal{A}_{4,1}}
\newcommand{\Xgn}{\mathcal{X}_{g,n}}
\newcommand{\Xbargn}{\mathcal{\overline{X}}_{g,n}}
\newcommand{\Xgln}{\mathcal{X}_{g,\ell n}}
\newcommand{\Xbargln}{\mathcal{\overline{X}}_{g,\ell n}}
\newcommand{\Agrln}{\mathcal{A}_{rg,\ell n}}
\newcommand{\Xgrln}{\mathcal{X}_{rg,\ell n}}
\newcommand{\Xgrn}{\mathcal{X}_{rg,n}}
\newcommand{\Agrn}{\mathcal{A}_{rg,n}}
\newcommand{\Hodge}{\oldmathcal{H}}

\newcommand{\kbar}{\overline{k}}
\newcommand{\pibar}{\overline{\pi}}
\newcommand{\Pibar}{\overline{\Pi}}
\newcommand{\Tbar}{\overline{T}}
\newcommand{\Tmod}{\mathcal{T}_{g,n,\ell}}
\newcommand{\ftilde}{\tilde{f}}

\newcommand{\Wedge}{\wedge}

\makeatletter
\newskip\@bigflushglue \@bigflushglue = -100pt plus 1fil

\def\bigcentering{\let\\\@centercr\rightskip\@bigflushglue%
\leftskip\@bigflushglue
\parindent\z@\parfillskip\z@skip}

\makeatother

\tikzset{
  joinedge/.code 2 args=\tikzset{after node path={%
  \ifx\tikzchainprevious\pgfutil@empty\else(\tikzchainprevious)%
  edge[#1]#2(\tikzchaincurrent)\fi}},
  joinnode/.code=\tikzset{after node path={%
  \ifx\tikzchainprevious\pgfutil@empty\else(\tikzchainprevious)%
  edge[every join]#1(\tikzchaincurrent)\fi}},
  joinlabel/.code=\tikzset{joinnode={node [above] {#1}}},
  joinbijlabel/.code=\tikzset{joinnode={node [isbij,above] {#1}}},
  joinbij/.code=\tikzset{joinnode={node [bij] {$\sim$}}},
  isbij/.style={after node path={node [bij] {$\sim$}}},
  chainjoin/.style={every on chain/.append style={join}},
  nodeonchain/.style={every node/.append style=on chain},
  mysmallchain/.style={start chain,node distance=2em, every join/.style={myarrow}},
  mychain/.style={mysmallchain, node distance=3em },
  equal/.style={-,double distance=2pt},
  myarrow/.style={->,shorten >=1pt,shorten <=1pt, >=stealth'},
  bij/.style={anchor=base,sloped,inner sep=0.2pt},
  description/.style={fill=white,inner sep=2pt},
  mydiag0/.style={myarrow, text height=1.5ex, text depth=0.25ex},
  mydiag/.style={auto,mydiag0},
  mymatrix0/.style={matrix of nodes, nodes in empty cells},
  mymatrix/.style={row sep=3em, column sep=3em, mymatrix0},
  mycompressedmatrix/.style={row sep=3em, column sep=2em, mymatrix0},
  mysmallmatrix/.style={row sep=1em, column sep=1em, mymatrix0},
  mathmatrix/.style={matrix of math nodes, nodes in empty cells},
  beamermatrix0/.style={mymatrix0,ampersand replacement=\&},
  beamermatrix/.style={beamermatrix0,row sep=1cm, column sep=1.5cm},
  beamersmallmatrix/.style={beamermatrix0, row sep=1em, column sep=1em},
  beamermathmatrix/.style={nodes in empty cells, matrix of math nodes, ampersand replacement=\&,row sep=1.0cm,column sep=1.5cm},
  information text/.style={rounded corners,fill=blue!10,inner sep=2ex},
}

\RequirePackage{twoopt}
\newcommandtwoopt{\exact}[5][][]{%
\begin{tikzpicture}[mydiag]
    {[mychain]
    \node[on chain,join] {$0$} ;
    \node[on chain,join] {$#3$} ;
    \node[on chain,joinlabel={$#1$}] {$#4$};
    \node[on chain,joinlabel={$#2$}] {$#5$};
    \node[on chain,join] {$0$};
    }
\end{tikzpicture}}%
\newcommandtwoopt{\exactfull}[7][][]{%
\begin{tikzpicture}[mydiag]
    {[mychain]
    \node[on chain,join] {$#3$} ;
    \node[on chain,join] {$#4$} ;
    \node[on chain,joinlabel={$#1$}] {$#5$};
    \node[on chain,joinlabel={$#2$}] {$#6$};
    \node[on chain,join] {$#7$};
    }
\end{tikzpicture}}%

\makeatletter
\def\fonction{\@ifnextchar[{\@fonctions}{\@fonction}}

\def\@fonction#1{\@ifnextchar[{\@fonctionb{#1}}{\@fonctionh{#1}}}
\def\@fonctionh#1#2#3{\@ifnextchar[{\@fonctionhb{#1}{#2}{#3}}{\@fonctionhh{#1}{#2}{#3}}}
\def\sidispl@y#1#2{\mathchoice{#1}{#2}{#2}{#2}}
\def\@fonctionhb#1#2#3[#4][#5]{\sidispl@y{\begin{array}[t]{r@{\ }rcl}
        \displaystyle #1\colon  & \displaystyle #2 & \longrightarrow &
\displaystyle #3 \\
                & \displaystyle #4 & \longmapsto & \displaystyle #5
\end{array}}{#1:\ #2\to #3,\ #4\mapsto #5}}
\def\@fonctionhh#1#2#3{#1\colon#2\to #3}
\def\@fonctionb#1[#2][#3]{#1\colon#2\mapsto #3}
\def\@fonctions[#1]{\@ifnextchar[{\@fonctionsb}{\@fonctionsh}}
\def\@fonctionsh#1#2{\@ifnextchar[{\@fonctionshb{#1}{#2}}{\@fonctionshh{#1}{#2}}
}
\def\@fonctionshb#1#2[#3][#4]{\sidispl@y{\begin{array}[t]{rcl}
        \displaystyle #1 & \longrightarrow & \displaystyle #2 \\
        \displaystyle #3 & \longmapsto & \displaystyle #4
\end{array}}{#1\to #2,\ #3\mapsto #4}}
\def\@fonctionshh#1#2{#1\to #2}
\def\@fonctionsb[#1][#2]{#1\mapsto #2}
\makeatother

%% file: sections/introduction.tex
\section{Introduction}
Let $g,m \geq 1$ be  integers, $p$ be a prime, $q=p^m$ and $\mathcal{W}$ be the isogeny class of  a given dimension-$g$ abelian variety $A$ over $\FF_q$. The elements of $\mathcal{W}$ will be the $\FF_q$-isomorphism classes of abelian varieties over $\FF_q$ which are $\FF_q$-isogenous to $A$. Thanks to the work of Tate \cite{tate} and Honda \cite{honda}, one knows that  
the Weil polynomial $W$ is an invariant on $\mathcal{W}$. One can also characterize the finite list $S(q,g)$ of possible Weil polynomials for given $q$ and $g$. These finite lists have  been made explicit up to genus $5$ \cite{haloui1,haloui2,hayashida}.  Representing now an isogeny class $\mathcal{W}$ by a polynomial $W \in S(q,g)$, a harder task is to describe the finite set of elements (i.e. $\FF_q$-isomorphism classes of abelian varieties)  inside $\mathcal{W}$. Currently, there is no unified nor complete way to achieve this task. To our best knowledge, one can get a full abstract description
\begin{enumerate}
\item  for $g=1$ \cite{Wat};
\item for ordinary abelian varieties \cite{deligne, Ser, howe95, Mar, JKP+18};
\item for abelian varieties $A \sim E^g$ where $E$ is a supersingular elliptic curve either over $\FF_p$ or over $\FF_{p^2}$ with trace $\pm 2 p$; \cite{JKP+18};
\item when $q=p$ and $W$ has no real root \cite{centeleghe};
\item for $p$-rank $g-1$ simple abelian varieties over fields of odd characteristics \cite{oswal}.
\end{enumerate}
Roughly speaking, the above descriptions functorially  relate $\FF_q$-isomorphism classes of (non-polarized) abelian varieties in $\mathcal{W}$ and certain finitely generated modules over orders in products of number fields or quaternion algebras.  
Notice that even for $g=2$, the situation is still incomplete as far as we know: there are only partial results for supersingular and superspecial abelian surfaces \cite{ibukiyama,chiafu, nart} and $p$-rank 1 split isogeny classes seem untouched.

The situation is even more critical if one is interested in $\FF_q$-isomorphism classes of \emph{polarized} abelian varieties in $\mathcal{W}$. Since the distinction is
important for one of our goal (identifying Jacobians in the isogeny class),
we denote the $\FF_q$-isomorphism classes of principally polarized abelian varieties isogenous to  $A$
by $\mathcal{W}_1$. Notice that there is  no inclusion between the elements of $\mathcal{W}$ and $\mathcal{W}_1$ since the notions of isomorphism classes are distinct.
When  the abelian varieties in $\mathcal{W}$ are isogenous to products of non-isogenous ordinary simple abelian varieties, there are algorithms  to enumerate the elements of $\mathcal{W}$ or  $\mathcal{W}_1$  (see \cite{Mar}). The LMFDB database is currently keeping  track of the cardinality of these sets for small values of $g$ and $q$ \cite{LMFDB}.\\

In the present article, we consider a different case from \cite{Mar}, namely $\mathcal{W}$ is the isogeny class of the $g$-th power of an ordinary elliptic curve $E/\FF_q$. Let $\pi$ be the Frobenius endomorphism of $E$ and $R=\ZZ[\pi,q/\pi]=\ZZ[\pi]$.
The set $S_E$ of $\FF_q$-isomorphism classes of elliptic curves $\{E_1,\ldots,E_r\}$ isogenous to $E$ is in bijection with the ideal class monoid $\textrm{ICM}(R)$ of $R$.  Moreover, it is always possible to identify in this set  or directly construct one elliptic curve isogenous to $E$ with minimal endomorphism ring, i.e. equal to $R$ (see the discussion at the beginning of Section~\ref{sec:effectivekernel}).
We will assume from now on that this is our curve $E$. The functor given in \cite{JKP+18} which associates to any $A \in \mathcal{W}$ the finitely generated torsion-free $R$-module (or in short $R$-lattice) $\Hom(A,E)$ of rank $g$ is an equivalence of categories and provides an inverse denoted $\F$. Note that this functor is distinct from the one used for instance in \cite{Mar} (it is contravariant and exact) and there is no easy way to compare them away from projective $R$-modules. But both functors lead to the conclusion that the elements in $\mathcal{W}$ are represented by products of elliptic curves $E_1,\ldots,E_g$ in $S_E$ corresponding to a sequence of orders $R \subset \End(E_1) \subset \ldots \subset \End(E_g)$ and  invertible $\End(E_i)$-ideal classes $I_i$ with a given fixed product $I_1 \cdots I_g$ in $\textrm{ICM}(R)$  (see \cite[Th.1]{kani}, \cite{Mar}, \cite[Th.3.2]{JKP+18}).
 
Since we are interested in $\FF_q$-isomorphisms classes of polarized abelian varieties, we need to translate the notion of polarization in the category of $R$-lattices through the functor $\Hom(A,E)$. We show in Theorem~\ref{th:equivalencepol} and Corollary~\ref{cor:equivalencepol} that this can indeed be done: the elements in $\mathcal{W}_1$ are in correspondence with the unimodular positive definite hermitian $R$-lattice $(L,h)$ of rank $g$ (see Section~\ref{subsection1} for a review on these notions for lattices). This result is no surprise to the specialists as it generalizes  a similar result of Serre \cite[Appendix]{Lau} when $R$ is the maximal order in $\QQ(\pi)$ and is analogue of the result of \cite{howe95,Mar} using a different functor. \\

How to enumerate the lattices $(L,h)$? This is part of a broader and beautiful theory which has been developed for general orders in number fields or quaternion algebras. However, even in the case of imaginary quadratic orders, the algorithms have been mainly implemented in the case where $R$ is a maximal order, cf. \cite{Sch,Det}. In Section~\ref{sec:class}, we recall some elements of this theory restricted to imaginary quadratic orders and show how to adapt these algorithms when $R$ is not maximal. This generalization comes at the price of much slower algorithms which can be sped up if one restricts to lattices which are projective $R$-modules (or equivalently to abelian varieties which are products of elliptic curves with endomorphism rings isomorphic to $R$). While our method for enumerating projective $R$-modules is quite efficient,  we believe that there is still lot of room for improvements in the general case.\\

Such descriptions, though powerful, do not allow to get a real grasp on a given polarized variety $(A,\LL)$. In particular, given an abstract description of an element in $\mathcal{W}_1$, one would like for instance to see if it is the Jacobian of a curve and if so, to give an equation  of the curve. For this, we have to jump back to the algebraic geometry side and associate to the abstract description some data describing the embedding $\phi_{\LL^i}$, $i \geq 3$, of $A$ into a projective space $\PP^N$. When $p \ne 2$, Mumford showed how to extend the classical theory over $\CC$ by  using an algebraic version of the theta constants, called a \emph{theta null point}. These constants are projectively the image by $\phi_{\LL^i}$ of $0 \in A$ for a careful choice of basis of $\PP^N$. However, if this data is not available before hand for at least one principally polarized abelian variety in $\mathcal{W}_1$, the only known method to compute it is to work with a lift of $A$ and its polarization to $\CC$, perform analytic computations with enough precision, hopefully recognize algebraic numbers and eventually reduce the result over the finite field.
When $A$ is simple, this is the classical setting of the Complex Multiplication methods (see for instance \cite[Chap.18]{handbook}) but the output is heuristic when $g > 2$ \cite{sutherlandCM,streng}.

In our case, we will  take advantage that it is easy to compute the theta null point on $A_0=E^g \in \mathcal{W}_1$  with the product polarization $\LL_0$. It boils down to computing the (projective) thetanull point on $E$. The formula for their fourth power is a particular case of Thomae's formula. We will give an elementary proof of this result and show that one can take arbitrary fourth roots (see Lemma~\ref{lem:fourth} and Corollary~\ref{cor:algebraicthomae}). Doing so, we will also prepare for a `modular version' of the thetanull point that we will need later and take great care of the constant involved.

 We also show how to deduce from the lattice description $(L,h)$ of $(A,\LL) \in \mathcal{W}_1$ an isogeny $f :  A_0 \to A$ such that $f^* \LL = \LL_0^{\ell}$ for a certain $\ell \geq 1$. This is achieved by looking for $g$ orthogonal vectors of norm $\ell$ in $\Ld$ (a certain dual of $L$ for $h$), see Section~\ref{sec:ortho}.  We can then give $f$ through an explicit maximal isotropic kernel $K$ in $A_0[\ell]$, see Section~\ref{sec:effectivekernel}. The explicit \emph{isogeny formula} developed in \cite{DRniveau} allows then to transport the thetanull point on $(A_0,\LL_0)$ to the one on $(A,\LL)$.  This leads to the following overview of our algorithm.

\begin{algorithm}[H]
\begin{algorithmic}[1]
  \Require{An integer $g >1$ and the Weil polynomial $W$ of an ordinary elliptic curve over $\FF_q$ (with some technical restrictions, see the discussion below).}
 \Ensure{The theta null points of all indecomposable  principally polarized abelian varieties with Weil polynomial $W^g$.}
\State{Let $R=\ZZ[x]/(W)$ and compute an elliptic curve $E/\FF_q$ such that $\End(E)= \ZZ[\pi] \simeq R $ (see Section~\ref{sec:effectivekernel}).}
\State{Use  Algorithm~\ref{UnimodR} (resp. \ref{algo:proj}) to get a list of all (resp. all projective) indecomposable unimodular positive definite hermitian $R$-lattices $(L,h)$ up to isometry.}
\State{Apply Algorithm~\ref{algo:kerneliso} to compute a maximal isotropic kernel $K$ of an isogeny $f : E^g \to \F(L)$ for each $(L,h)$.}
\State{{\bf return} the output of Algorithm~\ref{algo:thetaquotient} on each $((E)_{i=1,\ldots,g},K)$.}
\end{algorithmic}
 \caption{Overview of the full algorithm} \label{algo:full}
\end{algorithm}

In practice, there are restrictions on the $W$ for which this algorithm is going to work. Indeed, the current implementation of the isogeny formula imposes several constraints on the kernel $K$ of $f$. We list them below, starting from what would require the most work if one intends to remove it. This should be taken with a grain of salt as it is of course impossible to predict possible obstacles without an actual study. 
\begin{enumerate}
\item The algorithm imposes $p$ to be odd since it uses theta structures of even level;
\item  The algorithm imposes to look for  $f$ such that $f^* \LL=\LL_0^{\ell}$ for an integer $\ell>0$, whereas the strategy would work with $f^* \LL$ any completely decomposable polarization. Because of this, $f$ does not always exist (see Example~\ref{ex:noell}). We give necessary and sufficient conditions for its existence in Theorem~\ref{ellid} (for instance, it does always exist is $g$ is odd);
\item  The algorithm imposes $\ell$ to be  coprime to $2p$, see Remark~\ref{rem:details}. We work out in Section~\ref{sec:ortho} a thorough local analysis of the lattices which gives a refinement of Theorem~\ref{ellid}. For instance, when $g$ is odd it is sufficient that the conductor of $R$ is odd to find such an $\ell$;
\item Even when $\ell$ is coprime to $2p$, we have to discard cases when the kernel $K$ from Algorithm~\ref{algo:full} is not isomorphic as an abstract group to $(\ZZ/\ell\ZZ)^g$. Notice that when $\ell$ is square free, $K$ is necessarily isomorphic to  $(\ZZ/\ell\ZZ)^g$ so the algorithm always works. We did not try to get a proof of the existence of such a good $\ell$ and we pragmatically chose to test the group structure of a given kernel $K$ until we get the required abstract group isomorphism.
\end{enumerate}
The full cost of the algorithm is hard to estimate: it  heavily depends on the smallest good $\ell$ one can find (when it exists) and it is an open question to find an upper bound  in terms of $R$ and $g$ for the maximum of the minimal $\ell$ for a given $\mathcal{W}_1$.  Once $\ell$ is given, a lower bound for the complexity is given by the one of Algorithm~\ref{algo:thetaquotient} which is $O(\ell^g)$. Be aware that this hides a large constant, since the computations have to be performed on the extension of $\FF_q$ where all $\ell$-torsion points of $E$ are defined. Typically, the algorithm works for a given element of $\mathcal{W}_1$ in reasonable time when $\ell$ is smaller than $41$ (resp. $19$, resp. $7$) for $g=2$ (resp. $3$, resp. $4$). Then the full cost depends also on the cardinality of $\mathcal{W}_1$ which can be computed by  \cite{HaKo2} for $g=2$ and $3$. When $R=\End(E)$ is maximal, a lower bound for this cardinality grows linearly in $(\textrm{disc}(R))^{g^2/4}$ for fixed $g$.

 The restrictions above artificially increase the smallest $\ell$ we would like to consider. We therefore urge the reader to consider Algorithm~\ref{algo:full} as a proof of concept, allowing computations which were completely out of reach before for 
various classes $\mathcal{W}_1$  in dimension $2,3$ and $4$ with $R$ maximal or not (see Section~\ref{sec:ex}). \\

We finally move to one last new algorithmic result. In Section~\ref{sec:modular}, we show how to evaluate a Siegel modular form $\chi$ of level $\Sp_{2g}(\ZZ)$ and even weight\footnote{when $g$ is odd, all of them have even weight.} at a principally polarized abelian variety $(A,\pol)/\FF_q$ when $\chi$ is defined as a homogeneous polynomial $P$ in the theta constants with coefficients in $\FF_q$. A Siegel modular form is a section of a power of the Hodge bundle on the universal abelian variety, so to give it a value only makes sense once a $\FF_q$-rational basis of regular differentials on $A$ is fixed. We show that choosing such a basis yields a particular affine lift of the theta null point on $(A,\pol)$ which we call a \emph{modular lift} (see Definition~\ref{def:modularlift}). The coordinates of a modular lift are characterized, up to a common sign,  by considering all products of two theta coordinates as Siegel modular forms of weight $1$.
Evaluating $\chi$ is then computing the value of $P$ in the coordinates of the modular lift.  We show that a certain affine version of the isogeny formula preserves the modular lift property (see Theorem~\ref{th:modular}). Since in our Thomae's formula for elliptic curves we took care of having such a modular lift, we can therefore carry it to $(A,\pol)$ through the isogeny (see Algorithm~\ref{algo:modular}) and perform the computation of the modular form on ($A,\pol)$.\\
 
As an application and in order to illustrate our algorithms, we consider curves over $\FF_q$ with many points. A curve $C$ of genus $g \geq 1$ over $\FF_q$ has at most $1+q+g \lfloor 2 \sqrt{q} \rfloor$ and when this bound is reached, we say that $C$ is a \emph{defect-$0$ curve}. The best upper bounds are known only for $g \leq 2$ and sparse families of $g,q$. If $C$ is a defect-$0$ curve, then its Jacobian $\Jac C$ is isogenous to $E^g$ where $E$ has trace $-\lfloor 2 \sqrt{q} \rfloor$. If $E$ is ordinary (which is always the case for instance when $q=p^m$ with $m=1$ or $3$ and $q \ne 2,3$ \cite[II.6.4]{Ser}), we can try to find $\Jac C$ among the indecomposable principally polarized abelian varieties $(A,\pol)$ in the isogeny class of $E^g$.  

When $g=2$, each such $(A,\pol)$ is automatically the Jacobian of a defect-$0$ curve. It is therefore enough to know that an indecomposable principally polarized abelian surface isogenous to $E^2$ exists which can already be obtained on the lattice side of the picture using \cite{Hof} and \cite[Th.3.9.1]{Ser}. Now, if one wants an equation of the curve, it can be provided using Algorithm~\ref{algo:full}.

When $g=3$, although each $(A,\pol)$ is geometrically the Jacobian of a unique curve $C/\FF_q$, there may be an obstruction, called \emph{Serre's obstruction}, for $C$ to have defect-$0$.
Fortunately, the modular form $\chi_{18}$ which is a Siegel modular form of weight $18$ defined as the product of the $36$ even theta constants determines this obstruction as we shall recall in Section~\ref{sec:ex}. Since we can compute algebraically the values of $\chi_{18}$ at all $(A,\pol)$ in the isogeny class of $E^3$, we can compute the obstruction for each of them and check if a defect-$0$ genus-3 curve exists over $\FF_q$. This gives the first \emph{provable} computation of this obstruction as, so far, one had only a heuristic method using lifting and approximations over $\CC$ \cite{Rit}.

 We conclude with an example in genus $4$. We first show that Igusa modular form cuts the locus of  Jacobians and decomposable principally polarized abelian varieties over any algebraically closed field of characteristic different $2$ (see Theorem~\ref{th:igusa}). We then use this to show that a certain class of isogeny does not contain Jacobians (see Example~\ref{ex:g4}).\\

The code and examples of our algorithms are available at
\cite{AlgebraicObstructionCode}.
In the future, we hope to improve the overall speed of the algorithm (for instance  by working with $A_0$ products of distinct elliptic curves $E_i$ instead of $E^g$) and waive the technical limitations above. 
Notice that the method presented here may be adapted to other cases: one could replace $E$ ordinary with $E$ supersingular over $\FF_p$ or over $\FF_{p^2}$ with trace $\pm 2p$; one could also replace $E$ by a principally polarized abelian variety $B$  for which a thetanull point is known (with some restrictions, see \cite{amir} and \cite[Sec.8]{JKP+18}).  

\proof[Acknowledgements]
We would like to thank Andrew Sutherland who kindly provided us a fast \texttt{Magma} code to check when an ordinary elliptic curve has minimal endomorphism ring and Jeroen Sijsling for helping us using his \texttt{Magma} packages. We also thank Valentijn Karemaker and Stefano Marseglia for discussions about the references in the introduction.

%% file: sections/hermitianlattices.tex
\section{Hermitian lattices}
\label{sec:hermitianlattices}

\subsection{Basic definitions and notations}\label{subsection1}
Let $F=\QQ(\sqrt{d})$, where $d<0$ is a squarefree negative integer.
The discriminant $d_F$ of $F$ equals $d$ if $d \equiv 1 \pmod{4}$ and $4d$ otherwise.
The non-trivial Galois involution of $F/\QQ$ will be denoted by $\bar{\cdot}$. 
Further, let 
\[ \Nr \colon F \to \QQ,\; x \mapsto x \overline{x} \quad \mbox{and} \quad \Tr \colon F \to \QQ,\; x \mapsto x + \overline{x} \]
be the usual norm and trace of $F/\QQ$.

\begin{definition}
A \emph{hermitian space} $(V, h)$ over $F$ is a finite dimensional vector space $V$ over $F$ equipped with a sesqui-linear map $h \colon V \times V \to F$ such that
\begin{enumerate}
\item $h(\alpha v + \beta v',w) = \alpha h(v,w) + \beta h(v', w) $ for all $\alpha, \beta \in F$ and all $v,v', w \in V$. 
\item $h(v,w) = \overline{h(w,v)}$ for all $v,w\in V$.
\end{enumerate}
The \emph{rank} of a hermitian space $(V,h)$ is the dimension of $V$ over $F$.
For a tuple $b = (b_1,\dots, b_r) \in V^r$ we define its \emph{Gram matrix} by
\[ \Gram(b) =(h(b_i,b_j)) \in F^{r \times r} \:. \]
\end{definition}

Every hermitian space $(V, h)$ in this paper is assumed to be \emph{non-degenerate}, i.e. if $v \in V$ with $h(v,w) = 0$ for all $w \in V$ then $v = 0$.
This is equivalent to say that the \emph{Gram matrix} of any basis $b$ of $V$ is invertible.

\begin{definition}
Let $b$ be a basis of a hermitian space $(V, h)$. Then
\[ \det(V, h):= \det(\Gram(b))\]
is called the \emph{determinant} of $(V,h)$. It is well defined when viewed as an element of $\QQ^*/\Nr(F^*)$.
\end{definition}

\begin{definition}
Two hermitian spaces $(V,h)$ and $(V', h')$ over $F$ are called \emph{isometric} if there is an isomorphism $\varphi \colon V \to V'$ such that $h'(\varphi(v), \varphi(w)) = h(v,w)$ for all $v,w \in V$.
The map $\varphi$ is then called an \emph{isometry} between $(V,h)$ and $(V', h')$.
Moreover, 
\begin{align*}
 \GU(V,h) &= \{ \varphi \colon V \to V \mid \varphi \mbox{ is an isometry} \} \quad \mbox{and}  \quad
 \SU(V,h) = \{ \varphi \in \GU(V, h) \mid \det(\varphi) = 1 \} \:.
\end{align*}
are the \emph{unitary} and \emph{special unitary groups} of $(V, h)$ respectively.
\end{definition}

Let $\Primes$ denote the set of prime numbers.
For $p \in \Primes \cup \{\infty\}$ let $F_p := \QQ_p \otimes_\QQ F$ be the completion of $F$ at $p$.
Let $(V,h)$ be a hermitian space over $F$.
The map $h$ extends to $V_p:= F_p \otimes_F V$ by linearity. This yields a hermitian space $(V_p, h)$ over $F_p$.
If $p = \infty$, then $\QQ_\infty = \RR$ and $(V_\infty, h)$ is a hermitian space over $F_\infty = \CC$.
The signature of this complex hermitian space is called the signature of $(V, h)$.

The following local-global principle is well known.

\begin{theorem}[Landherr]
Two hermitian spaces over $F$ are isometric if and only if they are isometric over every place of $\QQ$.
\end{theorem}

Hermitian spaces over $\CC$ are parameterized by their signatures while hermitian spaces over $\QQ_p$ are parameterized by their ranks and determinants (viewed as elements of $\QQ_p^*/\Nr(F_p^*)$).
We will only deal with positive definite spaces, i.e. spaces with $h(v,v) > 0$ for all non-zero $v \in V$.
For these spaces, we can make Landherr's theorem more explicit.

\begin{remark}\label{FindSpace}
Let $g$ be a positive integer and let $\Primes_{\textnormal{ns}}$ be the set of primes which do not split in $F$.
\begin{enumerate}
\item
Let $(V,h)$ be a positive definite hermitian space of rank $g$.
Since $\QQ_p^*/\Nr(F_p^*)$ has at most two elements, the isometry type of $(V,h)$ is uniquely determined by 
\[ I:= \{ p \in \Primes \mid \det(V,h) \notin \Nr(F_p^*) \} \subseteq \Primes_{\textnormal{ns}} \:. \]
The product formula for Hasse's norm residue symbols shows that $I$ is a finite set of even cardinality.
\item Let $I \subseteq \Primes_{\textnormal{ns}}$ be a finite subset of even cardinality.
There exists a positive definite hermitian space $(V, h)$ of rank $g$ such that
\[ I = \{ p \in \Primes \mid \det(V,h) \notin \Nr(F_p^*) \} \:. \]
Moreover, this space admits the Gram matrix
\[ \diag(1,\dots,1,a) \]
with some positive integer $a$ whose prime divisor are in $I \cup\{q\}$ for some prime $q$. 
This gives a method to construct a positive definite hermitian space of rank $g$ with given determinant, see \cite[Section 3.4]{Kir} for details.
\end{enumerate}
\end{remark}

For the remainder of this section, let $(V, h)$ be a hermitian space over $F$ of rank $g$.
Further let $R$ be an \emph{order} in $F$, that is a subring of $F$ which is a free $\ZZ$-module of rank $2$.
The ring of integers $\OO$ of $F$ is an order and it contains every other order $R$ of $F$.
Thus the index $f:= [\OO : R]$ is finite and it is called the \emph{conductor} of $R$ in $F$.
Note that $R$ is the unique quadratic order of discriminant $f^2 d_F$.
Moreover,
\[ \OO=\ZZ[\omega] \quad\mbox{and} \quad R=\ZZ[f\omega] \quad \mbox{where }  \omega = \frac{d_F + \sqrt{d_F}}{2}\:.\]
A \emph{fractional $R$-ideal} $\fa$ is an $R$-submodule of $F$ which has rank $2$ over $\ZZ$.
It is said to be an \emph{invertible} $R$-ideal if there exists a fractional $R$-ideal $\fb$ such that $\fa\fb=R$.
Given two fractional $R$-ideals $\fa,\fb$ we can define the fractional $R$-ideal $(\fa : \fb)=\{x\in F, x\fb\subseteq \fa\}$
called the \emph{colon-quotient} of $\fa$ and $\fb$. The particular case $(\fa : \fa)$ is called the \emph{multiplicator ring} of $\fa$. It is the unique order in $F$ for which $\fa$ is invertible.

\begin{definition}
An \emph{$R$-lattice} of rank $r$ is a finitely generated $R$-submodule of $V$ such that $FL := L\otimes_R F$ has dimension $r$.
If $r=g$ we call $L$ a \emph{full} $R$-lattice in $V$.
\end{definition}

The following result is due to Borevich and Faddeev \cite{BoFa}.

\begin{proposition}
Let $L$ be a full $R$-lattice in $V$.
Then there exist a basis $(x_1,\dots,x_g)$ of $V$, some fractional ideals $\fa_1,\dots,\fa_g$ of $R$ and a chain of orders $R \subseteq R_1\subseteq\dots\subseteq R_g$ such that $\fa_i$ is an invertible $R_i$-ideal and
\[ L=\fa_1 x_1 \oplus \dots \oplus \fa_g x_g \:. \]
The list of pairs $(\fa_i,x_i)_{i=1,\dots,g}$ is called a \emph{pseudo-basis} of $L$.
\end{proposition}

In the implementation of our algorithms we represent an $R$-lattice either via a pseudo basis or a $\ZZ$-basis
and we use the results of \cite{BoFa} to switch between these two types of representations.

\begin{definition} \label{def:decomp}
Let $L$ be an $R$-lattice in $V$.
\begin{enumerate}
\item
The \emph{dual lattice} of $L$ is
\[ L^\#=\{x \in V \mid h(x,L)\subseteq R\} \:. \]
\item The lattice $L$ is called \emph{integral} if $L\subseteq L^\#$ and \emph{unimodular} if $L=L^\#$.
\item An integral $R$-lattice $L$ is called even, if $h(x,x) \in 2\ZZ$ for all $x \in L$; otherwise it is called odd.
\item
The lattice $L$ is called \emph{decomposable} if there exists two non-trivial $R$-submodules $L_1,L_2$ of $L$ such that $L=L_1 \oplus L_2$ and $h(x_1,x_2) = 0$ for all $x_i \in L_i$.
If this is the case, we write $L = L_1 \perp L_2$.
\item If $L$ is a free $R$-lattice with basis $b$, then $\det(L):= \det(\Gram(b))$ is the \emph{determinant} of $L$.
It is a well defined element in $\QQ^* / \Nr(R^*)$. 
\item Given $a_1,\dots,a_g \in \QQ^*$, we denote by
\[ \langle a_1,\dots,a_g \rangle \]
the free hermitian $R$-lattice $(L', h')$ of rank $g$ having an orthogonal basis $(b_1,\dots,b_g)$ such that $h'(b_i,b_i) = a_i$ for all $1 \le i \le g$.
\end{enumerate}
\end{definition}

Let $L$ be an $R$-lattice with pseudo-basis $(\fa_i,x_i)$. Denote by $(x_i^\#)$ the dual basis $(x_i)$, i.e. the basis of $V$ such that $h(x_i,x_i^\#)=\delta_{i,j} $ for all $1 \le i,j \le g$.
Then 
\[ L^\#=\bigoplus_{i=1}^g \overline{(R : \fa_i)} x_i^\#\:. \]
From this fact and the relation $(R : (R :\fa))=\fa$ it is easy to see that $(L^\#)^\#=L$. 

\begin{lemma}\label{LatIndec}
Let $L$ be an $R$-lattice in $(V,h)$ and let $L_1,\dots,L_n$ be $\ZZ$-submodules of $L$.
For $a \in F$ let
\[ f_a \colon V \times V \to \QQ,\; (x,y) \mapsto \Tr(a h(x,y) ) \:. \]
The following are equivalent:
\begin{enumerate}
\item $L = L_1 \perp \ldots \perp L_n$ is an orthogonal decomposition into $R$-lattices.
\item $L = \bigoplus_i L_i$ and $f_1(L_i, L_j) = f_{\sqrt{d}}(L_i, L_j) = \{0\}$ for all $i \ne j$.
\end{enumerate}
\end{lemma}
\begin{proof}
We only need to prove that (2) implies (1).
Let $x \in L_i$ and $y \in \bigoplus_{j \ne i} L_j$.
Then $f_1(x,y) = f_{\sqrt{d}}(x,y) = 0$ and thus $\Tr(a h(x,y)) = 0$ for all $a \in F$.
Since $F/\QQ$ is separable, it follows that $h(x,y) = 0$.
Let $r \in R$. Then $h(rx, y) = 0$ and thus $f_a(rx, y) = 0 $ for all $a \in F$.
Hence $rx \in \QQ L_i \cap L = L_i$. So $L_i$ is indeed an $R$-module.
\end{proof}

If $(V, h)$ is positive definite, then so is the rational bilinear map $f_1$ from above.
In this case, a well known result of Kneser shows that there exists a unique decomposition of $L$ as in Lemma~\ref{LatIndec} (2) into minimal $\ZZ$-submodules.
It can be computed as in \cite[Algorithm 4.5]{HeVa}.
Hence the previous lemma shows that any positive definite hermitian $R$-lattice $L$ has a unique decomposition into indecomposable sublattices and it yields a method to compute these sublattices.

For a prime $p \in \Primes$ let $R_p:= \ZZ_p \otimes_\ZZ R$ and $L_p:= R_p \otimes_R L$ be the completions of $R$ and $L$ at $p$.
Then $L_p$ is an $R_p$-lattice in $(V_p, h)$.
The introduced notion for $R$-lattices carries over to $R_p$-lattices.
For example we call an $R_2$-lattice $L$ even, if $h(x,x) \in 2 \ZZ_2$ for all $x \in L$.

%% file: sections/classificationlattices.tex
\subsection{Enumeration of positive definite unimodular hermitian lattices}\label{sec:class}

Let $R = \ZZ[\omega f]$ be the order of conductor $f$ in $F$.
In this section, we present an algorithm to enumerate all positive definite unimodular $R$-lattices of a given rank.

\begin{definition}
Let $L$ and $L'$ be full $R$-lattices in the hermitian spaces $(V,h)$ and $(V',h')$.
The lattices $L$ and $L'$ are said to be isometric, if there exists an isometry $\varphi$ from $(V,h)$ to $(V',  h')$ such that $\varphi(L) = L'$.
In this case, we write $L \cong L'$.
Further let
\[
 \cls(L) = \{ \varphi(L) \mid \varphi \in \GU(V,h) \} \quad \mbox{and} \quad
 \Aut(L) = \{ \varphi \in \GU(V,h) \mid \varphi(L) = L \}
\]
be the \emph{isometry class} and the \emph{automorphism group} of $L$.
Similarly one defines isometries between the completions $L_p$ and $L'_p$ at a prime $p$.
The \emph{genus} of $L$ is
\[ \gen(L):= \{L' \subset V \mid L' \mbox{ is an $R$-lattice such that } L_p \cong L'_p \mbox{ for all } p \in \Primes \} \:. \]
\end{definition}

When $R$ is the maximal order, the following remark shows how to find a lattice in a given genus.

\begin{remark}\label{FindGenus}
Every $\OO_p$-lattice $L_p$ admits an orthogonal decomposition 
\[ L_p = L_1 \perp \ldots \perp L_r \quad \mbox{such that } p^{s_i} L_i^{\#} = L_i \]
for some integers $s_1 < \ldots < s_r$.
Jacobowitz \cite{Jac} shows that $(s_1,\dots,s_r)$ together with the ranks and determinants of the lattices $L_i$ uniquely describe the isometry class of $L_p$ unless $p=2$ and $F_2/\QQ_2$ is ramified. 
For the remaining case, he shows that some additional invariants are needed.\\
Let $G$ be a genus of hermitian $\OO$-lattices. Suppose that for each prime $p$, we are given the $p$-adic local invariants of the lattices in $G$.
Then we can construct an $\OO$-lattice in $G$ as follows.
\begin{enumerate}
\item Since the local invariants yield the determinant of $L_p$, we can construct a hermitian space $(V,h)$ over $F$ that contains this genus using Remark~\ref{FindSpace}.
\item Fix any $\OO$-lattice $L$ in $V$.
Then the set of all primes $p$ where $L_p$ has the wrong invariants is finite.
\item If $L_p$ has the wrong invariants, let $X$ be any $\OO$-lattice in some hermitian space $(V', h')$ over $F$ such that $X_p$ has the correct invariants.
Approximate an isometry between $(V'_p, h')$ and $(V_p, h)$ by some $F$-linear map $\varphi \colon V' \to V$.
If the approximation is good enough, then $\varphi(X)_p$ has the same invariants as $X_p$.
Then there exists $a,b \in \ZZ$ such that
\[ p^a L_p \subseteq \varphi(X)_p \subseteq p^b L_p \:. \]
Now the lattice $(\varphi(X) + p^a L) \cap p^b L$ coincides with $L$ at all places different from $p$ and it has the correct invariants at $p$.
So if we iterate this step, we end up with an $\OO$-lattice in $G$.
\end{enumerate}
A different approach is suggested in \cite[Section 3.5]{Kir}.
\end{remark}

Let $L$ be an $R$-lattice in a positive definite hermitian space over $F$.
The analogue of Landherr's theorem does not hold for hermitian $R$-lattices, i.e. the genus of $L$ does not necessarily consist of a single isometry class.
However, the genus of $L$ is a disjoint union of finitely many isometry classes
\begin{equation}\label{eq:decomp}
\gen(L) = \biguplus_{i=1}^{h(L)} \cls(L_i) \:. 
\end{equation}
The number of classes $h(L)$ is called the \emph{class number} of (the genus of) $L$.
There are only very few partial results like \cite{HaKo,HaKo2} on how to deduce the class number from local invariants and these only deal with $\OO$-lattices.

Thus an important problem is to work out the class number $h(L)$ or more generally to make the decomposition in Equation \eqref{eq:decomp} explicit.
This can be done by Kneser's neighbour method. It is explained in great detail in \cite{Sch} for $\OO$-lattices.
Note that this is all we need, since we will reduce the case that $R$ is non-maximal to this special case in Algorithm~\ref{UnimodR}.

The basic idea of Kneser's method the following:
Let $\fp$ be a prime ideal of $\OO$ over $p > 2$ such that $L_p$ is unimodular.
An $\OO$-lattice $L'$ in $V$ is called a $\fp$-neighbour of $L$ if $L/(L \cap L') \cong \OO/\fp$ and $L'/(L \cap L') \cong \OO/\overline{\fp}$.
Any $\fp$-neighbour of $L$ lies in $\gen(L)$ and the $\fp$-neighbours of $L$ can be enumerated quickly.
Strong approximation yields a finite set $S$ of unramified prime ideals of $\OO$ such that given $L' \in \gen(L)$, there exists a sequence of $\OO$-lattices $L = L_0,L_1,\dots,L_r \cong L'$ such that $L_i$ is a $\fp_i$-neighbour of $L_{i-1}$ for some $\fp_i \in S$.
In fact, Shimura \cite[Theorem 5.24 and its proof 5.28]{Shi} shows how to choose such a set $S$.
Note that if $g$ is even, his description makes use of the groups $\{\det(g) \mid g \in \Aut(L_p)\}$ at primes $p$ that ramify in $F$.
These groups have recently been worked out in \cite{Det}.
So the isometry classes in $\gen(L)$ are found by repeatedly computing $\fp$-neighbours for some $\fp \in S$.

Note that this procedure can be sped up considerably by using Siegel's mass formula as a stopping condition:
Since isometric lattices have isomorphic automorphism groups, the \emph{mass} of $L$
\[ \Mass(L):= \Mass(\gen(L)) = \sum_{i=1}^{h(L)} \frac{1}{\# \Aut(L_i)} \]
is a well-defined positive rational number, which only depends on the genus of $L$.
It can be computed a priori using Siegel's mass formula, which expresses $\Mass(L)$ in terms of special values of $L$-series and local factors that depend on the genus of $L$.
The local factors have been worked out by Gan and Yu \cite{GY} for all primes $p$, except if $p=2$ ramifies in $F$.
In this exceptional case the local factors can be worked out as explained in \cite[Sections 4.3 and 4.5]{Kir}.

So if $R=\OO$ is maximal, we can construct lattices in a given genus and enumerate the isometry classes in this genus.
We will now extend these methods to enumerate the isometry classes of positive definite unimodular $R$-lattices.
Note that these lattices might lie in non-isometric hermitian spaces.

\begin{lemma}\label{LtoM}
Let $L$ be a unimodular hermitian $R$-lattice. Then $M:= \OO L$ is an integral $\OO$-lattice and
\[ fM^{\#, \OO} \subseteq L \subseteq M \:. \]
\end{lemma}
\begin{proof}
The fact that $M$ is integral and the inclusion $L \subseteq M$ are clear. 
Suppose now $z \in fM^{\#, \OO}$. Hence $h(z/f, M) \subseteq \OO$. This implies $h(z, L) \subseteq f\OO \subseteq R$.
So $z \in L^{\#, R} = L$.
\end{proof}

\begin{algorithm}[H]
\caption{Enumeration of unimodular positive definite hermitian $R$-lattices of rank $g$.}
\label{UnimodR}
\begin{algorithmic}[1]
    \Require An order $R$ of conductor $f$ in an imaginary quadratic number field $F$ and an integer $g \ge 1$.
    \Ensure A set $\mathcal{L}$ of $R$-lattices representing the isometry classes of positive definite, unimodular hermitian $R$-lattices of rank $g$.
    \State $\mathcal{L} \gets \emptyset$.
    \State Let $p_1,\dots,p_s$ be the prime divisors of $fd_F$ that do not split in $F$.
    \ForAll{subsets $I \subseteq \{p_1,\dots,p_s\}$ of even cardinality}
      \State Using Remark~\ref{FindSpace} construct some positive definite hermitian form $h \colon F^g \times F^g \to F$ such that \[ \{ p \in \Primes \mid \det(F^g, h) \notin \Nr(F_p^*) \} = I \:. \]
      \State Using Remark~\ref{FindGenus} find $\OO$-lattices $G_1, \dots, G_r$ representing the genera of all integral $\OO$-lattices $M$ in $(F^g, h)$ such that $f M^{\#, \OO} \subseteq M$.
      \For{$1 \le i \le r$}
         \State Let $M_1,\dots,M_s$ represent the isometry classes of $\OO$-lattices in $\gen(G_i)$ using Kneser's method.
	 \If{$R = \OO$}
            \State $\mathcal{L} \gets \mathcal{L} \cup \{M_1,\dots,M_s\}$.
	 \Else
         \For{$1 \le j \le s$}
            \State Let $L_1,\dots, L_t$ be orbit representatives of the action of $\Aut(M_j)$ on 
              \[ \{ L \subseteq M_j \mid L \mbox{ a unimodular $R$-lattice containing $fM_j^{\#, \OO}$ with $\OO L = M_j$}\} \:. \]
            \State $\mathcal{L} \gets \mathcal{L} \cup \{L_1,\dots, L_t\}$.
         \EndFor
	 \EndIf
      \EndFor
    \EndFor
    \State \Return $\mathcal{L}$.
\end{algorithmic}
\end{algorithm}

\begin{proposition}
Algorithm~\ref{UnimodR} which takes as input an order $R$ of conductor $f$ in an imaginary quadratic field and an integer $g \ge 1$ outputs the list of $R$-lattices representing the isometry classes of positive definite, unimodular hermitian $R$-lattices of rank $g$.
\end{proposition}
\begin{proof}
Let $L$ be a unimodular, full $R$-lattice in a positive definite hermitian space $(V, h')$ of rank $g$.
We first show that the set $\mathcal{L}$ returned by the algorithm contains a lattice isometric to $L$.
Let $p$ be a prime not dividing $f d_F$. 
Then $L_p$ is a unimodular $\OO_p$-lattice.
If $p$ splits in $F$, then $\det(V_p, h') \in \QQ_p^* = \Nr(F^*_p) $.
Suppose now $p$ is non-split.
By \cite[Proposition 4.4]{Jac} $L_p$ admits an orthogonal basis. Hence $\det(V_p, h')$ has a representative in $\ZZ_p^* \subseteq \Nr(F^*_p)$.
So Landherr's theorem implies that $(V, h')$ is isometric to one of the spaces $(F^g, h)$ the algorithm considers.
After replacing $L$ by an isometric copy, we may therefore assume that $M:= \OO L$ is one of the lattices $M_j$ in line 7.
Proposition \ref{LtoM} shows $fM_j^{\#, \OO} \subseteq L \subseteq M_j$.
Thus $\mathcal{L}$ contains an $R$-lattice isometric to $L$.\\
Next we show that $\mathcal{L}$ does not represent any isometry class twice.
Suppose $L_1, L_2 \in \mathcal{L}$ are isometric.
This isometry extends to an isometry between $\OO L_1$ and $\OO L_2$.
By construction, this implies $\OO L_1 = \OO L_2$.
Hence $L_1$ and $L_2$ are in the same orbit under $\Aut(\OO L_1)$.
This shows $L_1 = L_2$.
\end{proof}

If we restrict ourselfs to projective unimodular $R$-lattices, we can speed up Algorithm~\ref{UnimodR} considerably.
To this end, let $L$ be a full, projective $R$-lattice in a positive definite hermitian space $(V, h)$ over $F$ and set $M = \OO L$.
The $R$-lattice $L$ has a pseudo-basis  
\[ L = \bigoplus_{i=1}^g \fa_i x_i \]
with invertible fractional ideals $\fa_1,\dots,\fa_g$ of $R$ since $L$ is a projective $R$-module. Let $(x_1^\#,\dots,x_g^\#)$ denote the dual basis of $(x_1,\dots,x_g)$. Then
\[ M = \bigoplus_{i=1}^g \OO \fa_i x_i, \quad L^{\#, R} = \bigoplus_{i=1}^g \overline{(R: \fa_i)} x_i^\#  \quad \mbox{and} \quad M^{\#, \OO} =   \bigoplus_{i=1}^g \overline{(\OO: \OO \fa_i)} x_i^\# = \OO L^{\#, R}\:. \]
Since $(R:\fa_i)$ is an invertible $R$-ideal, we see that $ L^{\#, R}$ is projective as well.

\begin{proposition}\label{prop:descent}
Let $L$ be a full, projective $R$-lattice in a hermitian space $(V, h)$ and let $M = \OO L$.
Let $\Phi$ be the bilinear map defined by
\begin{equation}\label{eq:Phi}
 \Phi \colon M/fM \times M/fM \to \OO/R \cong \ZZ/f\ZZ, \; (x,y) \mapsto h(x,y) + R \:.
\end{equation}
Then the following hold.
\begin{enumerate}
\item
If $L$ is a unimodular $R$-lattice, then $M$ is a unimodular $\OO$-lattice.
\item If $M$ is a unimodular $\OO$-lattice, then the following are equivalent:
\begin{enumerate}
\item
$L$ is a unimodular $R$-lattice.
\item
$L$ is an integral $R$-lattice.
\item
$L/fM$ is an isotropic subspace of $(M/fM, \Phi)$, i.e. $\Phi(x,y) = 0$ for all $x,y \in L/fM$.
\end{enumerate}
\end{enumerate}
\end{proposition}
\begin{proof} (1) The discussion before the proposition shows that $L = L^{\#, R}$ implies $M = M^{\#, \OO}$.
(2b) $\implies$ (2a):
We have $L \subseteq L^{\#,R}$ by assumption. Equality follows from the fact that the projective $R$-modules $L$ and $L^{\#, R}$ both have index $f^g$ in $M = M^{\#, \OO} = \OO L^{\#, R}$.
The implications (2a) $\implies$ (2b) $\iff$ (2c) are clear.
\end{proof}

\begin{algorithm}[H]
\begin{algorithmic}[1]
\caption{Enumeration of projective unimodular $R$-lattices of rank $g$.} \label{algo:proj}
\Require{An integer $g \ge 2$ and an order $R$ in $F$.}
\Ensure{A set of representatives of the isometry classes of projective, positive definite, unimodular hermitian $R$-lattices of rank $g$.}
\State Fix a chain of minimal overorders $R = \OO^{(0)} \subsetneq \OO^{(1)} \subsetneq \ldots \subsetneq \OO^{(r)} = \OO$.
\State Using Algorithm~\ref{UnimodR} compute a set $\mathcal{S}$ of representatives of isometry classes of unimodular hermitian $\OO$-lattices of rank $g$.
\For{$i=r, \dots, 1$}
\State Let $p$ be the index of $\OO^{(i-1)}$ in $\OO^{(i)}$.
\State $\mathcal{T} \gets \emptyset$.
\For{$M \in \mathcal{S}$}
\State Let $\mathcal{V}$ represent the orbits of all $g$-dimensional isotropic subspaces of $(M/pM, \Phi)$ under the action of $\Aut(M)$ where $\Phi$ is chosen as in Equation \eqref{eq:Phi}.
\For{$V \in \mathcal{V}$}
\State Let $L$ be the full preimage of $V$ under the canonical epimorphism $M \to M/pM$.
\State If $L$ is an integral $\OO^{(i-1)}$-lattice with $\OO^{(i)}L = M$ then include $L$ to the set $\mathcal{T}$.
\EndFor
\EndFor
\State $\mathcal{S} \gets \mathcal{T}$.
\EndFor
\State \Return $\mathcal{S}$.
\end{algorithmic}
\end{algorithm}
\begin{proposition}
Algorithm~\ref{algo:proj} which takes as input an order $R$ in an imaginary quadratic field and an integer $g \ge 2$ outputs the list of $R$-lattices representing the isometry classes of positive definite, unimodular, projective hermitian $R$-lattices of rank $g$.
\end{proposition}
\begin{proof}
After line 2, $\mathcal{S}$ is a set of representatives of the isometry classes of projective, unimodular hermitian $\OO^{(r)}$-lattices.
Let $L$ be a projective unimodular hermitian $\OO^{(r-1)}$-lattice.
Then $M:= \OO^{(r)} L$ is a projective unimodular hermitian $\OO^{(r)}$-lattice. So without loss of generality $M \in \mathcal{S}$.
Thus Proposition \ref{prop:descent} shows that the set $\mathcal{T}$ in line 13 contains an $\OO^{(r-1)}$-lattice isometric to $L$.
Suppose it contains two such lattices $L_1$ and $L_2$. Then there is an isometry $\sigma \colon L_1 \to L_2$ which induces an isometry $\OO^{(r)}L_1 \to \OO^{(r)}L_2$.
But then $\OO^{(r)}L_1 = M = \OO^{(r)}L_2$ and $\sigma \in \Aut(M)$. Hence $L_1$ and $L_2$ are in the same $\Aut(M)$-orbit.
This shows that $L_1 = L_2$.
Hence after line 13, $\mathcal{S}$ is a set of representatives of the isometry classes of projective, unimodular hermitian $\OO^{(r-1)}$-lattices.
By induction it follows that after $r$ iterations, $\mathcal{S}$ represents the isometry classes of projective, unimodular hermitian $R$-lattices.
\end{proof}

Note that Algorithm \ref{algo:proj} calls Algorithm~\ref{UnimodR}.
But if $R = \OO$ is maximal, the expensive steps 11--14 of Algorithm~\ref{UnimodR} are skipped.
They are replaced by a much more refined descent in lines 3--13 of Algorithm~\Ref{algo:proj}, which is based on Proposition~\ref{prop:descent}.

Also note that in Algorithm \ref{algo:proj} it would be possible to go from $\OO$-lattices to $R$-lattices directly.
But then it would be much more difficult to find the desired (projective) $R$-lattices between $fM$ and $M$.

%% file: sections/orthogonal.tex
\subsection{Orthogonal families inside a lattice} \label{sec:ortho}

Let $(V,h)$ be a positive definite hermitian space over $F$ of rank $g$.
Let $R$ be the order in $F$ of conductor $f$.

In this section, we give necessary and sufficient conditions for a unimodular hermitian $R$-lattice to contain a free $R$-sublattice isometric to $\langle \ell, \dots, \ell \rangle$ for some $\ell \in \NN$, which we may require to be odd.
Such a free sublattice will be needed in Section~\ref{sec:effectivekernel}, where we provide an algorithm for finding a good isogeny from our target principally polarized abelian variety to a totally decomposable one.
But we also think that this problem arises naturally and should deserve more investigations around the smallest values of $\ell$ that can be obtained.

We will prove the following result.

\begin{theorem}\label{ellid}
Let $L$ be a full $R$-lattice in $(V, h)$.
Then the following hold:
\begin{enumerate}
\item There exists an orthogonal basis $(b_1,\dots,b_g) \in L^g$ of $V$.
\item \label{ellid:2}
There exists an integer $\ell$ and a free $R$-sublattice $L'$ of $L$ such that $L' \cong \langle \ell,\dots,\ell \rangle$ if and only if $g$ is odd or $\det(V,h) \in \Nr(F^*)$.
\item \label{item3}
Let $L$ be unimodular and let $a \in \ZZ \setminus \{0\}$.
Suppose $g$ is odd or $\det(V,h) \in \Nr(F^*)$.
There exists some positive integer $\ell$ coprime to $a$ and a free $R$-sublattice $L'$ of $L$ such that $L' \cong \langle \ell,\dots,\ell \rangle$ if and only if the following conditions hold.
\begin{enumerate}
\item For all primes $p \mid a$ the module $L_p$ is free over $R_p$.
\item If $a$ is even then there exists some $\ell_2 \in \ZZ_2^*$ such that $L_2 \cong \langle \ell_2,\dots,\ell_2\rangle$.
\item If $g$ is even, then $\det(L_p, h) \in \Nr(R_p^*)$ for all odd primes $p$ such that $p \mid \gcd(a, f)$.
\end{enumerate}
\end{enumerate}
\end{theorem}

Remark~\ref{ellidcheck} below shows how to check the conditions (a)--(c) in part 3 of Theorem~\ref{ellid}.
Let $L$ be an \mbox{$R$-lattice} $L$ in $V$.
Then we can find an orthogonal basis of $V$ in $L$ as follows.
For any positive rational number $\ell$ the map
\[ q_\ell \colon V \to \QQ,\; v \mapsto \Tr(h(v,v)/\ell) \]
is a positive definite quadratic form on the $\QQ$-space $V$ and
\[ \{ v \in L \mid h(v,v) = \ell \} \subseteq \{ v \in L \mid q_\ell(v) = 2 \} \:. \]
Note that the right hand side is finite and it can be enumerated using the Fincke-Pohst algorithm~\cite{FP}.
This allows us to compute the set of vectors in $(L, h)$ of norm $\ell$.

It is now clear how to find an orthogonal basis as in Theorem~\ref{ellid}.
For part (1), we use the usual Gram-Schmidt process.
For parts (2) and (3), we apply Algorithm~\ref{algo:ortho2} to $\ell=1,2,3,\dots $ until we find a suitable basis.
As all our algorithms, its complexity is at least exponential in the rank $g$.
We could not find in the literature any result about a possible upper bound on $\ell$ when it exists.

\begin{algorithm}[H]
\begin{algorithmic}[1]
\Require A full $R$-lattice $L$ in $V$ and a rational number $\ell > 0$.
\Ensure An orthogonal basis of $V$ consisting of vectors in $L$ of norm $\ell$ if possible; otherwise $\emptyset$.
\Function{BackTrack}{$F$, $S$}
\IIf{$\#F = g$}{~\Return $F$} \EndIIf
\IIf{$\#F + \dim \langle S\rangle < g$}{~\Return $\emptyset$} \EndIIf
\State Pick some $v \in S$.
\If{$h(v, f) = 0$ for all $f \in F$}
\State $T \gets $ \Call{BackTrack}{$F \cup \{v\}$,  $\{w \in S \mid h(v,w) = 0 \})$}.
\IIf{$T \ne  \emptyset$}{~\Return $T$} \EndIIf
\EndIf
\State \Return \Call{BackTrack}{$F$, $S \setminus \{v\}$}.
\EndFunction
\IIf{$\ell^g \cdot \det(V, h) \notin \Nr(F^*)$}{~\Return $\emptyset$}\EndIIf
\State $S \gets \{ v \in L \mid h(v,v) = \ell \}$.
\State \Return \Call{BackTrack}{$\emptyset$, $S$}.
\end{algorithmic}
\caption{Computation of an orthogonal family of $g$ vectors of norm $\ell$} \label{algo:ortho2}
\end{algorithm}

The remainder of this section gives a proof of Theorem~\ref{ellid}.
We start by giving a classification of all free unimodular hermitian $R_p$-lattices which admit an orthogonal basis.
If $R_p$ is maximal, this follows from Jacobowitz classification of local hermitian lattices \cite{Jac}.

\begin{proposition}\label{LatOGBasis}
Let $L$ be a free, unimodular hermitian $R_p$-lattice of rank $g$. Then
\[ L = L_1 \perp \ldots \perp L_r \]
for some free unimodular hermitian $R_p$-sublattices $L_i$ of rank at most $2$.
If one of $p, g$ or $L$ is odd, then all $L_i$ can be chosen to have rank $1$.
\end{proposition}
\begin{proof}
Let $(b_1,\dots,b_g)$ be a basis of $L$.
Suppose first that $h(b_i, b_i) \in \ZZ_p^*$ for some $i$.
Then $L = R_p b_i \perp \sum_{j\ne i} R_p (b_j - \frac{h(b_j, b_i)}{h(b_i, b_i)} b_i)$.
Suppose now that such an index $i$ does not exist.
Since $L$ is free and unimodular, there exist $1 \le i < j \le g$ such that $h(b_i, b_j) \in R_p^*$.
If $p \ne 2$, we can replace $b_i$ with $b_i' := b_i + 1/(2h(b_j,b_i)) b_j$. Then $h(b_i', b_i') \in \ZZ_p^*$ and we obtain a splitting $L = R b_i' \perp L'$ as before.
If $p = 2$, we may assume that $h(b_i, b_j)=1$. Then $L = (R_p b_i \oplus R_p b_j) \perp L' $ where
\[ L' = \bigoplus_{k \ne i,j} R_p (b_k - \frac{h(b_j, b_j)h(b_k, b_i) - h(b_k, b_j)}{h(b_i,b_i)h(b_j,b_j)-1} b_i -  \frac{h(b_i, b_i)h(b_k, b_j) - h(b_k, b_i)}{h(b_i,b_i)h(b_j,b_j)-1} b_j) \:. \]
So in any case, we obtain a decomposition $L = L_1 \perp L'$ with free, unimodular lattices $L_1$ and $L'$ such that the rank of $L_1$ is at most $2$.
The first assertion now follows by induction on the rank $g$ and we have also seen that we can choose all $L_i$ of rank $1$ when $p$ is odd.\\
Suppose now $p = 2$ and also suppose that $g$ or $L$ is odd.
If $L$ is odd, we can choose the vector $b_1$ in our original basis such that $h(b_1,b_1) \in \ZZ_2^*$.
If $g$ is odd, then one of the $L_i$ must have rank $1$.
So in both cases, there exists a summand $ L_i = R_2 x_1$ of rank $1$.
Suppose $L_j = R_2 x_2 \oplus R_2 x_3$ is binary. If $h(x_2,x_2) \in \ZZ_2^*$ or $h(x_3,x_3) \in \ZZ_2^*$, we can split $L_j$ just as before.
So suppose $h(x_2,x_2), h(x_3,x_3) \in 2\ZZ_2$. Let $x_2':= x_2 + x_1$.
Then as before $L_i \oplus L_j = (R_2 x_2' \oplus R_2 x_3) \perp R_2 x_1'$ for some $x'_1 \in L_i \oplus L_j$.
But now  $h(x_2', x_2') \in \ZZ_2^*$ and thus $L_i \oplus L_j$ has an orthogonal basis.
Iterating this argument shows that $L$ has an orthogonal basis.
\end{proof}

\begin{remark} Let $L$ be a free unimodular hermitian $R_p$-lattice.
\begin{enumerate}
\item If $p=2$ and the rank of $L$ is odd, then $L$ is odd.
\item $L$ has an orthogonal basis if and only if $p>2$ or $L$ is odd.
\end{enumerate}
\end{remark}

The classification of all free unimodular hermitian $R_p$-lattices which have an orthogonal basis more or less boils down to a description of the norm group $\Nr(R_p^*)$.
To this end, let
\[ \ZS_p = \{ u^2 \mid u \in \ZZ_p^*\} = \{\Nr(u) \mid u \in \ZZ_p^* \} \]
be group of squares in $\ZZ_p^*$.

\begin{lemma}\label{NormGroup}
If $p$ is odd, then
\begin{align*}
\Nr(R_p^*) &=
\begin{cases}
\ZZ_p^* & \text{if } p \nmid f d_F, \\
\ZS_p & \text{if } p \mid f d_F
\end{cases}
\intertext{and}
\Nr(R_2^*) &=
\begin{cases}
\ZZ_2^* & \text{if $2 \nmid d_F$ and $4 \nmid f$},\\
\ZS_2 \uplus (1-\tfrac{d_F}{4}) \ZS_2 & \text{if $8 \mid d_F$ and $2 \nmid f$},\\
\ZS_2 & \text{if } 2^5 \mid f^2 d_F,\\
\ZS_2 \uplus 5 \ZS_2 & \text{otherwise}.
\end{cases}
\end{align*}
\end{lemma}
\begin{proof}
We have $\ZS_p \subseteq \Nr(R_p^*) \subseteq \ZZ_p^*$ and the structure of $ \ZZ_p^* / \ZS_p$ is well known.
In particular, the square classes can be distinguished modulo $4p$.
Any unit $u \in R_p = \ZZ_p[f \omega]$ is of the form $u = x + yf\omega$ with $x,y \in \ZZ_p$ and
\[ \Nr(u) = (x+yf\omega) \overline{(x+yf\omega)} = x^2 + xyfd_F + y^2f^2 \frac{d_F^2 - d_F}{4} \in \ZZ_p^* \:. \]
The result now follows by a case by case discussion of the possible $p$-adic valuations of $f$ and $d_F$.
\end{proof}

\begin{corollary}\label{FreeRpLatticesOdd}
Let $L$ be a free unimodular hermitian $R_p$-lattice of rank $g$. Let $u \in \ZZ_p^*$ be a representative of $\det(L) \in \ZZ_p^* / \Nr(R_p^*)$.
If $p > 2$, then $L \cong \langle 1,\dots, 1, u \rangle$.
\end{corollary}
\begin{proof}
Let $\varepsilon \in \ZZ_p^*$ be a non-square.
Proposition \ref{LatOGBasis} shows that $L \cong \langle u_1,\dots,u_g \rangle$ with $u_i \in \{1,\varepsilon\}$.
It is well known that there exists some $U \in \GL_2(\ZZ_p)$ such that $\trsp{U} \diag(1,1) U = \diag(\varepsilon, \varepsilon)$.
Hence $\langle 1, 1 \rangle \cong \langle \varepsilon , \varepsilon \rangle$ and thus we can assume that $u_1= \ldots = u_{g-1} = 1$.
\end{proof}

\begin{proposition}\label{FreeRpLattices}
Let $L$ be a free, odd, unimodular hermitian $R_2$-lattice of rank $g\ge 2$.
Let $u \in \ZZ_2^*$ be a representative of $\det(L) \in \ZZ_2^* / \Nr(R_2^*)$.
\begin{enumerate}
\item If $R_2$ is maximal or $3 \in \Nr(R_2^*)$ or $7 \in \Nr(R_2^*)$, then $L \cong \langle 1,\dots, 1, u \rangle$.
\item If $g>2$ and the conditions in (1) are not satisfied then either
\[ L \cong \langle 1,\dots,1, u\rangle \quad  \mbox{or} \quad L \cong \langle 1,\dots,1,3,3,u \rangle \]
but not both.
\item If $g=2$ and the conditions in (1) are not satisfied then either $L \cong \langle 1, u \rangle$ or $u\equiv1,5 \pmod{\Nr(R_2^*)}$ and $L \cong \langle 3,3u \rangle$.
\end{enumerate}
\end{proposition}
\begin{proof}
If $R_2$ is maximal, the result follows from \cite[Theorem 7.1 and Proposition 10.4]{Jac}.
Suppose now $R_2$ is not maximal.
Proposition \ref{LatOGBasis} shows that $L \cong \langle u_1,\dots,u_g \rangle$ with $u_i \in \ZZ_2^*$.
If $3 \in \Nr(R_2^*)$ or $7 \in \Nr(R_2^*)$ we may assume that $u_i \in \{1,5\}$ for all $i$.
As in the proof of Corollary \ref{FreeRpLatticesOdd} we conclude that $u_1 = \dots = u_{g-1} = 1$. The first assertion follows.\\
Suppose now $3,7 \notin \Nr(R_2^*)$ and $g \ge 3$.
\cite[Theorem 93:16]{OMe} yields some $T \in \GL_g(\ZZ_2)$ and $e \in \{1,3\}$ such that
\[ \trsp{T} \diag(u_1,\dots,u_g) T =  \diag(1,\dots, 1, e, e, \prod_i {u_i}) \:. \]
Hence $L \cong \langle 1,\dots,1,e,e,u \rangle$.
It remains to show that $M:= \langle 1,\dots,1,1,1,u \rangle$ and $N:= \langle 1,\dots,1,3,3,u \rangle$ are not isometric.
Let $V$ be the ambient hermitian space of $M$ and $N$.
Let $X$ and $Y$ be the $\ZZ_2$-lattices $M$ and $N$ equipped with the bilinear form $V \times V \to \QQ_2, (x,y) \mapsto \Tr(h(x,y)/2)$.
Lemma \ref{NormGroup} shows that $R_2 = \ZZ_2 \oplus \alpha \ZZ_2$ for some $\alpha \in R_2$ with $\Tr(\alpha) = 0$  and $n:= \Nr(\alpha) \in 4 \ZZ_2$.
Hence $X = X_0 \perp X_1$ where $X_0$ and $X_1$ are free with Gram matrices $\diag(1,\dots,1, u)$ and $\diag(n,\dots,n,un)$.
Similarly $Y = Y_0 \perp Y_1$ where $Y_0$ and $Y_1$ are free with Gram matrices $\diag(1,\dots,1, 3,3,u)$ and $\diag(n,\dots,n,3n,3n,un)$.
Suppose $M$ and $N$ are isometric hermitian $R_2$-lattices. Then $X$ and $Y$ are isometric bilinear $\ZZ_2$-lattices.
By \cite[Theorem 93:29 (ii)]{OMe}, this implies that $X_0$ is isometric to $Y_0$, which is impossible since the two ambient quadratic spaces have different Hasse-Witt invariants.
The case $g=2$ follows along the same lines.
\end{proof}

The above proof shows that the possible cases in part (2) and (3) of Proposition~\ref{FreeRpLattices} can be distinguished as follows.

\begin{remark}\label{FreeRpLattices2}
Let $L \cong \langle u_1,\dots,u_g \rangle$ where $u_i \in \ZZ_2^*$ and $g \ge 2$. Write $u = \prod_i u_i$.
Suppose that $R_2$ is not maximal and that $3,7 \notin \Nr(R_2^*)$.
Then $L \cong \langle 1,\dots,1, u \rangle $ if and only if $\prod_{i<j} (u_i, u_j)_2 = 1$ where $(\_,\_)_2$ denotes the Hilbert-Symbol of $\QQ_2$.
\end{remark}

We are now ready to prove the main result of this section.

\begin{proof}[Proof of Theorem~\ref{ellid}.]
The first assertion is the Gram-Schmidt process. 
For the remainder let $\mu \in \NN$ be a representative of $\det(V,h) \in \QQ^*/\Nr(F^*)$.
Let $(V', h')$ be a hermitian space over $F$ with Gram matrix $\mu \cdot I_g$.
If $g$ is odd or $\det(V,h) \in \Nr(F^*)$, then $(V,h)$ and $(V', h')$ have the same rank, the same determinant and the same signature.
Hence they are isometric by Landherr's Theorem. 
Thus $(V,h)$ contains a free $R$-lattice $M \cong \langle \mu,\dots,\mu \rangle$.
Let $m \in \NN$ such that $L' := m M \subseteq L$. Then $L' \cong \langle \ell,\dots, \ell \rangle$ where $ \ell =m^2 \mu$.
Conversely, if such a lattice $L'$ exists and $g$ is even, then $\det(V,h) = \ell^g \in \Nr(F^*)$.
This proves the second assertion.\\
Suppose now $L$ has a sublattice $L'$ as in (3).
For any prime divisor $p$ of $a$, we have
\[ L'_p \subseteq L_p \subseteq L_p^\# \subseteq (L'_p)^\# = L'_p \:. \]
Hence $L_p = L'_p \cong \langle \ell,\dots,\ell \rangle$ and if $g$ is even, then $\det(L_p, h) = \ell^g \in \Nr(R_p^*)$.
Finally suppose that the three conditions of part $(3)$ hold.
If $g$ is even and $a$ is odd, set $r = 1$.
If $g$ and $a$ are both even let $r \in \NN$ such that $r/\ell_2 \in \Nr(R_2^*)$.
If $g$ is odd, we also choose some integer $r$, but much more carefully.
For all $p \mid a$ the assumption that $L_p$ is free and unimodular implies $\det(L_p, h) \in \ZZ_p^*$. 
Hence we may assume that the representative $\mu \in \NN$ of $\det(V,h)$ from above is coprime to $a$.
Dirichlet's theorem on primes in arithmetic progressions yields some prime $r$ such that
\begin{align*}
r &\equiv \ell_2 \pmod{\Nr(R_2^*)} \text{ if } 2 \mid a, \\
r &\equiv \mu \pmod{\Nr(R_p^*)} \text{ for all } 2 \ne p \mid a, \\
r &\equiv \mu \pmod{\Nr(F_p^*)} \text{ for all } p \mid \mu d_F \text{ and } p \nmid a.
\end{align*}
Notice that if $2 \mid a$, then $\ell_2 \equiv \ell_2^g \equiv \mu \pmod{\Nr(F_2^*)}$ and for $p \nmid r a\mu d_F$ we have $r/\mu \in \ZZ_p^* \subseteq \Nr(F_p^*)$.
Hence $r/\mu \in \Nr(F_p^*)$ for all primes $p \ne r$.
The product formula for norm symbols and Hasse's norm theorem imply that $r/\mu \in \Nr(F^*)$.\\
So whether $g$ is even or odd, we have $r^g / \mu \in \Nr(F^*)$.
As in part (2) it follows that $(V,h)$ has a Gram matrix $r \cdot I_g$.
Thus $(V,h)$ contains a full $R$-lattice $M \cong \langle r,\dots, r \rangle$.
Corollary \ref{FreeRpLatticesOdd}, condition $(3c)$ and the choice of $r$ show that for $p \mid a$ there exists some local isometry $\sigma_p \colon M_p \to L_p$.
Since $M$ has an orthogonal basis, we may assume that $\det(\sigma_p) = 1$.
Strong approximation yields some  $\sigma \in \SU(V,h)$ such that $\sigma(M)_p = L_p$ for all $p \mid a$, cf. \cite{Approx}.
Hence there exists an integer $b$ coprime to $a$ such that $b\sigma(M) \subseteq L$.
Then $L':= b \sigma(M) \cong \langle \ell,\dots,\ell \rangle$ with $\ell = b^2r$.
This proves the third assertion.
\end{proof}

\begin{remark}\label{ellidcheck}
Let $L$ be a unimodular $R$-lattice in $(V, h)$ given by a pseudo basis $L = \bigoplus_{i=1}^g \fa_i x_i$.
Then the conditions in part (3) of Theorem~\ref{ellid} can be checked as follows.
\begin{enumerate}
\item
The $R_p$-module $L_p$ is free if and only $\fa_i R_p$ is principal for all $i$.
Since $R$ is Gorenstein, the latter condition holds if and only if the conductor of $R$ and the conductors of the multiplicator rings of all $\fa_i$ have the same $p$-adic valuation.
In particular, this holds if $R_p$ is maximal.
\item Let $p>2$ be a prime such that $p \mid \gcd(a,f)$ and suppose $L_p$ is free. For $1 \le i \le g $ pick some $a_i \in \fa_i$ such that $a_i R_p = \fa_i R_p$.
Then $L_p = \bigoplus_i R_p b_i$ with $b_i = a_i x_i$ and thus $\det(L_p, h) = \det(\Gram(b))$.
This can be used to check the condition (3c) as the norm group $\Nr(R_p^*)$ has been worked out in Lemma~\ref{NormGroup}.
\item Suppose $2 \mid a$, $L_2$ is free and $g$ is odd.
The existence of $\ell_2$ is guaranteed whenever $R_2$ is maximal or $3 \in \Nr(R_2^*)$ or $7 \in \Nr(R_2^*)$ since in these cases all free unimodular $R_2$-lattices in $(V_p,h)$ of determinant $\det(L_p, h)$ are isometric, cf. Proposition~\ref{FreeRpLattices}.
So suppose we are not in this case. 
Since the square classes of $\ZZ_2^*$ are represented by $\{1,3,5,7\}$, there are at most 4 possibilities for $\ell_2$.
As before we obtain an $R_2$-basis of $L_2$.
The proof of Proposition~\ref{LatOGBasis} yields an orthogonal basis of $L_2$ and thus $u_1,\dots,u_g \in \{1,3,5,7\}$ such that $L_2 \cong \langle u_1, \dots, u_g\rangle$.
By Remark~\ref{FreeRpLattices2} we have $L_2 \cong \langle \ell_2,\dots,\ell_2 \rangle$ if and only if $\ell_2 \equiv \prod_i u_i \pmod{\Nr(R_2^*)}$ and $\prod_{i<j} (u_i, u_j)_2 = (\ell_2, \ell_2)_2^{(g-1)/2}$.
This gives an effective method to find the element $\ell_2$ or to show that it does not exist.
\item Suppose $2 \mid a$, $L_2$ is free and $g$ is even.
If $2 \nmid f d_F$ then $L_2 \cong \langle 1,\dots, 1 \rangle$ by \cite[Proposition~10.4]{Jac}.
So we may assume that $2 \mid f d_F$ and we compute a Gram matrix $G$ of $L_2$.
The existence of $\ell_2$ implies that $\det(G) \in \Nr(R_2^*)$ and $L_2$ is odd.
The first condition is readily checked and the second holds if and only if some diagonal entry of $G$ lies in $\ZZ_2^*$.
Suppose these conditions both hold.
As in in the case of odd ranks, the existence of $\ell_2$ is now guaranteed whenever $R_2$ is maximal or $3 \in \Nr(R_2^*)$ or $7 \in \Nr(R_2^*)$.
In the other cases, the proof of Proposition~\ref{LatOGBasis} shows how to compute $u_1,\dots,u_g \in \{1,3,5,7\}$ such that $L_2 \cong \langle u_1, \dots, u_g\rangle$.
Then $L_2 \cong \langle \ell_2, \dots, \ell_2 \rangle$ if and only if $\prod_{i<j} (u_i, u_j)_2 = (\ell_2, \ell_2)_2^{g/2}$. This again yields an effective method to decide if $\ell_2 \in \{1,3,5,7\}$ exists.
\end{enumerate}
\end{remark}

\begin{example} \label{ex:noell}
Let $F = \QQ(\sqrt{-10})$ and let $\fp$ be the (non-principal) prime ideal of $\OO$ over $2$.
Equip $F^2$ with the hermitian form $h$ induced by $\diag(1,2)$.
Then 
\[ L:= \fp \cdot (2,0) \oplus \frac{1}{4} \OO \cdot (\sqrt{-10}+2, 1) \]
is a unimodular (and projective) $\OO$-lattice in $(F^2, h)$ but $\det(F^2, h) = 2$ is not a norm in $F$.
\end{example}

\begin{example}
Let $R = \ZZ[2i]$ be the order of conductor $2$ in $\QQ(i)$. Let $L$ be the free hermitian $R$-lattice with Gram matrix
\[G =
\begin{pmatrix}
      3  &    2i   & 2i - 1 \\
    -2i  &       3 &   2i + 1 \\
 -2i - 1 & -2i + 1 &        3
\end{pmatrix} \in R^{3 \times 3} \:.
\]
The determinant of $G$ is $1$, so $L$ is unimodular.
We find that $L_2 \cong \langle 1,3,3 \rangle$ and $\Nr(R^*) = \ZS_2 \uplus 5 \ZS_2$.
Now $(1,3)_2^2 \cdot (3,3)_2 = -1$ but $(1, 1)_2^{3} = (5,5)_2^3 = +1$.
Hence $L_2 \not\cong \langle \ell_2,\ell_2,\ell_2 \rangle$ for any $\ell_2 \in \ZZ_2^*$.
In particular, $L$ does not contain a free $R$-sublattice $L' \cong \langle \ell, \ell, \ell \rangle$ for any odd integer $\ell$.
\end{example}

%% file: sections/equivalence.tex
\section{The description of polarized abelian varieties in terms of lattices}
We set up the essential tools to introduce the equivalence of categories which allows us to interpret certain polarized abelian varieties as hermitian lattices.
\subsection{The equivalence of categories}  Let $\mathcal{C}$ be an abelian category, let $E$ be an object of $\mathcal{C}$ and let $R$ be a ring. Fix a morphism $\rho\colon R\rightarrow \End(E)$. Let $L$ be a finitely presented left $R$-module and let
$$R^m\xrightarrow{\varphi} R^n \rightarrow L\rightarrow 0$$ be a finite presentation. We identity  the map $\varphi\in M_{n,m}(R)$ with its image in $M_{n,m}(\End(E))$ by the map induced by $\rho$, where $M_{n,m}(R)$ denotes the ring of matrices with $n$ rows and $m$ columns with coefficients in $R$. It defines a morphism $$E^n\xrightarrow{\trsp \varphi}E^m.$$
The object $\ker(\trsp \varphi)$ does not depend on the presentation of $L$ and \cite[III.Sec.8.1]{Ser} uses this to define the functor $\F$ as $\F(L)=\ker(\trsp \varphi)$ on objects.  Let us look now on what $\F$ does on arrows. Let $f:L_1\rightarrow L_2$ be a morphism of $R$-modules. Given finite presentations $R^{m_i}\xrightarrow{\varphi_i} R^{n_i} \rightarrow L_i\rightarrow 0$ of $L_i$ we can lift $f$ to a commutative diagram of $R$-modules as follows.
\[
\xymatrix{R^{m_1} \ar[r]^{\varphi_1} \ar[d]^{G} & R^{n_1} \ar[d]^{F}\ar[r] & L_1 \ar[d]^f \ar[r] & 0  \\
R^{m_2} \ar[r]^{\varphi_2} & R^{n_2} \ar[r] & L_2 \ar[r] & 0 .}
\]
We can define $\F(f)$ as the map induced by $\trsp F$ by restriction to $\ker(\trsp \varphi_2)\rightarrow\ker(\trsp\varphi_1)$ 
\[
\xymatrix{E^{m_2}  \ar[d]^{\trsp G} & E^{n_2} \ar[l]^{\trsp \varphi_2}\ar[d]^{\trsp F} & \ker(\trsp \varphi_2) \ar[d]^{\F(f)} \ar[l]  & 0 \ar[l] \\
E^{m_1}  & E^{n_1}\ar[l]^{\trsp \varphi_1}& \ker(\trsp \varphi_1) \ar[l] & 0. \ar[l]}
\]
We now focus on the case where $\mathcal{C}$ is the category of group schemes over $\Fq$ (with $\FF_q$-morphisms), $E/\FF_q$ is an ordinary elliptic curve and $R=\End(E)$. The ring $R$ is an order in an imaginary quadratic field $F=\Frac R$. Denote by $\pi\in R$ the Frobenius endomorphism of $E.$
Let $\RLat$ be the category of finitely presented torsion-free left $R$-modules (this is the category of $R$-lattices from Section~\ref{sec:hermitianlattices}) and $\Ab_E$ be the sub-category of $\mathcal{C}$ of abelian varieties $\FF_q$-isogenous to a power of $E.$
\begin{theorem}
Let $E$ be an ordinary elliptic curve over $\Fq$. Then $\F$ defines an equivalence of categories between $(\RLat)^{\text{opp}}$, the opposite category of $\RLat$, and $\Ab_E$ if, and only if, $R=\ZZ[\pi]$. Moreover the functor $\F$ is exact.
\end{theorem}
The reader can refer to \cite[Theorem 7.6]{JKP+18} and \cite[Theorem 4.4]{JKP+18} for proofs.
\begin{remark}
Serre also introduces another functor $M\mapsto M\otimes E\colon = \Coker \varphi$ which is further studied in \cite[Appendice]{Lau}, \cite[section 8]{JKP+18} or \cite{amir}. This functor is covariant but not exact. We also prefer to use $\F$ since the theory is settled for an arbitrary order $R$ whereas Serre only develops it for the maximal order. In general there is no easy way to compare the two functors if the $R$-module is not projective. Notice that the image of a projective $R$-module $L \in (\RLat)^{\text{opp}}$ by $\F$ is an abelian variety $A$ isomorphic to a product of elliptic curves $E_i \sim E$ such that $\End (E_i)=R$ for all $1\leq i\leq g$. Indeed, for an $R$-ideal $I_i$ such that $E_i =\F(I_i)$, $\End(E_i)\simeq (I_i\colon I_i)$ and since $R$ is Gorenstein, $I_i$  is invertible if and only if $(I_i\colon I_i)=R$  \cite[Prop.2.1]{Mar}. 
\end{remark}

Notice that if $E$ is such that $R=\End(E) \supset \ZZ[\pi]$ then the image of $\F$ consists of the abelian varieties isomorphic to  products of elliptic curves $E_i$ such that the conductor of $\End(E_i)$ divides the conductor of $\End(E)$, as subrings of the maximal order of $F=\Frac(R)$ (see \cite[Theorem 7.5]{JKP+18}). However, if $R\neq \End(E)$, it may occur that $\F(L)$ is not even an abelian variety (see \cite[Remark 4.6]{JKP+18}). 

Notice that, given an ordinary elliptic curve $E/\Fq$ with Frobenius endomorphism $\pi$, for each order $R$ containing $\ZZ[\pi]$, there exists an elliptic curve over $\FF_q$, isogenous to $E$, with endomorphism ring isomorphic to $R$ (see \cite[Theorem 4.2]{Wat}). Hence, in what follows, we will always assume that the assumption $R=\ZZ[\pi]=\End(E)$ is satisfied.
Also notice that the main result of \cite{JKP+18} is more general and can also deal with certain supersingular elliptic curves.\\

\subsection{Polarizations}

Let $A$ be an abelian variety over $\Fq$ isogenous to a power of an elliptic curve $E$ such that $R=\End(E)=\ZZ[\pi]$. Let us recall that a polarization is an isogeny $\phi_\mathcal{L}\colon A\rightarrow \widehat{A}$ with $\mathcal{L}$ an ample line bundle. Let $L$ be a $R$-lattice. As in  \cite[Sec.4.3]{JKP+18},  we denote $L^*$ the $R$-lattice $\Hom_R(L,R)$ with the action of $r\in R$ on $\alpha \in L^*$ given by $r . \alpha(x) =\alpha(\bar{r} x)$. We want to translate  polarizations in the category of $R$-lattices.

\begin{theorem} \label{th:equivalencepol}
Let $E/\FF_q$ be an ordinary elliptic curve with $R=\End(E) = \ZZ[\pi]$ where $\pi$ is the Frobenius endomorphism of $E$. Let $F=\Frac(R)$. The functor $\F$ defines an equivalence of categories between polarized abelian varieties $A $ which are isogenous to $E^g$ and positive definite hermitian $R$-lattices $(L,h)$ of rank $g$ where $h(x,y)=\Lambda(x)(y)$ with $\Lambda : L \otimes F=V \to V^*$ a linear map such that $\Lambda^{-1}(L^*) \subset L$.  Moreover the degree of the polarization is equal to $[L:\Lambda^{-1}(L^*)]$. 
\end{theorem}

Notice also that, since $\F$ is exact, a hermitian lattice $(L,h)$ is indecomposable (see Definition~\ref{def:decomp}) if and only if the corresponding polarized abelian variety $(A,a)$ is indecomposable (i.e. $(A,a)$ is not the product of two non-trivial polarized abelian sub-varieties).

\begin{remark}
In \cite[Chap.III.Sec.8]{Ser}, Serre uses the  functor $M \to M \otimes E$ to get Theorem~\ref{th:equivalencepol} under the hypothesis that $R$ is the maximal order. In another direction, \cite[Th.A]{amir} gets a similar result for arbitrary $R$ (not necessarily quadratic) but only for projective modules. 
\end{remark}
Before giving various lemmas which will culminate in the proof of Theorem~\ref{th:equivalencepol}, in order to stick with the terminology of Section~\ref{sec:hermitianlattices} and to lead to an algorithmic version of the theorem, we now give its translation in terms of the dual lattice $L^{\#}=\{x\in V, h(x,L)\subseteq R\}.$
\begin{lemma}\label{lem:iso}
With the notation above, $\lambda:= \Lambda^{-1}$  is an isomorphism between the $R$-modules $\Ld$ and $L\s$.
\end{lemma}
\begin{proof}
Notice that $x \in V$ belongs to $\textrm{Im} \lambda$ if and only if $\Lambda(x)\in L^*$ which is the case if and only if $\forall y\in L, h(x,y)=\Lambda(x)(y)\in R$. This means, by definition, $x\in L^\#.$ Hence, $L^\#=\textrm{Im} \lambda$. 
\end{proof}
Under this isomorphism, one obtains a more natural functor as follows.
\begin{corollary}
 \label{cor:equivalencepol}
Let $E/\FF_q$ be an ordinary elliptic curve with $R=\End(E) = \ZZ[\pi]$ where $\pi$ is the Frobenius endomorphism of $E$. There is an equivalence of categories between polarized abelian varieties $A $ which are isogenous to $E^g$ and positive definite hermitian $R$-lattices $(L,h)$ of rank $g$ such that $L^\#$ is integral. Moreover, the degree of the polarization is equal to $[L:\Ld]$.
\end{corollary}
Hence, the isomorphism  classes of principally polarized abelian varieties in the isogeny class $E^g$ correspond to the isometry classes of unimodular positive definite hermitian $R$-lattices. \\

The rest of the section is devoted to the proof of Theorem~\ref{th:equivalencepol}, which will use several lemmas.\\

 In \cite[Th.4.7]{JKP+18}, it is shown that the dual of $A=\F(L)$ is functorially isomorphic to $\F(L^*)$. Hence we can  relate polarizations and injective morphisms from $L^* \to L$. Now,
 a  morphism $\lambda \colon L^*\rightarrow L$ also induces a sesquilinear form $$H_{\lambda} \colon L^*\times L^* \rightarrow R, (\alpha,\beta)\mapsto \alpha \lambda \beta.$$
We first prove the following lemma.
\begin{lemma} \label{lem:sym}
The form $H_{\lambda}$ is hermitian if and only if there exists a line bundle $\LL$ on $A=\F(L)$ such that $\F(\lambda)= \phi_{\LL}$.
\end{lemma}
\begin{proof}
Let $f : E^g \to A$ be an isogeny induced by an inclusion $\iota : L \to N \simeq R^g$. Observe that the isogeny $a=\F(\lambda)$  is of the form $\phi_{\LL}$ if and only if $a'=\hat{f} a f$ is of the form $\phi_{\LL'}$ for a line bundle 
$\LL'$ on $E^g$. The direct implication is obvious since $\hat{f} a f = \phi_{f^* \LL}$. As for the other direction, let $\ell$ be any prime distinct from the characteristic of $\FF_q$. By \cite[Th.2, p.188]{mumford-book}, the form $e_{\ell}(x,a' y)$ is  skew-symmetric and therefore the form $e_{\ell}(x,a y)$ is as well. Still using  \cite[Th.2]{mumford-book}, we then have that there exists a line bundle $\mathcal{M}$ such that $2 a = \phi_{\mathcal{M}}$ and \cite[Th.3,p.231]{mumford-book} shows that there exists $\LL$ such that $\mathcal{M} \simeq \LL^2$ hence $a=\phi_{\LL}$.

 Now, denote $\lambda'=\iota \lambda \iota^*$ so that $\F(\lambda')=a'$. Similarly, the form $H_{\lambda}$ is hermitian if and only if the form $H_{\lambda'}$ is. This equivalence can be checked on the $F$-vector spaces $FL$ and $F N$ where $\iota$ is an isomorphism. There we have that $H_{\lambda'}(\alpha',\beta')= \alpha' (\iota {\lambda} \iota^*) \beta' = 
H_{\lambda}(\alpha' \iota,\beta' \iota)$, so it is only a change of basis and the equivalence is clear.

We can therefore assume that $A=\F(R^g)=E^g$.  Let $\lambda_0 : R^* \to R$ be the isomorphism defined by $\alpha \mapsto \alpha(1)$. Since the dual of $E$ is only defined up to isomorphisms, we can assume by composing with an isomorphism that $\F(\lambda_0) : E \to \hat{E}$ is the unique principal polarization $P\mapsto \mathcal{O}([O]-[P])$ on $E$. Then the product polarization $a_0=\F(\Lambda_0)$ where $\Lambda_0 : (R^g)^* \to R^g$ is defined by $(\alpha_1,\ldots,\alpha_g) \mapsto (\alpha_1(1),\ldots,\alpha_g(1))$.  Now let $M =  \lambda \Lambda_0^{-1} \in \End(R^g) = M_g(R)$. Since $\Lambda_0 M^* \Lambda_0^{-1} = {^t \bar{M}}$, the Rosati involution $\dag$ induced by $a_0$ on $\End(E^g)=M_g(R)$ is $M \mapsto {^t \bar{M}}$. Hence, ${^t \bar{M}}=M$ if and only if $(a_0^{-1} a)^{\dag} = (a_0^{-1} a)$ i.e. $a = \phi_{\LL}$ by \cite[Prop.17.2]{milne}.
 On the other hand, the form $H_{\lambda}$ is hermitian if and only if  ${^t \bar{M}}=M$. 
\end{proof}

\begin{lemma}\label{lemdiago}
Let $L$ be a $R$-lattice of rank $g$ and $\lambda\colon L^*\rightarrow L$ injective such that $H_{\lambda}$ is hermitian. Then there exists a free over-lattice $L\xhookrightarrow{\iota} N=\bigoplus_{i=1}^g R e_i $ 
and integers $(\ell_i)_{1\leq i\leq g}$ such that if $\lambda'=\iota \lambda \iota^*$, then $H_{\lambda'} : N^* \times N^* \to R$ satisfies $H_{\lambda'}(e_i^*,e_j^*) = \ell_i \delta_{ij}$.
\end{lemma}
\begin{proof}
Since $\lambda$ is injective, the hermitian form $H_{\lambda}$ is non-degenerate. As in Theorem~\ref{ellid}(1), we can find a basis $(\alpha_i)$ of $V^*=FL^*$ of vectors of $L^*$ which is orthogonal for $H_{\lambda}$, i.e., $\alpha_i \lambda \alpha_j=\ell_{i}\delta_{ij}$ with $\ell_i\in\ZZ$. Consider $N'=\bigoplus_{i=1}^g R\alpha_i\subseteq L^*$ and then $N=N^{'*}\supseteq L^{**} \simeq L$, the last isomorphism being the evaluation map $ev: L \to L^{**}$. Denote by $\iota \colon L\rightarrow N$ the injection and $(e_i)$ the dual basis of $(\alpha_i)$. Noticing that $\alpha_i^{**} = \alpha_i \circ ev^{-1}$, we get that 
$$H_{\lambda'}(e_i^*,e_j^*) = \alpha_i^{**} (\iota \lambda \iota^*) \alpha_j^{**} = \alpha_i \lambda \alpha_j = \ell_i \delta_{ij}.$$

\end{proof}

\begin{lemma}\label{lempol}
Let $f\colon A\rightarrow B$ be an isogeny and $\mathcal{L}$ be an invertible line bundle on $B.$ Then $\mathcal{L}$ is ample if and only if $f^*\mathcal{L}$ is ample.
\end{lemma}
\begin{proof}
  An isogeny is a finite faithfully flat morphism.
  So ampleness ascends along the isogeny, since it is finite, by
  \cite[II.5.1.12]{EGA},
  and descends since it is faithfully flat \cite[IV.2.7.2]{EGA}
  (the proof holds for relative ampleness but it is easy to adapt it for ampleness, see also \cite[Exercise 5.1.29]{liu-book}). 
\end{proof}
\begin{lemma}\label{cordiago}
Let $L$ be a $R$-lattice and $A=\F(L)$ be the corresponding abelian variety. Let $\lambda : L \to L^*$ be such that $H_{\lambda}$ is hermitian and $a=\F(\lambda) \colon A\rightarrow \widehat{A}$ be the corresponding isogeny. Then there exists an isogeny $f\colon E^g\rightarrow A$, integers $(\ell_i)_{1 \leq i \leq g}$,  a map $D \in \End(E^g) : (x_1,\ldots,x_g) \mapsto (\ell_1 x_1,\ldots,\ell_g x_g)$ and a commutative diagram

\begin{figure}[H]
 \[
\xymatrix{ A  \ar[r]^{a} & \widehat{A} \ar[d]^{\widehat{f}}  \\
E^g \ar[u]^{f} \ar[r]^{a_0\circ D}& \widehat{E^g} }
\]
\caption{Fundamental diagram} \label{fig:fund}
\end{figure}
where $a_0$ is the product polarization on $E^g$. Moreover $a$ is a polarization if and only if $\ell_i>0$ for all $i$, or equivalently if and only if $H_{\lambda}$ is positive definite on $FL^*$.
\end{lemma}
\begin{proof}
Let $\iota: L\hookrightarrow N=\bigoplus_{i=1}^g R e_i$ and $\lambda'= \iota \lambda \iota^*$ be as in Lemma~\ref{lemdiago}. Consider the isomorphism $u\colon N\xrightarrow{\sim} R^g$ given by the basis $(e_i)$ of $N.$  Hence we have
\[
\xymatrix{ & & L^*  \ar[r]^{\lambda} & L \ar[d]^{\iota} &   \\
(R^g)^*\ar[r]^{\sim}_{\Lambda_0} &R^g\ar[r]^{\sim}_{u^*}& N^* \ar[u]^{\iota^*} \ar[r]^{\lambda'}& N \ar[r]^{\sim}_{u} & R^g.}
\]
We obtain the desired diagram by composing this diagram by $\F$ and taking $f=\F(u\iota).$\\
Since $H_{\lambda}$ is hermitian, there exists a line bundle $\LL$ on $A$ such that $a=\phi_{\LL}$.
Since $a_0 D $ is the pullback of $a$ by $f$ the isogeny $a$ is a polarization if and only $a_0 D $ is a polarization by  Lemma~\ref{lempol} below. Moreover, $a_0 D $ is a polarization of $E^g$ if and only if $\ell_i>0$ for all $i$. As in the first part of Lemma~\ref{lem:sym}, we can conclude that $H_{\lambda}$ is positive definite if and only if $H_{\lambda'} = \textrm{diag}(\ell_1,\ldots,\ell_g)$ is, and we have the final equivalence of the lemma.
\end{proof}{}

\begin{remark} \label{rem:cumber}
The fact that  $H_{\lambda}$ is a hermitian form on $L^*$ and not on $L$ is a bit cumbersome. Since $\lambda$ is injective, it induces an isomorphism $\Lambda:=(\lambda\otimes_R \Id_F)^{-1} \colon L \otimes F=V \rightarrow V^*$. This defines a hermitian form on $V$ given by
$$h \colon V\times V\rightarrow F, (x,y)\mapsto \Lambda(x)(y)$$
which makes $(L,h)$ a hermitian $R$-lattice.
\end{remark} %

\begin{proof}[Proof of Theorem \ref{th:equivalencepol}]
Simply combine Lemmas~\ref{lem:sym} and \ref{cordiago} with Remark~\ref{rem:cumber} to get $h$ on $L \times L$ instead of $H_{\lambda}$ on $L^* \times L^*$.
The final statement about the degree of the polarization is easily obtained  using \cite[Theorem 4.4]{JKP+18} which computes the degree of an isogeny corresponding to an inclusion of lattices with equal rank $\iota \colon L \rightarrow M$ by $\deg \F(\iota)=[M\colon \iota(L)].$  
\end{proof}

\subsection{Description of the abelian variety as a quotient of $E^g$} \label{sec:effectivekernel}
Let $(L,h)$ be a hermitian lattice with $L^\#$ integral. The goal of this section is to  compute the kernel of the isogeny $f\colon E^g\rightarrow A=\F(L)$ of Corollary \ref{cordiago} obtained by the inclusion $L\hookrightarrow R^g$ induced by $\iota$ in Lemma \ref{lemdiago} after identification of $N=\bigoplus R e_i$ with $R^g$.

As a first step, to apply Corollary~\ref{cor:equivalencepol}, we need to start with an explicit elliptic curve $E/\FF_q$ with ring of endomorphism $\ZZ[\pi]$. If $q$ is small, it is efficient to use an algorithm which determines the endomorphism ring of an ordinary elliptic curve \cite{eisentrager-lauter} or \cite{bisson-sutherland}. In the package implemented, we consider a version of the latter kindly provided by Sutherland and apply it to the list of elliptic curves with a given trace (which can be naively obtained from the list of all elliptic curves by computing the trace on each of them). If the discriminant of $\ZZ[\pi]$ is small, the best way to obtain $E$ is to compute a root $j$ over $\FF_q$ of the Hilbert class polynomial of $\ZZ[\pi]$ and find among the elliptic curves with $j$-invariant $j$ the one with the right trace. We refer to \cite{enge2009complexity,sutherlandCM} for algorithms to compute this class polynomial. Note that the complexity of the rest of the algorithm strongly depends on the discriminant of $\ZZ[\pi]$, so choosing this second method, this step is never the bottleneck of the whole algorithm.

Given an inclusion of equal rank $g$ $R$-lattices $\iota \colon L_1\rightarrow L_2$ and surjective morphisms $T_i\colon R^{m_i}\rightarrow L_i$ we can lift $\iota$ to $P\in M_{m_2,m_1}(R)$ 
 \[
\xymatrix{  R^{m_1}  \ar[r]^{T_1}\ar[d]_{P} & L_1 \ar[d]^{\iota}   \\
R^{m_2} \ar[r]^{T_2}& L_2 }
\]
by computing the image of the canonical basis of $R^{m_1}$ by $\iota\circ T_1$ and taking any preimages by $T_2.$ Since the morphisms $T_i$ are surjective, $\F(T_i)$ are injective and the kernel of the corresponding isogeny $\F(\iota)=f\colon \F(L_2)\rightarrow \F(L_1)$ can be computed by $\ker f=\F(T_2)^{-1}\ker \trsp{P}.$

In the present situation, $L_1=L, L_2=N = \bigoplus R e_i, m_2=g$ and $T_2=\Id$. It remains to make $T_1$ and the $(e_i)$ explicit. Consider a pseudo-basis $L=\fa_1 x_1\oplus\dots\oplus \fa_g x_g.$ Since the $\fa_i$ are fractional $R$-ideals they have at most $2$ generators. Hence, $L$ is generated by $g\leq r\leq 2g$ generators. So, $T_1$ is the surjective morphism $R^r\twoheadrightarrow L$
sending the canonical basis of $R^r$ on the generators of $L$.
Applying the functor $\F$ to the composition leads to the commutative diagram
\[
\xymatrix{E^{r} & &\ar@{_{(}->}[ll]_{\F(T_1)} A   \\
E^g \ar[u]^{\trsp{P}} \ar[rru]_{f}   }
\]
 and $\ker f =\ker(\F(T_1)\circ f)=\ker \trsp{P}.$ By Figure~\ref{fig:fund}, one sees that $\ker f \subseteq \ker D=\prod_{i=1}^g E[\ell_i]\subseteq E[\ell]^g$ with $\ell=\textrm{lcm}(\ell_i)$ and $D$ the map of Lemma \ref{cordiago}.  Thus, it is enough to compute the action of $\trsp{P}$ on a basis of the $\ell$-torsion of $E^g$ to have the whole kernel. To go on, we will assume that $\ell$ is prime to $\car \Fq$ so $E[\ell] \simeq \left(\ZZ/\ell\ZZ\right)^2$ is \'etale and we can work with geometric points. 

Let clarify how  to compute the family $(e_i)_{1 \leq i \leq g}$. Let us recall that we defined it as the dual basis of an orthogonal family of $L^*$ so they satisfy $e_i^*\lambda' e_j^*=\ell_i\delta_{ij}$. This means that $\lambda'(e_i^*)=\ell_ie_i$ and then $h(e_i,e_i)=\frac{1}{\ell_i}$ (see proof of Lemma~\ref{lem:iso}). Consider an orthogonal family $(u_i)_{1\leq i\leq g}$ of vectors of $L^\#$ of norm $(\ell_i)_{1 \leq i \leq g}$  and let $e_i=\frac{1}{\ell_i}u_i.$ By the inclusion $\bigoplus_{i=1}^{g} R u_i\subseteq L^\#$ we have $$L\subseteq \left(\underset{1 \leq i \leq g}{ \mathlarger{\mathlarger{\perp}}} R u_i\right)^\#= \underset{1 \leq i \leq g}{ \mathlarger{\mathlarger{\perp}}} (R u_i)^{\#} =\underset{1 \leq i \leq g}{ \mathlarger{\mathlarger{\perp}}} R e_i.$$ 
Hence, if we find an orthogonal family $(u_i)_{1\leq i\leq g}$ of $L^\#$ with norm $(\ell_i)_{1 \leq i \leq g}$ then $(e_i)_{1 \leq i \leq g}=(1/\ell_i \cdot u_i)_{1\leq i\leq g}$ is an orthogonal family of norm $(1/\ell_i)_{1 \leq i \leq g}$ suited for the inclusion $\iota \colon L\rightarrow \bigoplus_{i=1}^g R e_i$.

We summarize these computations in Algorithm~\ref{algo:kerneliso}.

\begin{algorithm}[H] %
\begin{algorithmic}[1] %
	\Require A $R$-lattice $(L,h)$ and $E$ an elliptic curve over $\FF_q$ with $\End(E)=\ZZ[\pi] \simeq R$.
   \Ensure A basis of the kernel of an isogeny $f : E^g \to \F(L)$ such that the polarization $a$ on $L$ induced by $h$ satisfies $\hat{f} a f$ is a completely decomposable polarization. 
   \State Compute an orthogonal family $(u_i)_{1\leq i \leq g}$ of $\Ld$ of norms $\ell_i \in \ZZ$ using the Gram-Schmidt process (when $\ell_1=\ldots=\ell_g$, use Algorithm~\ref{algo:ortho2}). Define $e_i=u_i/\ell_i$, $\ell=\lcm(\ell_i)$, $N=\bigoplus_{i=1}^g R e_i$ and  $\iota$ the inclusion of $L$ in $N$. 
   \State    Compute a pseudo-basis $L=\fa_1 x_1\oplus\dots\oplus \fa_g x_g$ given by $1 \leq r \leq 2g$ generators. Let  $T_1\colon R^r\rightarrow L$ be a surjective morphism that sends the canonical basis of $R^r$ on the generators of $L$. 
    \State Let $P$ be the matrix of the morphism $\iota \circ T_1 : R^r \to N$ in the canonical basis  of $R^r$ and the basis $(e_i)$ of $N$.
    \State Compute a basis $(b_0,b_1)$ of $E[\ell]$. This allows us to identify $E[\ell]$ with $(\ZZ/\ell\ZZ)^{2}$ and let $\mu\colon E[\ell]^{2g}\rightarrow (\ZZ/\ell\ZZ)^{2g}$ be the isomorphism induced by the identification.
    \State Compute the action of the Frobenius $\pi$ on $(b_0,b_1)$ as a matrix  $\Pi \in M_2(\ZZ/\ell\ZZ).$
    \State Create a matrix $Q\in M_{2r,2g}(\ZZ/\ell\ZZ)$ by replacing each entry $a+b \pi$ of $\trsp{P}$ by $a I_2 + b \Pi$. 
    \State Compute a basis $\mathcal{B}$ of  $\ker Q\in (\ZZ/\ell\ZZ)^{2g}$.
     \State \textbf{return} $\mu^{-1}(\mathcal{B})$ .
\end{algorithmic}
\caption{Computation of the kernel of an isogeny $E^g \to A$} \label{algo:kerneliso}
\end{algorithm}
We will use this algorithm with the additional condition $\ell_1=\ldots=\ell_g$. Indeed, in Section~\ref{sec:theta}, we need more specific properties about the kernel $K$ of the isogeny in order to be able to compute the theta null point on $A$ using the current algorithms.
We first require that $\ell=\ell_1=\ldots=\ell_g$ is odd and prime to $q$ (see \cref{rem:details} for the condition $\ell$ odd). We have discussed in Theorem~\ref{ellid}, when this can be achieved.
By \cite[Prop.16.8]{milne}, $K$ is a maximal isotropic subgroup of $E[\ell]^g$ for the Weil pairing on $E^g$ induced by the product polarization. However for the algorithms we also need $K$ to be of rank~$g$, that is
isomorphic as a group to $(\ZZ/\ell\ZZ)^g$. We call such a $K$ a \emph{totally
isotropic} subgroup. Equivalently, for an abelian variety $A_0$, $K \subset A_0[\ell]$ is a \emph{totally isotropic} subgroup of level~$\ell$ if it is isotropic, and one can find a symplectic decomposition $A_0[\ell]=K \oplus K'$.
If $K$ is maximal isotropic, it is always totally isotropic when $\ell$ is square
free, but this can fail if $\ell$ has a square factor (for instance $A_0[\ell]$ is maximal isotropic in $A_0[\ell^2]$). We adopt a pragmatic approach here and test that a given $K$ has indeed the right group structure. In every computation we made, when an odd $\ell$ exists, we always found one for which $K$ was totally isotropic.

%% file: sections/theta2.tex
\section{Theta structures and a modular interpretation of the isogeny formula}
\label{sec:theta}
In this section, $k$ is any field of characteristic $p \ne 2$. We will first recall in \cref{sec:isogenyformula} how to use the so-called isogeny formula to derive the theta null point on a target abelian variety from a (well-chosen) isogenous one. Then, in Section~\ref{sec:universal}, we will show that the isogeny formula is actually valid over the universal abelian scheme. Although the proof basically follows the same lines as the proof over a field, this result, and the notation introduced there, will be useful in Section~\ref{sec:modular}, where we will derive a precise affine version of the isogeny formula.
More precisely, we introduce a particular choice of affine lifts of the theta null points which we call \emph{modular}, since they are derived from interpreting the theta constants as modular forms, and we show in Theorem~\ref{th:modular} that the isogeny formula respects the modular lifts.
In \cref{sec:thomae}, we explain how to compute $k$-rational modular lifts for
a product of elliptic curves with a product polarization.
Combining the `modular' isogeny formula and these initial modular lifts
allow us in \cref{sec:siegelisogeny} to
compute values of Siegel modular forms of even weight given as polynomials in the theta constants with coefficients in $k$ on the span of the isogeny class (see \cref{th:siegelmodular} and Algorithm~\ref{algo:modular}).

\subsection{Input for the isogeny formula over $k$} 
\label{sec:isogenyformula}

Let $(A, \pol, \Theta_\pol)/k$ be a principally polarized  abelian variety of
dimension~$g$ with a totally symmetric theta structure $\Theta_\pol$ of level~$n$ on $\pol$. This implies that $n$ is even which we assume from now on (until the end of this section).
Let $K$ be a $k$-rational totally isotropic subgroup for the Weil pairing of
$\pol^\ell$, with $\ell$ prime to $n p$ or to $n$ if $p=0$.

In~\cite{DRniveau,DRisogenykernel}, an algorithm (which we call the \emph{isogeny
formula} and implemented in the package \texttt{Avisogenies} \cite{DRAVIsogenies}) is given to compute the
isogeny $f: (A, \pol, \Theta_\pol) \to (B, \bpol, \Theta_\bpol)$ where
$B=A/K$, $f^\ast \bpol=\pol^\ell$ and $\Theta_\bpol$ is the unique
symmetric theta structure of level $n$ on $\bpol$ compatible with
$\Theta_\pol$ (the unicity comes from the fact that $\ell$ is 
prime to $n$).
More precisely the algorithm takes as input the (projective) theta null
point $\theta^A(0)\defby\left(\theta^A_{i \in \Zn}(0)\right) \in \PP(\bar{k})^{n^g-1}$ of $A$, where
$\Zn=(\ZZ/n\ZZ)^g$, along with the theta coordinates of the geometric
points of $K$ (or suitable equations giving the kernel $K$) and outputs the
theta null point $\theta^B(0)\defby\left(\theta^B_{i \in \Zn}(0)\right)$ of $B$ along with
the equations for the isogeny $f$. We usually take $n=4$ (since this is the smallest even $n$ which gives an embedding of the variety into projective space) and the theta null point completely characterizes $(B,\bpol)$ up to $\bar{k}$-isomorphism. We will describe in more details (a generalisation of) this algorithm in \cref{sec:universal}. In this section we explain how to compute the inputs for the isogeny formula in our situation.

Let $E/k$ be an elliptic curve. If $(B,\bpol)$ is isogenous to $E^g$, we show how to compute $\theta^B(0)$ of level $4$ by applying the algorithm with $A=\prod_{i=1}^g E_i$ where $E_i$ are elliptic curves over $k$ isogenous to $E$ and $\pol$ the principal product polarization on $A$. For this, we need three elements as inputs for the algorithm:
\begin{itemize}
\item compute a totally isotropic kernel $K$ such that $B=A/K$. When $k$ is a finite field and $E=E_1=\ldots=E_g$ is ordinary,
we have seen in Section~\ref{sec:effectivekernel} how to do this effectively;
\item compute the theta null point $\theta^A(0)$ of level $4$ on $(A,\pol)$. As we deal with the product polarization, the coordinate $\theta_{i_1,\ldots,i_g}^A(0)$ of $\theta^A(0)$ is equal to $\prod_{1 \leq j \leq g} \theta_{i_j}^{E_j}(0)$. Getting the theta null point on an elliptic curve (over a field of odd characteristic) is a classical result. In Corollary~\ref{cor:thomae}, we give an even more precise version of this to which we refer now and that we use in Step 1 and 2 of Algorithm~\ref{algo:thetaquotient}. 
\item compute the theta coordinates of the points in the kernel $K$. Likewise
  since we have a product polarization,
  $\theta_{i_1,\ldots,i_g}^A(x_1, \ldots, x_g) = \prod_{1 \leq j \leq g} \theta_{i_j}^{E_j}(x_j)$.
Computing the theta coordinates $\left(\theta^{E_i}_j(x_i)\right)_{j \in \ZZ/4\ZZ}$ is also classical \cite{mumford-tata2}, \cite[Chapter~5]{cosset2011thesis}, and is implemented in \texttt{Avisogenies} \cite{DRAVIsogenies}.
\end{itemize}
We therefore get the following algorithm~\ref{algo:thetaquotient}.

\makeatletter
\newcommand\footnoteref[1]{\protected@xdef\@thefnmark{\ref{#1}}\@footnotemark}
\makeatother

\begin{algorithm}[H]
\begin{algorithmic}[1]
  \Require{Elliptic curves $E_i/k$   with equation  $y^2=(x-e_{1i})(x-e_{2i})(x-e_{3i})$ where $k$ is of characteristic $p$ different from $2$, a $k$-rational totally isotropic subgroup $K$ of $A=\prod_i E_i$ of order prime to $2 p$ (or just prime to $2$ if $p=0$).}
 \Ensure{The theta null point $\theta^B(0)$ of level $4$ on $B=A/K$ with $\bpol$ the polarization induced by the product polarization on $A$.}
\State{For all $1 \leq i \leq g$, define $\theta'^{E_i}_{0}=\sqrt[4]{e_{1i}-e_{3i}},  \theta'^{E_i}_1= \sqrt[4]{e_{1i}-e_{2i}}, \theta'^{E_i}_2 =\sqrt[4]{e_{2i}-e_{3i}}$ for arbitrary choices of the roots.}
\State{Compute $\theta^{E_i}_{0}(0)=\theta'^{E_i}_{0} +\theta'^{E_i}_{1}$, $\theta^{E_i}_{2}(0)=\theta'^{E_i}_{0} -\theta'^{E_i}_{1}$ and $\theta^{E_i}_{1}(0)=\theta^{E_i}_{3}(0)=\theta'^{E_i}_{2}$ for all $1 \leq i \leq g$.}
\State{Compute  $\theta^A_{(i_1, \ldots, i_g)}(0)=\theta^{E_1}_{i_1}(0) \cdots \theta^{E_g}_{i_g}(0)$ for all $(i_1,\ldots,i_g) \in \Zfour$.}
\State{For all $1 \leq i \leq g$ and for all  $x=(x_1,\ldots,x_g) \in K\setminus \{0\}$, compute the theta coordinates $\left(\theta^{E_i}_j(x_i)\right)_{j \in \ZZ/4\ZZ}$, using \cite[Chapter~5]{cosset2011thesis}, %
}
  \State{Compute  for all $j=(j_1,\ldots,j_g) \in \Zfour$ and  for all $x=(x_1,\ldots,x_g) \in K \setminus \{0\}$
$\theta^{A}_j(x) = \theta^{E_1}_{j_1}(x_1) \cdots \theta^{E_g}_{j_g}(x_g)$.}
\State{Use \cite{DRniveau}%
taking as input $\theta^{A}(0)$ and the theta coordinates of the points of $K$ and output $\theta^B(0)$.}\\
\Return{$\theta^B(0) \in \PP(\bar{k})^{4^g-1}$.}
\end{algorithmic}
 \caption{Computation of the theta null point of level $4$ on the quotient variety} \label{algo:thetaquotient}
\end{algorithm}

\begin{remark} \label{rem:details}
Some remarks on the code:
\begin{itemize}
\item %
The original version of \texttt{Avisogenies} assumed $\ell$ to be a prime.   The only modification to the code we had
to make is on how to construct a matrix $F \in \Mat_r(\ZZ)$ such that
$\transp{F} F = \ell \Id$ used by Koizumi's formula~\cref{eq:koizumi}.
The integer $r$ depends on $\ell$ being a square (hence $r=1$), $\ell$ being a sum of two positive squares (hence $r=2$) or a sum of four positive square (hence $r=4$). Adapting the construction of $F$ to $\ell$ odd non prime is straightforward by multiplicativity of the complex norm (if $r=2$) or of the quaternionic norm (if $r=4$).
\item  The restriction $\ell$ odd is not necessary in theory
    if some great care is taken. First the lift from level~$n$ to level~$\ell n$
    is more complicated since we cannot work only on the points in the
    kernel $K$. We first need to compute a basis of points $P_i$ such that $n
    P_i$ is a basis of $K$ (this was given to us for free before by the
    CRT). Furthermore this basis has to be compatible with the level~$n$
    structure on $A$, so this may require first to act by an automorphism of
    the theta structure to make the level~$n$ structure on $A$ compatible
    with $K[n]$. Secondly, if $\ell$ is not odd, then there may be several
    symmetric theta structures on $B$ compatible with the one on $A$. So the
    isogeny formula in this case yields several solutions.
    This has not yet been implemented in \cite{DRAVIsogenies}.
\end{itemize}
\end{remark}

\subsection{The isogeny formula on the universal abelian scheme}
\label{sec:universal}

In this section we reformulate the isogeny formulae from \cite{DRniveau} to
show that the formulae are polynomials with coefficients in
$\ZZ[\frac{1}{\ell n}]$ in the coordinates of the points of $K$. Since the
fine moduli scheme (or stack if $n \leq 2$) $\Agn$ of abelian
varieties with a symmetric theta structure of level~$n$ is smooth (or by
rigidity~\cite[\S~6]{mumford1994geometric}), the isogeny formula is thus
valid on the universal abelian variety defined over $\ZZ[\frac{1}{\ell
n}]$. Though well known to experts, this is not completely obvious in the
formulation of~\cite{DRniveau} since the authors only work with fields and
implicitly use divisions in their equations.

We first give some motivations for this result. In \cref{sec:modular} we
give an algebraic modular interpretation of the isogeny formula by first
considering the analytic modular interpretation over $\CC$. It is then
possible, by standard lifting arguments to extend this result to ordinary
abelian varieties over a finite field. But, while possible, this is a bit
painful to do properly since we want to control the lifts of the
endomorphisms along with the differentials, and then give an algebraic
meaning to the reduction of the period matrix modulo $p$. By contrast,
showing that the isogeny formula is actually defined over
$\ZZ[\frac{1}{\ell}]$ yields a much simpler proof that the analytic
interpretation holds algebraically. Indeed, by smoothness, the modular
interpretation is ultimately a statement about the equality of two
multivariate polynomials defined over $\ZZ[\frac{1}{\ell}]$. But this
equality holds when it holds over $\CC$.
In addition, this proof holds for all abelian varieties rather than just the ordinary ones.
The notations introduced in this section will also be useful in
\cref{sec:modular} where we keep track of each modular factor at each step
of the algorithm.

In order to avoid heavy notation, we will often let the theta
structure $\Theta_{\pol}$ (and eventually the polarization $\pol$) be implicit,
along with the coordinate group $\Zn$.

Assume from now on that $n$ is (even and) greater or equal to $4$ and $\ell$ prime to $n$. Mumford
constructs in \cite{MumfordOEDAV2} the universal abelian variety
$\Xgn \to \Agn$ with a totally symmetric normalized  relatively ample line bundle\footnote{See \cite[Definition~p.78]{MumfordOEDAV2} for the definition of these terms.} and a symmetric theta structure of level~$n$
 over $\ZZ[1/n]$\ \footnote{The irreducible components are defined over $\ZZ[1/n,\zeta_n]$ since over this ring all points of the level~$n$ Heisenberg group $\Hn$ are defined.} as a quasi-projective scheme.
Moreover Mumford uses
Riemann's relations \cite[p.~83]{MumfordOEDAV2} to define a projective scheme $\Xbargn \to \Abargn$
(where the equations of $\Abargn$ are given by evaluating the Riemann's relations on the zero section, together with the symmetry relations $\theta_i(0)=\theta_{-i}(0)$) and an embedding of $\Xgn \to \Agn$ into $\Xbargn \to \Abargn$ (so that $\Xgn$ is the pullback of $\Xbargn$ to $\Agn$).
We denote $(\theta_i)_{i \in \Zn}$ the theta coordinates on either $\Xgn$ or
$\Xbargn$ and
$(\theta_i(0))_{i \in \Zn}$ the theta null point coordinates on either $\Agn$ or
$\Abargn$ coming from the section $s: \Abargn \to \Xbargn$ (which
restricted to $\Agn$ corresponds to the zero section).

On $\Xbargn$,
we have an explicit
action $\lambda$ of the Heisenberg group $\Hn$ on $\pol_{\Xbargn}$
\cite[Step~1, p.~84]{MumfordOEDAV2}. Writing $\Hn=\Gm \times \Zn \times
\dZn$ where $\dZn \iso \oplus_{i=1}^g \mu_n$ is the Cartier dual of $\Zn$,
this canonical action is given by $\lambda(i).\theta_j=\theta_{i+j}$ for $i
\in \Zn$ and $\lambda(i).\theta_j=<i,j>\theta_j$ for $i \in \dZn$
where $<i,j>$ is the canonical pairing between $\Zn$ and its Cartier dual $\dZn$.
Acting on the zero section $s$ gives a canonical basis of $n$-torsion.

Mumford's isogeny theorem \cite{MumfordOEDAV1} then describes the universal isogeny (with a descent of
level of the theta structure) 
\begin{equation}
\pi_1:\Xgln \to \Xgn, (\theta_i)_{i \in \Zln} \mapsto (\theta_i)_{i \in \Zn \subset \Zln}.
\end{equation}

On $\Xgln$ the level $\ell n$ theta structure induces a symplectic basis of
the $\ell n$-torsion, and in particular a symplectic decomposition
$K_1 \oplus K_2$ of the $\ell$-torsion. Concretely over a field $k$, 
$K_1 = \{ (<i,j> \theta_j(0))_{j \in \Zln}\}_{i \in \dZl}$ is the kernel of $\pi_1$, while $K_2=\{ (\theta_{i+j}(0))_{j \in \Zln}\}_{i \in \Zl}$
is such that $\pi_1(K_2)=\{ (\theta_{i+j}(0))_{j \in \Zn}\}_{i \in \Zl}$ is the kernel of the contragredient isogeny $\tilde{\pi}_1$.

Using $\pi_1$, we can now describe the isogeny formula in three steps.

\myparagraph{Step 1}
Denote $\Pi_1: \Xgln \to \Xgn^{\ell^g}, (\theta_i)_{i \in \Zln} \mapsto
\left(\pi_1(\lambda(i) (\theta_j))_{j \in \Zln}\right)_{i \in Z(l)}$, where
$\lambda$ is the action of the Heisenberg group $\Hln$ described above. 
For $j \in \Zl$ the component $\Pi^j_1$ of $\Pi_1$ is given by
\begin{equation}
  {\Pi^j_1}^\ast(\theta^{\Xgn}_i)=\theta^{\Xgln}_{i+j}, \quad i \in \Zn.
  \label{eq:Pi1}
\end{equation}
The
image of the restriction of $\Pi_1$ to $\Agln$ (seen as the zero section of
$\Xgln$)
then describes the moduli scheme $\Tmod$ of abelian varieties with a level~$n$
symmetric theta structure together with the points of an isotropic kernel
of the $\ell$-torsion.

It is easy to see that $\pi_1$ extends to a morphism $\pibar_1: \Xbargln \to
\Xbargn$. Since the action $\lambda$ is defined on $\Xbargln$, we can also
extend $\Pi_1$ to a morphism $\Pibar_1: \Xbargln \to \Xbargn^{\ell^g}$.
Let $\Tbar$ be the image of $\Abargln$. By construction $\Tmod$ embeds into $\Tbar$
and since we have explicit equations for $\Abargln$ we have equations for
$\Tbar$.

By construction, given a  $k$-point $(A_0,K_0)$ of $\Tmod$, 
geometric points of
$\Pi_1^{-1}(A_0,K_0) \to \Agln$ corresponds to abelian varieties $B_{0,\kbar} \in
{\Agln}(\kbar)$ with a level $\ell n$ symmetric theta structure
such that the universal isogeny $\pi_1$ restricted to $B_0$
is the contragredient isogeny of $A_{0,\kbar} \to A_{0,\kbar}/K_{0,\kbar}$.
In particular, starting with our abelian variety $(A, \pol)/k$, if $k'$ is an \'etale
extension of $k$ such that all points of $K$ are defined, then fixing an
isomorphism $\Zl \to K$ over $k'$ yields a $k'$-point of $\Tmod$. A
$k"$-point in $\Pi_1^{-1}(A,K)$ then correspond to a theta structure on
$(B, \Mpol^\ell)$ defined over $k"$ such that the contragredient isogeny
$\ftilde: B \to A$ is given by the pullback of $\pi_1$ to $B$.

The discussions in \cite[Corollary~3.6, Proposition~3.7]{DRarithmetic},
\cite[Algorithm~4.4.10]{DRphd}), \cite[\S~4.1]{DRniveau},
\cite{DRisogenies} can then be  reinterpreted as a way to use Riemann
relations to give explicit equations for $\Pibar_1^{-1}(A,K)$ and
$\Pi_1^{-1}(A,K)$.

\myparagraph{Step 2}
Now let $r=1$ if $\ell$ is a square, $r=2$ if $\ell$ is a sum of two
squares and $r=4$ otherwise (the reason of our choice of $r$ will appear in
Step 3). On $\Agln$ the Segre embedding yields a map $\pi_2: \Ag \to \Agrln$,
which sends the universal abelian variety $\Xgln$ to $\Xgln^r$ with its
product theta structure~\cite[Lemma~1, p. 323]{MumfordOEDAV1}.
Concretely,
\begin{equation}
\pi_2^\ast(\theta^{\Xgrln}_{i_1,\ldots,i_r})=\theta^{\Xgln}_{i_1} \cdot \dots \cdot \theta^{\Xgln}_{i_r}
\label{eq:segre}
\end{equation}

In particular, $\pi_2$ sends the theta null point of level~$\ell n$ of $(B,
\Mpol^\ell)$ to the theta null point of $(B^r, \Mpol^\ell \star  \dots
\star \Mpol^\ell)$\ \footnote{If $\pol_1$ is a line bundle on $A_1$ and
  $\pol_2$ is a line bundle on $A_2$ we use the notation $\pol_1 \star \pol_2$
  to denote the line bundle $p_1^\ast \pol_1 \otimes p_2^\ast \pol_2$ where
 $p_i$ is the projection $ A_1\times A_2 \to A_i$.}. 

\myparagraph{Step 3}
Let $F$ be an $r \times r$ matrix with integral coefficients such that $\transp{F} F=\ell \Id$ (see Remark~\ref{rem:details}). Then the Koizumi-Kempf formula \cite{koizumi1976theta, kempf1989linear} yields a map
$\pi_3: \Agrln \to \Agrn$ which corresponds to the isogeny $F: \Xgln^r \to
\Xgln^r$ along with the descent of product theta structure from level $\ell
n$ to level $n$.
The formula is given, for $(i_1, \ldots, i_r) \in \Zn^r$, by
\begin{equation}
  \pi_3^\ast(\theta^{\Xgrn}_{i_1,\ldots,i_g})=F^\ast(\theta^{\Xgn}_{i_1} \cdot \dots \cdot \theta^{\Xgn}_{i_r})
  = \sum_{\substack{(j_1, \ldots, j_r) \in {\Zln}^r \\
  F(j_1,\ldots,j_r)=(i_1,\ldots, i_r)}}
  \theta_{j_1}^{\Xgln} \cdot \dots \cdot \theta_{j_r}^{\Xgln}.
  \label{eq:koizumi}
\end{equation}
Since \cref{eq:koizumi} is homogeneous, this is well defined for projective
coordinates.

In particular, $\pi_3$ uses $F$ to send $(B^r, \Mpol^\ell \star  \dots \star \Mpol^\ell)$
to $(B^r, \Mpol \star  \dots \star \Mpol)$, from which $(B, \Mpol)$ can be
recovered by projecting to one of the factor.

The \emph{isogeny formula} is then the composition $\pi_3 \circ \pi_2 \circ
\Pi_1^{-1}$.

\begin{theorem} \label{th:isoproj}
Let $n$ be an even integer greater or equal to $4$ and $\ell$ be an integer prime to $n$.
The image of $\Pi_1 \times \pi_3 \circ \pi_2: \Agln \to \Tmod \times \Agn$ induces a modular correspondence defined over $\ZZ[\frac{1}{\ell n}]$.

 Let $k$ be a field of characteristic prime to $\ell n$. If $(A,K)$ is a $k$-point of $\Tmod$, then $\pi_3 \circ \pi_2 \circ
  \Pi_1^{-1}(A,K)$ only has a single $\kbar$-point (with multiplicity $\ell^g$ and which is actually defined over~$k$), corresponding to $A/K$.

  This point can be computed in $O(\ell^{g \max(1,r/2)})$ operations in $k$ where, by assumption, $k$ contains the field of definition of the geometric points of $K$.
  \label{th:isogeny_formula}
\end{theorem}
\begin{proof}
The first part follows from the steps above. For the statement over a field $k$,
  by construction, each geometric point in $\Pi_1^{-1}(A,K)$ corresponds to
  $B=A/K$ with a level $\ell n$ structure compatible with the level $n$
  structure on $A$. Descending the product level $\ell n$ structure via $F$
  then induce the same level $n$ structure on $B$.

  For the complexity estimate, writing equations for $\Pi_1^{-1}$ is in
  $O(\ell^g)$ operations, the Segre embedding only depends on $n$ so is
  absorbed by the big $O$ notation, and computing $\pi_3$ requires
  $O(\ell^{r/2})$ operations, hence the total complexity. We refer to
  \cite{DRniveau} for more details.
\end{proof}

\subsection{Modular interpretation} \label{sec:modular}

Consider  again the algorithm from Theorem~\ref{th:isoproj} but suppose now that we would like to apply it to an affine lift of
a  theta null point of $(A,\pol,\Theta_{\pol})$.
Notice that the choice of an affine lift is induced by the choice of a trivialization of $\pol$ since the $\theta_i^A$ are sections of a power of $\pol$.
Since $\pi_1,
\pi_2$ and $\pi_3$ are well defined as affine morphisms (using the exact
same equations), we can also interpret the isogeny formula $\pi_3 \circ
\pi_2 \circ \Pi_1^{-1}$ as an \emph{affine isogeny formula}, yielding an affine lift
of the theta null point of $B=A/K$. 

In this section, we want to achieve two goals: give the precise relation between affine lifts on $A$ and $B$ through the affine isogeny formula (Theorem~\ref{th:modular}) and also show that we can compute Siegel modular forms constructed as polynomials in the theta constants.

For both purposes, we will need modularity and we therefore start with some classical notions on Siegel modular forms (see for instance \cite{chai,demum,faltingschai,vandergeer}). As before, let $g \geq 1$, $n$ even and greater or equal to $4$. Let $\pi: \Xgn \to \Agn$ be the universal abelian variety with a totally symmetric normalized relatively ample line bundle and a symmetric theta structure of level $n$  over $\ZZ[\frac{1}{n}]$ and $s: \Agn \to \Xgn$ be
the zero section. We denote $\Hodge = \Wedge^g (s^\ast \Omega_{\Xgn}) = \Wedge^g (\pi_\ast\Omega_{\Xgn})$ the Hodge line bundle. 

Let $R$ be a commutative ring with all residue fields $k$ of characteristic $p=0$ or prime to $n$.  
Recall that a (scalar) Siegel modular form $\chi$ of integral weight $\rho \geq 1$ and level $n$\ \footnote{Here by level~$n$ we mean the level group $\Gamma_g(n, 2n)$ of matrices $\gamma \in \Sp_{2g}(\ZZ)$ such that $\gamma = \begin{bmatrix} A & B \\ C & D \end{bmatrix} \equiv \textrm{Id} \pmod{n}$ and $2n$ divides the diagonals of $B$ and $C$.} over $R$
is a section of $\Hodge^\rho$ on $\Agn \otimes R$\ \footnote{At least when $g>1$. When $g=1$ we also need to check that the modular form stays bounded at infinity, or algebraically that the evaluation on the Tate curve is given by a Laurent series in $q$ with no negative terms.}.
For a given $(A,\pol,\Theta_{\pol}) \in \Agn(k)$ 
and  $w_A$ a basis of $k$-rational regular differentials on $A$, it can also be seen as a function
$\chi: (A,  \pol, \Theta_{\pol}, w_A) \mapsto k$,  such that $\chi(A, \pol, \Theta_{\pol}, \lambda w_A)=(\det
\lambda)^\rho \cdot \chi(A, \pol, \Theta_{\pol}, w_A)$ for any $\lambda \in \GL_g(\kbar)$.
Likewise, a Siegel modular form $\chi$ of weight $\rho$ and level~$1$\ \footnote{Meaning the full level group $\Gamma_g=Sp_{2g}(\ZZ)$ and not $\Gamma_1(1,2)$.} is a section of $\Hodge^{\rho}$ on the algebraic stack $\Ag$ of
principally polarized abelian schemes. In that case, we simply write  $\chi(A, \pol,w_A)$.
Let $\pol_{\Xgn}$ be the totally symmetric normalized relatively ample line bundle on
$\Xgn$ as in Section~\ref{sec:universal}. %
Let $\iota : \Spec k \to \Agn \overset{s}{\to} \Xgn$ corresponding to a closed point $(A,\pol,\Theta_A) \in \Agn(k)$. We have that $\iota^* \theta^{\Xgn}(0) = \theta^{A}(0)$, as projective coordinates. 
In the special case where $k=\CC$, let $\Omega$ be a Riemann matrix in the Siegel upper half-space $\mathbb{H}_g$  and let us denote  $\thetacarac{x_1}{x_2}(0,\Omega)$, the value at $0$ of the classical theta function with characteristic $(x_1,x_2) \in \QQ^{2g}$ \cite[p.192]{mumford-tata1}. We will refer to these complex values as \emph{theta constants} (in contrast with the theta coordinates when speaking about the $\theta_i^A(0)$).
Following \cite[Prop.~5.11]{mumford-tata3} (see also loc.~cit. Definition.~5.8 and p.~36), if 
$(A,\pol,\Theta_A)=\CC^g/(\ZZ^g+ \Omega \ZZ^g)$,
with its associated polarization induced by $\textrm{Im} \Omega^{-1}$ and associated canonical symmetric level structure induced by the canonical symplectic basis on the lattice,
then  $(\theta^{A}_i(0))_{i \in \Zn}$ is projectively equal  to $(\thetacar{0}{i/n}(0,\Omega/n))$  for  arbitrary lifts of $i \in \Zn$ to $\ZZ^g$.
In fact Mumford shows this equality for the adically defined theta
functions. For the level~$n$ algebraic theta functions, it suffices to
remark that both the algebraic $\theta_i(z)$ and analytic $\thetacar{0}{i/n}(z, \Omega/n)$ theta functions satisfy the canonical
irreducible representation of the Heisenberg group of level~$n$
\cite[Theorem~2 and definition p. 297]{MumfordOEDAV1}. 

We will use this projective equality to fix a particular choice of affine lifts over any field in the following way. Because of the transformation formula \cite[Cor.5.11]{mumford-tata1}, if we define for any $i,j \in \Zn$, 
\begin{equation}
\chi_{ij}(A,\pol,\Theta_A,(2i\pi dz_1,\ldots,2i\pi dz_g)) = \thetacar{0}{i/n}(0,\Omega/n)\cdot \thetacar{0}{j/n}(0,\Omega/n)
\label{eq:chiij}
\end{equation}
we get Siegel modular forms of weight $1$ and level $n$ over $\CC$. Since the Fourier coefficients of the theta constants belong to $\ZZ$, by the $q$-expansion principle \cite[p.140]{faltingschai}, this definition can be extended to a section of $\Hodge$ over $\ZZ[\frac{1}{n}]$ and therefore over $R$. Since the sections $(\chi_{ij})_{i,j \in \Zn}$ and $(\theta_i^{\Xgn}(0) \theta_j^{\Xgn}(0))_{i,j \in \Zn}$ are equal up to a constant over $\CC$,  for any $(A,\pol,\Theta_A) \in \Xgn(k)$ and $w_A$ a basis of $k$-rational regular differentials on $A$, $\chi_{ij}(A,\pol,\Theta_A,w_A)$ is an affine lift of $\theta_i^A(0) \cdot \theta_j^A(0)$.
This allows the following definition.

\begin{definition} \label{def:modularlift}
Let $(A,\pol,\Theta_{\pol}) \in \Agn(k)$ and $w_A$ a basis of regular differentials on $A$. A \emph{modular lift}, denoted
$\theta^A(0,\sqrt{w_A})=(\theta_i^A(0, \sqrt{w_A}))_{i \in \Zn}$, is an  affine lift of $\theta^A(0)$ 
such that for all $i,j \in \Zn$, $\theta_i^A(0, \sqrt{w_A}) \cdot \theta_j^A(0,\sqrt{w_A)} = \chi_{ij}(A,\pol,\Theta_{\pol},w_A)$. Notice that the modular lift is unique up to a common sign.

\end{definition}

\begin{remark}
We  consider the two by two products because they give modular
forms of weight one. The $\theta^A(0, \sqrt{w_A})$ themselves would be modular forms of
weight one half. But the line bundle $\pol_{\Agn}$ does not descend on
$\Ag$, only to a $\mu_2$-gerbe of $\Ag$~\cite{candelori}.
Since we only need to compute modular
forms of integral weight, this {\it ad hoc} definition is
sufficient and requires less abstract material. Notice also that as a consequence of
\cite[p.~82]{MumfordOEDAV2} and \cite[Th. 4.2.1]{candelori}, $\pol^2_{\Agn} \iso \Hodge$, which gives another purely algebraic proof of the modularity of $s^* (\theta_i^{\Xgn} \cdot \theta_j^{\Xgn})$.
In particular, a choice of basis of regular differentials gives a trivialization of
$\Hodge$, so a trivialization of $\pol^2_{\Agn}$ and corresponding affine lifts
for the $\chi_{ij}$.
\end{remark}

If we start with a principally polarized abelian variety $(A,\pol)$ over a field
$k$ with a $k$-rational basis of regular differentials $w_A$, we may need to go to
an extension to build the level $n$ structure $\Theta_{\pol}$ on $A$. Hence the  
$\theta_i^A(0, \sqrt{w_A})$ are not necessarily defined over $k$. However, consider a Siegel modular form
$\chi$ of level $1$ and of integral weight $\rho \geq 1$, written as a homogeneous polynomial $P$ of degree $2\rho$ in the theta constants of level $\Theta_{\pol}$ and with coefficients in $k$. As $2\rho$ is even, we can express $P$ as polynomial $Q$ in pairs of theta constants, and therefore  $P(\theta^A(0,\sqrt{w_A})) = Q((\chi_{ij}(A,\pol,\Theta_{\pol},w_A))) = \chi(A,\pol,w_A) \in k$.
This is important for our
application to the modular form $\chi_{18}$ in dimension $g=3$ (see  \cref{sec:serreobs}).

\begin{theorem} \label{th:modular}
Let $(A,\pol,\Theta_{\pol}) \in \Agn(k)$. Let $\ell$ be an integer prime to $np$ (or to $n$ if $p=0$). Let $K$ be a  $k$-rational totally isotropic subgroup for the Weil pairing of
$\pol^\ell$. Let $f: (A, \pol, \Theta_\pol) \to (B, \bpol, \Theta_\bpol)$ where
$B=A/K$, $f^\ast \bpol=\pol^\ell$ and $\Theta_\bpol$  be the unique 
symmetric theta structure of level $n$ on $\bpol$ compatible with
$\Theta_\pol$. Let $w_A$ be a basis of $k$-rational regular differentials on $A$ and $(\theta_i^A(0, \sqrt{w_A}))_{i \in \Zn}$ be a modular lift. Finally, let $r=1,2$ or $4$ depending on $\ell$ being a square, a sum of two square or not. Then the affine isogeny formula 
$\pi_3 \circ \pi_2 \circ \Pi_1^{-1}$
yields the products
  $(\theta_{i_1}^B(0, \sqrt{w_B})\times \cdots \times \theta_{i_r}^B(0, \sqrt{w_B}))_{i_1,\ldots,i_r \in \Zn}$ where $w_B$ is such that $f^\ast w_B=w_A$. Note that the product is uniquely defined except if $r=1$ in which case we get all constants up to a common sign.
\end{theorem}
\begin{proof}

  Using the results of \cref{sec:universal}, the statement of this
  theorem makes sense over $\ZZ[\frac{1}{n \ell}]$. We will thus prove this
  theorem for $\Xgn \to \Agn$ over $\ZZ[\frac{1}{n \ell}]$, the result will then be valid
  for any field of characteristic prime to $n \ell$.

  We note that the theta coordinates computed by the isogeny formula
  give sections of the very ample line bundle $\pol_{\Agrn}$ of $\Agrn$ over
  $B^r$. Thus the $s_i$ can also be interpreted as sections of $\pol^r_{\Agn}$
  over $B$. We are thus trying to prove the equality of two sections of $\pol_{\Agn}^r$, i.e. that for any $i_1, \ldots, i_r$ the corresponding
  theta null point of coordinates $(i_1, \ldots, i_r)$ computed by the
  isogeny formula is equal to
$(\theta_{i_1}^B(0, \sqrt{w_B})\cdots  \theta_{i_r}^B(0, \sqrt{w_B}))$. 

  Since $\Agn$ is smooth, $\pol_{\Agn}$ is without torsion, so we only need
  to check this equality over $\CC$. The abelian variety $A/\CC$ is
  isomorphic to a torus $A \iso \CC^g/(\ZZ^g \oplus \Omega\ZZ^g)$. First it
  is easy to check that if we change our affine lift by multiplying it by
  $\lambda \in \CC$, then the result of the isogeny formula is multiplied
  by $\lambda^r$. Indeed in Step~1 (in affine coordinates), the affine lift
  of the points of $K$ are normalized with respect  to the affine lifts of
  the theta null point. Multiplying the theta null point by $\lambda$
  multiply the points $Q \in \Pi_1^{-1}(A,K)$ by $\lambda$. Then applying
  the Segre embedding multiply the theta null point by $\lambda^r$, and
  Koizumi's formula does not change this constant.

  Changing the basis of regular differentials by a matrix $M \in \GL_g(\CC)$ changes the value of a modular lift by $\lambda=\sqrt{\det(M)}$ for a fixed choice of the square root, since their pair products are weight $1$ modular forms.
  This changes both the modular forms $(\theta_{i_1}^B(0, \sqrt{w_B}) \cdots  \theta_{i_r}^B(0, \sqrt{w_B}))$ and the result of the isogeny formula by a factor $\lambda^r$.
  So we may fix the differentials on $A$ to be $w_A=(2i \pi dz_1, \ldots, 2 i \pi dz_g)$ of $\CC^g$.

  By \cref{eq:chiij}, the corresponding modular lift of the theta null point on $A$ is
  then given by the analytic theta constants
  $\theta^A_i=\thetacarac{0}{i/n}(0, \Omega/n)$ (where we do a slight abuse
  of notations in identifying $i \in \Zn$ to a fixed lift to $\ZZ^g$).

  We can then keep track of the constants in each of the three steps of the
  isogeny formula of Section~\ref{sec:universal}.

{\bf  Step~1:} we compute an affine lift of a theta null point of level $\ell n$  on
  $B$, such that the isogeny theorem applied to $\tilde{f}$ gives our theta
  null point on $A$. From our hypothesis, $K$ corresponds to the subgroup
  $\frac{1}{\ell} \ZZ^g / \ZZ^g$, so $B= \CC^g/(\ZZ^g \oplus \ell\Omega\ZZ^g)$ and $f: z \mapsto \ell
  z$. The contragredient isogeny $\tilde{f}: B \to A$ is then given by
  $\tilde{f}: B \to A, z \mapsto z$.
  So we see that one possible lift for the theta null point of level $\ell
  n$ on $B$ is given by $\thetacarac{0}{\frac{i}{\ell n}}(0, \frac{\ell \Omega}{\ell n})$.
  By plugging any $i$ divisible by $\ell$ we see
  that the constant involved in Step 1 is $1$.
  Indeed, the isogeny theorem (the pullback $\tilde{f}$ of $\pi_1$ to $B$) is simply given in
  terms of analytic theta coordinates by
  $$\left(\thetacarac{0}{\frac{i}{\ell n}}(z, \frac{\ell \Omega}{\ell
  n})\right)_{i \in \Zln}
    \mapsto 
\left(\thetacarac{0}{\frac{i}{\ell n}}(z, \frac{\Omega}{n})\right)_{i \in \Zln, \ell \mid i} =
\left(\thetacarac{0}{\frac{i}{n}}(z, \frac{\Omega}{n})\right)_{i \in \Zn}.$$
  Algebraically, this
  means that we are computing
  $(\theta_i^{B,\Mpol^\ell}(0, \sqrt{w'_B}))_{i \in \Zln}$  where $w'_B$ is such that
  $\tilde{f}^\ast w_A=w'_B$. By definition of the contragredient isogeny,
  we have that $w'_B=w_B/\ell$ (as seen analytically by the fact that the
  map $f$ above acts by $\ell$ on the tangent space).

{\bf  Step~2:} the Segre embedding simply consists on taking the sections induced by the basis of regular differentials on $B^r$
given  by the pullbacks of the differentials $w'_B$  by the projections on each factor. 
 Notice that the theta constants on $B^r$ are then easily related to the ones on $B$ since
  $\thetacarac{0 \, 0}{b_1 b_2}\left(0, \ell \textrm{diag}( \Omega,
  \Omega)\right) = \thetacarac{0}{b_1}(0, \ell \Omega) \thetacarac{0}{b_2}(0,
 \ell \Omega)$. %

 {\bf Step~3:} For this step, we need a version of Equation~\eqref{eq:koizumi} taking into account the possible multiplicative constant. This is given for instance in \cite[Th\'eor\`eme~7.2.1]{cosset2011thesis}
 \begin{equation}
   c \cdot \thetacarac{0}{i_1}(Y_1,\ell\Omega/n)\cdots \thetacarac{0}{i_r}(Y_r,\ell\Omega/n)=
   \sum_{\mathclap{[t_1, \ldots, t_r] \in \Mat_{r\times
   g}(\ZZ)F^{-1}/\Mat_{r \times g}(\ZZ)}}
         \thetacarac{0}{j_1}(X_1+t_1,\Omega/n) \cdots \thetacarac{0}{j_r}(X_r+t_r,\Omega/n),
   \label{eq@changement@niveau@gen0}
 \end{equation}
   where $F \in M_r(\ZZ)$ is such that ${^t F}F=\ell \textrm{Id}$, $Y$ in~$(\CC^g)^r$, $X=YF^{-1} \in (\CC^g)^r$, $i \in \QQ^r$,  $j=iF^{-1}$ and $$c=[\Mat_{r\times g}(\ZZ)F^{-1}: \Mat_{r \times g}(\ZZ)]=
  [\Mat_{r \times g}(\ZZ):\Mat_{r \times g}(\ZZ)F] =
  \ell^{gr/2}.$$

 Taking into account that
  $F^{-1}=\frac{1}{\ell}\transp{F}$, that the kernel of $F$ in $\Zl^r$
  is exactly the image of $\transp{F}$, and taking $Y_i=0$, we can
  rewrite \cref{eq@changement@niveau@gen0} in terms of modular lifts
  $$ c \cdot \theta_{i_1}^{B, \Mpol}(0,\sqrt{w_B'}) \cdots \theta_{i_r}^{B, \Mpol}(0,\sqrt{w_B'}) = 
  \sum_{\substack{(j_1, \ldots, j_r) \in {\Zln}^r \\
  F(j_1,\ldots,j_r)=(i_1,\ldots, i_r)}}
  \theta_{j_1}^{B, \Mpol^\ell}(0,\sqrt{w'_B}) \cdots \theta_{j_r}^{B, \Mpol^\ell}(0,\sqrt{w'_B}).$$
  Since $w_B'=w_B/\ell$ we have $\theta^B(0,\sqrt{w_B'}) = \ell^{-1/2} \cdot \theta^B(0,\sqrt{w_B})$. This kills the constant $c$ and we get the result (up to a fixed sign if $r=1$ because there is no way to choose a canonical square root of $\ell$ in a field $k$ in general).%
\end{proof}

This theorem shows that, given a Siegel modular form $\chi$ of \emph{even} weight as a polynomial $P$ in the theta constants with coefficients in $k$, we can compute the value $\chi(B,\bpol,\Theta_B,w_B)$ from the corresponding modular lift on $(A,\pol)$. In practice \cite{DRAVIsogenies} does not compute all products
$(\theta_{i_1}^B(0, \sqrt{w_B})\cdot \dots \cdot \theta_{i_r}^B(0, \sqrt{w_B}))_{i \in \Zn}$ but only the products
$t_i := (\theta_{i}^B(0, \sqrt{w_B})\cdot \theta_{0}^B(0, \sqrt{w_B}) \dots \cdot \theta_{0}^B(0, \sqrt{w_B}))_{i \in \Zn}$, since this is enough for isogenies. It is also enough in our case: the weight being even means that each monomials of $P$ in the theta constants has a degree multiple of $4$ (and hence of $r$). We then get 
$$\chi(B,\bpol,\Theta_{\bpol},w_B) = P(\theta_i^B(0,\sqrt{w_B})) = t_0^{-\frac{(r-1) \rho}{r}} \cdot P(t_i).$$

The modular forms we will consider are written as polynomials in the theta constants with half characteristics and not in the algebraic theta of level~$4$. However it is easy to convert one into the other: see Remark~\ref{rem:alg2ana}

\subsection{An algebraic version of Thomae's formula}
\label{sec:thomae}

If $E:y^2=F(x)$ is an elliptic curve defined over $k$, we would like to
compute the modular lift of the theta null point of level~$4$ with respect to the
$k$-rational differential $w=dx/y$. Over $k \subset \CC$, the expression of the fourth powers of theta constants can be seen as an elementary case of Thomae's formula \cite[p.121]{mumford-tata2} for hyperelliptic curves (although a sign remains unspecified). For dimension 1, one could also use $\sigma$ functions as in \cite[p.55]{akhiezer}, but one still only gets expression for the fourth powers of the theta constants. We will reprove these formulas in the following lemma and show that one can take arbitrary fourth roots. This will be useful for the computation of Siegel modular forms of even weight at $(B,\bpol,w_B)$ in the isogeny class of $E^g$. 

\begin{lemma}[Analytic form of Thomae's formula]\label{lem:fourth}
Let $E$ be an elliptic curve with Weierstrass equation $y^2=F(x)$ defined over over $\CC$.
Let $e_1,e_2,e_3$ be the roots of $F$. Fix arbitrarily three fourth roots $a_1,a_2,a_3$ of $e_i-e_j$ for $(i,j) \in ((2,3),(1,2),(1,3))$.
There exists a basis $\delta_1,\delta_2$ of $H_1(E,\ZZ)$ such that if we denote $[\omega_1,\omega_2]=[\int_{\delta_1} dx/y,\int_{\delta_2} dx/y]$ then $\tau= \omega_2/\omega_1 \in \HHb_1$ and
$$\sqrt{c} \cdot \thetacarac{0}{0}(\tau) =  a_3, \quad \sqrt{c} \cdot \thetacarac{1/2}{0}(\tau) =  a_2, \quad \sqrt{c} \cdot \thetacarac{0}{1/2}(\tau) =  a_1$$
with $c=\frac{2 i \pi}{w_1}$ for an arbitrary fixed square root of $c$.
\end{lemma}
\begin{proof}
Let $\tau \in \HHb_1$ and denote $$\vartheta_{00}(z) = \thetacarac{0}{0}(z,\tau),\; \vartheta_{10}(z) = \thetacarac{1/2}{0}(z,\tau), \; \vartheta_{01}(z) = \thetacarac{0}{1/2}(z,\tau),$$ and $\vartheta_{11}(z) = \thetacarac{1/2}{1/2}(z,\tau).$ When $z$ does not appear, it denotes the corresponding value at $z=0$.
 As in \cite[p.125]{farkaskra}, let us consider the map $\phi : \CC \to \PP^2$ given by $$(\vartheta_{00}^2(z) \vartheta_{11}(z) : \vartheta_{00}(z) \vartheta_{01}(z) \vartheta_{10}(z) : \vartheta_{11}^3(z)).$$ 
Using the divisors of these sections, one can prove that the image by $\phi$ of $\CC/(\ZZ+ \tau \ZZ)$ is the elliptic curve
$$E_2: Y_2^2 Z_2= X_2 (\beta X_2- \alpha Z_2) (\alpha X_2 + \beta Z_2)$$
where $\alpha=\vartheta_{10}^2{\vartheta_{00}^2}$ and $\beta=\vartheta_{01}^2{\vartheta_{00}^2}$. Letting $Y_2=Y_1 \vartheta_{10} \vartheta_{01}/\vartheta_{00}^2, X_2=X_1$ and $Z_2=Z_1$, we can transform further in
$$E_1: Y_1^2Z_1 = X_1 (X_1- \alpha/\beta Z_1) (X_1 + \beta/\alpha Z_1).$$
Then letting $Z_1=(\vartheta_{01}^2 \vartheta_{10}^2) Z_0 $ and finally $Y_1=Y_0/(\vartheta_{01} \vartheta_{10})$ and $X_1=X_0$ one gets
$$E_0 : Y_0^2Z_0 = X_0 (X_0- \vartheta_{10}^4(0) Z_0) (X_0 + \vartheta_{01}^4(0) Z_0)$$
Let us study the regular differential $w_0=d(X_0/Z_0)/(Y_0/Z_0)=\frac{1}{\vartheta_{00}^2} \cdot d(X_2/Z_2)/(Y_2/Z_2)$ on $E_0$. Since $w_2=d(X_2/Z_2)/(Y_2/Z_2)$ is regular, $\phi^*(w_2)$  is a constant multiple of $dz$.  Now
\begin{align*}
\phi^* w_2 &= 2 \cdot \frac{\vartheta_{00}(z)' \vartheta_{11}(z) - \vartheta_{11}(z)' \vartheta_{00}(z)}{\vartheta_{10}(z) \vartheta_{01}(z)} \\
&= -2 \frac{\vartheta_{11}'(0) \vartheta_{00}(0)}{ \vartheta_{10}(0) \vartheta_{01}(0)}  & (\textrm{evaluating at } z=0) \\
&= 2 \pi \frac{\vartheta_{00} \vartheta_{10} \vartheta_{01} \vartheta_{00}}{\vartheta_{10} \vartheta_{01}}  & (\textrm{Jacobi identity } \vartheta_{11}' = - \pi \vartheta_{00} \vartheta_{01} \vartheta_{10})\\
&= 2 \pi \vartheta_{00}^2.
\end{align*}
Hence if $\psi_0 : \CC/( \ZZ+\tau \ZZ) \to  E_0$ is the isomorphism composed from $\phi$ and the changes of variables we get that 
$\psi_0^* w_0=2 \pi dz$ (notice that this is not the natural $2 i \pi dz$ we chose before but we will take care of the extra factor $i$ when we choose the fourth root).

Now, let us start with $E: y^2=F(x)$. If we make the change of variable $X=x-e_2$, $Y=y$, then we get $E' : Y^2 = X (X-(e_1-e_2)Z)(X+(e_2-e_3)Z)$.
 If we integrate $w=d(X/Z)/(Y/Z)$ along a basis of the homology of $E'$, we get a torus $\CC/(\omega_1 \ZZ+ \omega_2 \ZZ)$ and up to a change of the order in the basis, we can assume that  $\tau=\omega_2/\omega_1 \in \HHb_1$ and $\psi : \CC/(\omega_1 \ZZ+ \omega_2 \ZZ) \to E'$ an 
 isomorphism such that $\psi^* w=dz$. Let $s : \CC/(\omega_1 \ZZ+ \omega_2 \ZZ) \overset{\sim}{\to} \CC/(\ZZ+\tau \ZZ)$ such that $z \mapsto z/\omega_1$. The composition 
$$\xymatrix{  \CC/(\omega_1 \ZZ+ \omega_2 \ZZ)  \ar[d]_{s} & &E'   \ar[ll]_{\psi^{-1}} \ar[d]^{\mu}  \\ \CC/( \ZZ+ \tau \ZZ)  \ar[rr]_{\psi_0}& & E_0 
}$$
defines an isomorphism $\mu : E' \to E_0$ such that $(X:Y:Z) \to (a^2 X : a^3 Y : Z)$ with $a \in \CC^*$.  After a possible change in the generators of the homology of $E'$ (by a lift to $\SL_2(\ZZ)$ of a change of basis of $E'[2]$), we can even 
assume that $\mu$ maps the roots $0$ to $0$, $e_1-e_2$ to  $\vartheta_{10}^4/a^2$ and $e_2-e_3$ to $\vartheta_{01}^4/a^2$. Note that $ e_1-e_3=(\vartheta_{10}^4+\vartheta_{01}^4)/a^2=\vartheta_{00}^4/a^2.$
Now $\mu^* w_0 = w/a = (\psi^{-1})^* \circ s^* \circ \psi_0^* \, w_0 = 2 \pi/\omega_1 \cdot w$. Hence $a= \omega_1/2 \pi$. This means that we have the equalities
$$a_2^4= e_1-e_2 = -c^2  \vartheta_{10}^4, \quad a_1^4=e_2-e_3= -c^2 \vartheta_{01}, \quad a_3^4=e_1-e_3 = - c^2 \vartheta_{00}.$$

To conclude, we must show that we can choose the basis of homology for $E$ in order to choose the fourth root of unity arbitrarily and get the correct result up to a common fourth root of unity. As the two-torsion points are now fixed, this boils down to find some matrices in $\SL_2(\ZZ)$ which are congruent to the identity modulo $2$. If we call  $S=\begin{pmatrix} 
1 & 1 \\
0 & 1 
\end{pmatrix}$ and $T =\begin{pmatrix} 
0 & 1 \\
-1 & 0 
\end{pmatrix}$, let $H=<S^2,T^2,(ST)^3,(STS)^2>$
and
$(\alpha_1,\alpha_2,\alpha_3)=(i \sqrt{c}\vartheta_{01},i \sqrt{c}\vartheta_{10},i \sqrt{c}\vartheta_{00})$. Notice that the $\alpha_i$s do depend on $\tau$ but also on $\omega_1$. The actions of $S$ and $T$ on the lattice induce  actions on the $\alpha_i$ which can be computed through the classical transformation formula \cite[Th.7.1]{mumford-tata1}. Namely
\[
\left\{
\begin{array}{r c l}
S. \alpha_1 &=& \alpha_3, \\
S. \alpha_2 &=& e^{i\pi/4}\alpha_2,\\
S. \alpha_3 &=& \alpha_1,
\end{array}
\right.
\textrm{and \;}
\left\{
\begin{array}{r c l}
T. \alpha_1 &=& \sqrt{-i}\alpha_2, \\
T. \alpha_2 &=& \sqrt{-i}\alpha_1,\\
T. \alpha_3 &=& \sqrt{-i}\alpha_3.
\end{array}
\right.
\]
Hence we get
$$\left\{
\begin{array}{r c l}
S^2. \alpha_1 &=& \alpha_1, \\
S^2. \alpha_2 &=& i\alpha_2,\\
S^2. \alpha_3 &=& \alpha_3,
\end{array}
\right.
\textrm{, \;}
\left\{
\begin{array}{r c l}
T^2. \alpha_1 &=& -i\alpha_1, \\
T^2. \alpha_2 &=& -i\alpha_2,\\
T^2. \alpha_3 &=& -i\alpha_3,
\end{array}
\right.
$$
and
$$\left\{
\begin{array}{r c l}
(ST)^3.\alpha_1 &=& i\alpha_1, \\
(ST)^3.\alpha_2 &=& i\alpha_2,\\
(ST)^3.\alpha_3 &=& i\alpha_3,
\end{array}
\right.
\textrm{, \;} 
\left\{
\begin{array}{r c l}
(STS)^2.\alpha_1 &=& -i\alpha_1, \\
(STS)^2.\alpha_2 &=& -\alpha_2,\\
(STS)^2.\alpha_3 &=& -\alpha_3.
\end{array}
\right.
$$
 The group $\mu_4^3$ has generators $u_1:=(i,1,1),u_2:=(1,i,1),u_3:=(1,1,i)$. The expressions above show that $g_1=(ST)^3 (STS)^2$ (resp. $g_2=S^2$, resp. $g_3=g_1^3g_2^3 (ST)^3$) acts  on $(\alpha_1,\alpha_2,\alpha_3)$ as $u_1$ (resp. $u_2$, resp. $u_3$). Starting from $(a_1,a_2,a_3)$ it is therefore possible to find a $\tau$ such that $(a_1,a_2,a_3) = (\sqrt{c} \vartheta_{01}, \sqrt{c} \vartheta_{10}, \sqrt{c} \vartheta_{00})$.
\end{proof}

\begin{remark} \label{rem:alg2ana}
 The algebraic theta functions of level~$4$, $(\theta_1, \theta_2,
\theta_3, \theta_4)$ analytically correspond to the theta functions
$(\thetacarac{0}{i/4}(z, \Omega/4))_{i \in \ZZ/4\ZZ}$. Going to these
functions from the standard level~$(2,2)$ analytic theta
$\thetacarac{a/2}{b/z}(2z, \Omega)$ is given by a change of variables
\cite{mumford-tata1}, \cite[p. 38]{cosset2011thesis}
\begin{equation}\label{eq:linchange}
  \begin{aligned}
  \theta_0(z)=\thetacarac{0}{0}(z,\Omega)+\thetacarac{1/2}{0}(z,\Omega),&\quad
   \theta_1(z)=\thetacarac{0}{1/2}(z,\Omega)+\thetacarac{1/2}{1/2}(z,\Omega),\\
  \theta_2(z)=\thetacarac{0}{0}(z,\Omega)-\thetacarac{1/2}{0}(z,\Omega),&\quad
   \theta_3(z)=\thetacarac{0}{1/2}(z,\Omega)-\thetacarac{1/2}{1/2}(z,\Omega),
  \end{aligned}
\end{equation}
where $\theta_i(z) = \thetacarac{0}{i/4}(z, \Omega/4)$.

The functions $\thetacarac{a/2}{b/2}(2z, \Omega)$ also have algebraic
analogues as  partial Fourier transforms over $\Ztwo$ of the functions $\theta_i$
as explained in \cite[p.~334]{MumfordOEDAV1} and \cite[Exemple~4.4.9]{DRphd}.
If $\theta_i$ is a theta function of level~$n$, the partial Fourier
transform is given for $\alpha \in \dZtwo$ by
\begin{equation}
\thetacaraco{\alpha}{i}=\sum_{j \in \Ztwo} \alpha(j) \theta_{i+j}.
\label{eq:theta22}
\end{equation}
Analytically, $\thetacaraco{\alpha}{i}(z)=\thetacarac{\alpha/2}{2i/n}(2z,2\Omega/n)$,
so if $n=4$ we do recover the theta functions of level $(2,2)$.\\
\end{remark}

All these expressions for the theta constants over $\CC$ are true over $k$. Indeed, pairing them will give modular forms with integral Fourier expansion, so we get similar expression for the modular lift, up to a common sign which can be swallowed in the choice of the fourth root.

\begin{corollary}[Algebraic form of Thomae's formula] \label{cor:algebraicthomae}
Let $E$  be an elliptic curve with Weierstrass equation $y^2=F(x)$ defined over a field $k$ of characteristic $p \ne 2$.
Let $e_1,e_2,e_3$ be the roots of $F$ in $\bar{k}$. Fix arbitrarily three fourth roots $a_1,a_2,a_3$ of $e_i-e_j$ for $(i,j) \in ((2,3),(1,2),(1,3))$.
Then there is a level~$4$ symmetric theta structure on $E$, such that a modular lift of the theta null point on $E$ with respect to the regular differential $dx/y$ is
\begin{equation}
  \begin{aligned}
  \theta^E_0(0_E, \sqrt{dx/y})&=a_2+a_3, \quad&\theta^E_1(0_E, \sqrt{dx/y})&=a_1,\\
  \theta^E_2(0_E, \sqrt{dx/y})&=-a_2+a_3, \quad&\theta^E_3(0_E, \sqrt{dx/y})&=a_1.
  \end{aligned}
\label{eq:thomae}
\end{equation}
\label{cor:thomae}
\end{corollary}
\begin{proof}
  Define
$\thetacaraco{0}{0}(0_E) = a_3, \thetacaraco{1/2}{0}(0_E) = a_2, \thetacaraco{0}{1/2}(0_E) = a_1$.
  First we note that the first part of \cref{lem:fourth}
  is valid algebraically: we just need to replace the argument involving divisors by the algebraic Riemann relations instead.
  Indeed it is easy to check that the theta null point defined satisfy the Riemann relation $\thetacaraco{0}{0}(0_E)^4 = \thetacaraco{1/2}{0}(0_E)^4 + \thetacaraco{0}{1/2}(0_E)^4$ (this is the standard Jacobi relation to which Riemann relations reduce to in genus~$1$~\cite[p. 353]{MumfordOEDAV1}).
  Since we also have that $\thetacaraco{0}{0}(0_E)\thetacaraco{1/2}{0}(0_E)\thetacaraco{0}{1/2}(0_E)=a_1a_2a_3 \ne 0$, the theta null point we compute is valid projectively by \cite[p. 353]{MumfordOEDAV1}.
  This also proves that each choice of fourth root is valid.\footnote{Alternatively, the affine modular action of $\Gamma/\Gamma(4,8)$ induces a
  projective action \cite[Lemme~6.2.1]{cosset2011thesis} which holds true algebraically, as automorphisms of the
  Heisenberg group of level~$4$. So the same generators $g_1, g_2$ and
  $g_3$ as in the end of \cref{lem:fourth} acts by fourth-root of unity
projectively.}

  It remains to
  check that the affine lift given by \cref{eq:thomae} corresponds to the
  trivialization coming from the differential $w=dx/y$. Since the
  construction is valid over the universal elliptic curve with a level~$4$
  symmetric theta structure, whose moduli space is defined over $\ZZ[1/2]$,
  by considering the pullback to $\CC$ we may assume that $E$ is defined
  over $\CC$, as in the proof of \cref{th:modular}.
  Looking at the proof of \cref{lem:fourth},  we see that the
  isomorphism between $E$ and $E'$ does not change the differential $w$, while the one from
  $E'$ to $E_0$ acts by $a=2\pi/\omega_1$. Correcting for this last factor
  yields
\begin{equation}
\thetacaraco{0}{0}(0_E, \sqrt{dx/y}) = a_3, \quad \thetacaraco{1/2}{0}(0_E, \sqrt{dx/y}) = a_2, \quad \thetacaraco{0}{1/2}(0_E, \sqrt{dx/y}) = a_1.
\label{eq:thomae_an}
\end{equation}
Applying the linear change of variable \cref{eq:linchange} to
\cref{eq:thomae_an} yields \cref{eq:thomae}.
\end{proof}

\subsection{Computing a Siegel modular form on the isogenous variety}
\label{sec:siegelisogeny}
Combining Corollary~\ref{cor:thomae} with  Theorems~\ref{th:isogeny_formula} and Theorem~\ref{th:modular} gives the following theorem and Algorithm~\ref{algo:modular}.

\begin{theorem}
  \label{th:siegelmodular}
Let $g$ be a positive integer, $(E_i/k)_{1\leq i \leq g}$ be elliptic curves, $K$ be a $k$-rational totally isotropic subgroup of $\prod_i E_i$  of order $\ell^g$ prime to $2 p$ (or just prime to $2$ if $p=0$). Let $B=(E_1 \times \cdots \times E_g)/K$ with the principal polarization induced by the product polarization on $E_1 \times \cdots \times E_g$ and let $f: \prod_i E_i \to B$ be the quotient isogeny. 
Finally define $w_B$ such that $f^\ast w_B=(p_1^* dx_1/y_1, \ldots, p_g^* dx_g/y_g)$ where $p_i : E_1 \times \cdots \times E_g \to  E_i$ is the canonical projection. Let $r=1,2$ or $4$ depending on $\ell$ being a square, a sum of two squares or not.
  Algorithm~\ref{algo:modular} computes the
  products  $\theta_{i_1}^B(0,\sqrt{w_B}) \cdots \theta_{i_r}^B(0,\sqrt{w_B})$ of any $r$ modular lifts in time $O(\ell^{g \max(1,r/2)})$ operations in the
  field of definition of the points of $K$.
 Given a Siegel modular form $\chi$ of even weight as a polynomial $P$ in the theta constants with coefficients in $k$, Algorithm~\ref{algo:modular} also computes the value $\chi(B,\bpol,w_B) \in k$. 
  \label{th:finalalgo}
\end{theorem}

\begin{remark}
We can make several comments about this result.
\begin{itemize}
  \item
  Note that during the execution of the algorithm, we only need to take care to compute the modular lift of the theta null point. Indeed, apart from the theta null point, we only need to compute projective coordinates for the points in the kernel, the computation of $\Pi_1^{-1}$ will take care of normalizing these coordinates with respect to our choice of affine lift of the theta null point.
  \item We only require $\chi$ to be of even weight $w$ if $r=4$. Otherwise
    given the $r$-fold products
    $$\theta_{i_1}^B(0,\sqrt{w_B}) \cdots \theta_{i_r}^B(0,\sqrt{w_B})$$
    we can evaluate a modular form of odd weight.
  \item We do not need to evaluate all the $r$-fold products, but only the
    ones of the form
    \[t_i= \theta_{i}^B(0,\sqrt{w_B}) \cdots \theta_{i}^B(0,\sqrt{w_B})\]
    (provided $\theta_{0}^B(0,\sqrt{w_B}) \neq 0$).
    If $\chi$ is of weight $w$, it can then be evaluated as
    $\chi(B,\bpol,w_B)=P(t_i)/t_0^{w (r-1)/2}$.
  \item If the modular form $\chi$ that can be written as a
    polynomial with respect to the level~$2$ theta constants, we can do the
    whole isogeny computation in level~$2$. This gains a factor $2^g$ in
    the number of coordinates to compute.
\end{itemize}
\end{remark}

\begin{algorithm}[H]
  \label{algo:thetaEg}
\begin{algorithmic}[1]
  \Require{Elliptic curves $E_i/k$   with equation  $y^2=(x-e_{1i})(x-e_{2i})(x-e_{3i})$ where $k$ is of characteristic $p$ different from $2$, a $k$-rational totally isotropic subgroup $K$ of $A=\prod_i E_i$ of order $\ell^g$ prime to $2 p$ (or just prime to $2$ if $p=0$). A Siegel modular form $\chi$ of even weight as a polynomial $P$ in the theta constants with coefficients in $k$.}
 \Ensure{The theta null point of level $4$ and the value $\chi(B,\bpol,w_B)$ where $B=A/K$ with $\bpol$ the polarization induced by the product polarization on $A$ and $w_B$ such that $f^*w_B=(p_1^* dx_1/y_1, \ldots, p_g^* dx_g/y_g)$ where $f : A \to B$ is the quotient isogeny and  $p_i : E_1 \times \cdots \times E_g \to  E_i$ is the canonical projection.}
\State{For all $1 \leq i \leq g$, define $\theta'^{E_i}_{0}=\sqrt[4]{e_{1i}-e_{3i}},  \theta'^{E_i}_1= \sqrt[4]{e_{1i}-e_{2i}}, \theta'^{E_i}_2 =\sqrt[4]{e_{2i}-e_{3i}}$ for arbitrary choices of the roots.}
\State{Compute $\theta^{E_i}_{0}(0,\sqrt{dx_i/y_i})=\theta'^{E_i}_{0} +\theta'^{E_i}_{1}$, $\theta^{E_i}_{2}(0,\sqrt{dx_i/y_i})=\theta'^{E_i}_{0} -\theta'^{E_i}_{1}$ and $\theta^{E_i}_{1}(0,\sqrt{dx_i/y_i})=\theta^{E_i}_{3}(0,\sqrt{dx_i/y_i})=\theta'^{E_i}_{2}$ for all $1 \leq i \leq g$.}
\State{Compute all $\theta^A_{(i_1, \ldots, i_g)}(0,\sqrt{w_A})=\theta^{E_1}_{i_1}(0,\sqrt{dx_1/y_1}) \cdots \theta^{E_g}_{i_g}(0,\sqrt{dx_n/y_n})$ for all $(i_1,\ldots,i_g) \in \Zfour$.}
  \State{For all $1 \leq i \leq g$ and for all  $x=(x_1,\ldots,x_g) \in K\setminus \{0\}$, compute the theta coordinates $\left(\theta^{E_i}_j(x_i)\right)_{j \in \ZZ/4\ZZ}$.}
  \State{Compute  for all $j=(j_1,\ldots,j_g) \in \Zfour$ and  for all $x=(x_1,\ldots,x_g) \in K \setminus \{0\}$
$\theta^{A}_j(x) = \theta^{E_1}_{j_1}(x_1) \cdots \theta^{E_g}_{j_g}(x_g)$.}
\State{Use the affine version of the isogeny formula to compute $t_i = \theta_i^B(0,\sqrt{w_B}) \cdot \theta_0^B(0,\sqrt{w_B}) \cdots  \theta_0^B(0,\sqrt{w_B})$ which is a product of $r$ factors with $r=1$ if $\ell$ is a square, $r=2$ if $\ell$ is the sum of two positive squares and $r=4$ otherwise.}\\
\Return{$(t_i)_{i \in \Zfour}$ and  $t_0^{-\frac{(r-1) \rho}{r}} \cdot P(t_i).$}
\end{algorithmic}
 \caption{Algebraic computation of the theta null point and a Siegel modular form of even weight} \label{algo:modular}
\end{algorithm}

%% file: sections/serreobs2.tex
\section{Application to defect-0 curves of genus at most 4}

\label{sec:ex}

 Let $C$ be a curve of genus $g>0$ over $\FF_q$ with $q=p^m$. The Hasse-Weil-Serre bound asserts that $\# C(\FF_q) \leq 1+q+ gm$ where 
  $m=\lfloor 2 \sqrt{q} \rfloor$. A curve which number of rational points reaches with bound is called a \emph{defect-$0$ curve}.
   When $g>2$, it is not known in general for a given field $\FF_q$ whether a defect-$0$ curve $C/\FF_q$ of genus $g$ exists. 
 If it does, $\Jac C$ is isogenous to the $g$-power of an elliptic curve $E$ with trace $-m$. In order to see if such a curve exists,
 we therefore start by enumerating the  indecomposable principally polarized abelian varieties $(A_i,\LL_i)$ of dimension $g$ in the isogeny class of $E^g$. When $m$ is prime to $q$ and hence $E$ is ordinary,
 we have seen in Section~\ref{sec:effectivekernel} how to describe all of them  as a quotients of $E^g$ by given maximal isotropic subgroups $K \subset E[\ell_1] \times \cdots \times E[\ell_g]$. When we can moreover choose $\ell=\ell_1=\ldots=\ell_g$ odd, prime to the characteristic of $\FF_q$ (see the condition in Theorem~\ref{ellid}) and $K$ totally isotropic, we can use Algorithm~\ref{algo:thetaquotient} to compute the theta null point of level $4$ for each $(A_i,\LL_i)$.
 
 Now, we need to single out the ones which are Jacobians of curves of genus $g$ over $\FF_q$. By \cite{oort-ueno}, we know that any indecomposable principally polarized abelian variety $(A,\LL)$ of dimension $g\leq 3$ is the Jacobian of a curve $C_0$ of genus $g$ over $\bar{\FF}_q$. When $g=4$, this is not the case, but we will be able to distinguished them computing a certain Siegel modular form using Algorithm~\ref{algo:modular}, see Section~\ref{sec:g4}. However if $(A,\LL)$ is a Jacobian of dimension $4$ over $\bar{\FF}_q$ there is currently no way to check if it is also a Jacobian over $\FF_q$.
As for $g \leq 3$, notice that there is a big difference between the genus $2$ and genus $3$ case when dealing with the existence of $C$ over $\FF_q$. For the genus $2$, this is automatic: the existence of an indecomposable principally polarized abelian surface over $\FF_q$ in the class of $E^2$ is enough to ensure the existence of the curve $C$.
 For genus $3$ curves though, there may be an arithmetic obstruction as we shall recall in Section~\ref{sec:serreobs}. As we shall see this obstruction can be computed from the value of a Siegel modular form. %

For $g=2$ or $3$, we can even get an equation for the curve $C$ when it exists. In genus $2$, the construction of such a curve from its theta null point is classical and we refer for instance to \cite{DRniveau} ; in genus $3$, the formulae depend on the curve being hyperelliptic or not, which can be distinguished by exactly one of the $36$ even theta coordinates being $0$ or none.  In the hyperelliptic case, one can use \cite{weng}\footnote{
\cite{BILV} noticed that there are some mistakes in this article of Weng and \cite[Appendix]{lario2020inverse} gives a correct implementation (see also
 \cite{CurveReconstructionCode}%
). However, we did not try to implement the reconstruction in the genus $3$ hyperelliptic case.}
 to construct first a model $C_1$ over $\bar{\FF}_q$. Then one computes Shioda invariants\footnote{or computes them directly from the theta constants using for instance \cite{garcia-shioda} and overpass the difficulties mentioned above.} and then reconstruct via \cite{lercierhyp} when $p > 7$. In the non-hyperelliptic case, one can use Weber's formulae (\cite[p.108]{weber}, see also \cite{Fiorentino}) to get first a curve $C_1$ over an extension $\FF_{q^e}$ of $\FF_q$ ($p \ne 2$). To get an equation of $C_0$ over $\FF_q$, we implemented an explicit Galois descent taking advantage of the fact that $C_1$, being given with its full level-$2$ structure, has all its bitangents defined over $\FF_{q^n}$. Hence, all isomorphisms between $C_1$ and its Galois conjugates over $\FF_e$ are defined over $\FF_{q^e}$ as well.

It may still be that $\Jac C_0$ is not isomorphic over $\FF_q$ to the chosen $(A,\LL)$ as $C_0$ may be a twist of the right curve $C$. If the geometric automorphism group of $C$ is trivial (which can be read from the automorphism group of the lattice), then the curve has no automorphism, hence no non-trivial twist and $C_0 \simeq C$. Otherwise, one has to compute the list of all twists: in the hyperelliptic case see \cite[Sec.4.6]{lercierhyp} (implemented in \texttt{Magma}), and in the non-hyperelliptic case see \cite[Sec.4]{LRRS14}.

To conclude, it is then enough to check among the twists which ones are defect-$0$ curves over $\FF_q$, which can be achieved through naive point counting algorithms. Hence for $g=2$ and $3$ our algorithms provide an explicit list of all isomorphism classes of defect-$0$ curves over $\FF_q$.
\begin{remark}
A different way to do so is to pick a random $\FF_q$-rational divisor $D \in \Jac C'(\FF_q)$, and check if $(1+q-\textrm{Trace}(E))^g D =0$. 
   A better way would be to select the right Galois descent directly by keeping track of the Galois action on the two torsion points of $E^g$ through the isogeny. This could actually be achieved since a more general isogeny formula exists which can also be applied to an arbitrary torsion point of $E^g$. We did not implement this method yet.
\end{remark}

 \subsection{Curves of genus 2}
 
Let us give some examples to illustrate our algorithms. We start with a very simple one.
\begin{example} \label{ex:genus2}
 Let $E/\FF_{61}: y^2=x^3 + 11x + 17$ be an elliptic curve such that $R:=\ZZ[\pi]=\ZZ[w]$ with $w=\frac{1+\sqrt{-19}}{2}$. When $g=2$, the algorithm developed in Section~\ref{sec:hermitianlattices} shows that there is only one  indecomposable unimodular positive definite $R$-lattice of rank 2, namely $R^2$ with the hermitian form  $h=\begin{bmatrix} 2 & -\bar{w} \\ -w & 3 \end{bmatrix}$ (this can alternatively be seen from Schiemann's tables \cite{SchiemannTables}). Hence $A=\F(R^2)=E^2$ with the polarization $\LL$ induced by $h$ is the only Jacobian inside the isogeny class of $E^2$. Using Algorithm~\ref{algo:kerneliso} one can check that there is a polarized isogeny $f$ from $A_0=E^2$ with the product polarization to $(A,\LL)$ with kernel $K \subset A[\ell]$ with $\ell=3$. Explicitly $K$ is generated by the two affine points of $E^2$
 $$ ((51 a^3 + 39 a^2 + 36 a + 13, 59 a^3 + 43 a^2 + 48 a + 35), (3 a^3 + 31 a^2 + 38 a + 4, 44 a^3 + 22 a^2 + 
    19 a + 11)),$$ 
 $$((58 a^3 + 30 a^2 + 23 a + 36, 14 a^3 + 55 a^2 + 47 a + 45), (51 a^3 + 39 a^2 + 36 a + 13, 
2 a^3 + 18 a^2 + 13 a + 26))$$
where $a \in \FF_{61^4}$ has minimal polynomial $x^4+3x^2+40x+2$.
 We can also compute the theta null point which we express in the classical basis of theta constants characteristics. For instance $\theta_{00}^B(0) = \thetacarac{0 0}{0 0}(0)$ is equal to 
 $$43 b^{11} + 34 b^{10} + 28 b^9 + 11 b^8 + 6 b^7 + 19 b^6 + 30 b^5 + 27 b^4 + 27 b^3 + b^2 + 30 b + 59$$
 where $b \in \FF_{61^{12}}$ with minimal polynomial $x^{12} + 2 x^8 + 42 x^7 + 33 x^6 + 8 x^5 + 38 x^4 + 14 x^3 + x^2 + 15 x + 2$.
  Using the reconstruction method explained above, we find $C: y^2 =  45 x^6 + 13 x^5 + 25 x^4 + 23 x^3 + 3 x^2 + 20 x + 13$. 

Consider the complex expression $\chi_{5}(\tau) = \prod_{\epsilon \; \textrm{even}} \vartheta[\epsilon](\tau)$. Then $\chi_{10}=\chi_{5}^2$ is a Siegel modular form of weight $10$ and level $\Gamma_2$ defined over $\ZZ$. Using Algorithm~\ref{algo:modular}, we find that 
$\chi_{10}(A,\LL,w_A) = 22$ where $w_A$ is the basis of differentials constructed in Theorem~\ref{th:finalalgo}. There is a well-known relation with between $\chi_{10}$ and the discriminant of $C : y^2=f(x)$ (which is $2^8$ times the discriminant of $f$) up to the choices of bases of regular differentials. One must have that 
$\chi_{10}(A,\LL,w_A) /(2^{12} \cdot \textrm{Disc}(C))$ is a 10th power of the determinant of the change of bases, hence a 10th power in $\FF_q$. This is indeed the case.
\end{example}

\begin{example}
In a similar way, we can work out an example over $k=\FF_{5^3}$ with a non-maximal order of discriminant $-2^4$. In that case there is a unique defect-$0$ curve of genus $2$ over $k$, namely
$C : y^2 = 3 x^6 + 3 x^4 + 3 x^2 + 3$.
\end{example} 

\begin{example}
Let us consider now the case $k=\FF_{271}$ with a non-maximal order of discriminant $-60$. In that case, there are 9 indecomposable principally polarized abelian surfaces in the isogeny class. For only two of them, there exists an odd $\ell$ ($\ell=5$) %
and one can write down the corresponding curves, namely
$y^2 = 65 x^6 + 167 x^5 + 63 x^4 + 49 x^3 + 
        63 x^2 + 167 x + 65$ and
$y^2 = 89 x^6 + 224 x^5 + 155 x^4 + 16 x^3 + 
        155 x^2 + 224 x + 89$.
 For the seven other cases, such an $\ell$ does not exist: Theorem~\ref{ellid} shows that either there is no orthogonal basis with the same odd norm for two of them, or no orthogonal basis with the same norm for the last $5$ of them.
\end{example}

 \subsection{Curves of genus 3} \label{sec:serreobs}

 In his lectures at Harvard in 1985, Serre found that a principally polarized abelian variety $(A,\pol)$ of dimension $g>2$ defined over a perfect field $k$, which is geometrically a Jacobian, is not necessarily a Jacobian over $k$ (unlike in dimension $1$ or $2$). The obstruction is given by a quadratic character of $\mathrm{Gal}(\bar{k}/k)$ and is called  \emph{Serre's obstruction}. This obstruction is always trivial for hyperelliptic curves.
 When $k \subset \mathbb{C}$ and $g=3$, this character can be computed in terms of the value of the modular form
 defined over $\CC$ by $\chi_{18}(\tau)= -\frac{1}{2^{28}} \cdot \prod_{\epsilon} \vartheta[\epsilon](\tau)$,  where the product is over the $36$ even theta constants (\cite{Ser}, \cite{LR}, \cite{meagher}, \cite{LRZ}). 
 Using lifting techniques, one can  thus get the obstruction  for certain $(A,\pol)$ when $k$ is a finite field of characteristic different from $2$ and therefore address the question of maximal number of points of genus $3$ curves (see for instance \cite{Rit}). However,  the numerical approximations during the computation of the value of the modular form lead to heuristic results only.
 
 The techniques developed in Section~\ref{sec:modular} allows us to directly  work out these computations over an (extension) of the finite field.    In \cite{igusa2}, it is proved that $\chi_{18}$ is a modular form of degree $18$ and level $1$ and therefore it induces an element of $\Gamma(\Agthree(\CC), \Hodge^{18})$. Then \cite[Prop.3.4]{ichikawa}  proved that actually $\chi_{18} \in \Gamma(\Agthree(\ZZ), \Hodge^{18})$.  In \cite[Th.1.3.3]{LRZ}, over a number field, and in Proposition \cite[Prop.2.3]{Rit}, over a field $k$ of characteristic different from $2$, it is proved  for a principally polarized abelian threefold $(A,\pol)/k$ and any choice of $k$-rational basis of regular differentials $w_A$ on $A$, that $\chi_{18}(A,\pol, w_A)$ is a non-zero square in $k$ if and only if $(A,\pol)$ is the Jacobian of a non-hyperelliptic curve of genus $3$ over $k$. Using Algorithm~\ref{algo:modular}, we can compute this value and check whether $(A,\pol)$ is the Jacobian of a non-hyperelliptic genus $3$ curve over $k$ without computing the equation of the curve. Note that as we started with $(A,\pol)/\FF_q$ indecomposable, if $\chi_{18}(A,\pol,w_A)=0$, then $(A,\pol)$ is the Jacobian of a hyperelliptic genus $3$ curve over $\FF_q$.

\begin{example}[A unique defect-$0$ curve without non-trivial automorphism] \label{sec:exg3}
Let consider the question of the existence of defect-$0$ curve of genus $3$ over $\FF_q$ with $q=10313$. If there is such a curve $C/\FF_q$ then $\Jac C \sim E^3$ with  $E$ of trace $-m = -203$. 
The curve $E$ has therefore complex multiplication by the maximal order $\OO=\ZZ[\omega]$ of $\QQ(\omega)$ where $\omega=\frac{1+\sqrt{-43}}{2}$. As $\OO$ has class number $1$, there is a unique (non-polarized) abelian variety in the class of $E^3$ up to isomorphism, namely $E^3$ itself. Moreover using Algorithm~\ref{algo:proj} (see also \cite{Sch}), we find $5$ isomorphism classes of indecomposable positive definite unimodular hermitian $\OO$-lattices $(L,h_i)$ leading to $5$ indecomposable principally polarized abelian threefolds $(E^3,a_i)$. In Table~\ref{tab:resultg3}, we give $h$ by its Gram matrix in the canonical basis of $\OO^3$. For each lattice $(L,h_i)$, we also give the smallest odd $\ell$ determined by Algorithm~\ref{algo:ortho2}. Recall that it determines the degree $\ell^3$ of the isogeny we will compute using the Algorithm~\ref{algo:modular}.
We also display in Table~\ref{tab:resultg3} the order of the automorphism group of $(L,h_i)$, and if $\chi:=\chi_{18}(E^3,a_i,w_{E^3})=0$ or if it is a square in $\FF_q$.

We see that only $a_1$ leads to a non-trivial obstruction and therefore to a non-hyperelliptic defect-$0$ curve. This result agrees with the heuristic result which can be deduced from \cite[Table~2]{Rit}. An equation of $C$ is
 \begin{eqnarray*}
 x^4 &+&  7780 x^3 y + 8862x^3  + 456x^2y^2 + 2118x^2y  + 1846x^2  + 5713xy^3  + 10064xy^2 + 7494xy \\ &+& 6469x + 7559y^4 + 9490y^3 + 7458y^2 + 
       214y + 6746=0.\end{eqnarray*}
 Moreover by Torelli theorem \cite[p.790-792]{matsusaka}, since $\Aut(E^3,a_1)  \simeq  \Aut(L,h_1) \simeq \{\pm 1\}$ and $C$ is non-hyperelliptic, the automorphism group of $C$ is trivial. As far as we know, this is the first example of a finite field for which one can ensure  that the defect-$0$ curves have no extra-automorphism. As recalled in \cite{rit-optimal}, most of the methods developed to find curves of genus $3$ with many points use the existence of extra-automorphisms. The question of existence of a defect-$0$ curve over $\FF_{10313}$ could not have been solved in this way.

{\small
\begin{center}
\begin{table}
   \begin{tabular}{|c| c | c | c| c|  c |c| }
     \hline
 Case &     Gram matrix of $h_i$ & $\ell$ & $\#Aut(L,h_i)$ &  Is $\chi=0$? & Is $\chi$ a square? \\ \hline
   1 &  $\begin{pmatrix}       3 & 1 & 1-\overline{\omega}\\ 
       1   &     4 & 2\\
     1-w    &    2     &   5 \end{pmatrix}$ & $11$ & $2$ & no  & yes \\ \hline
2&    $\begin{pmatrix}     3 & 1 +\overline{\omega}&2-\overline{\omega}\\
     1+w    &    5 & -2-\overline{\omega}\\
     2-w   &  -2-w    &    5 \end{pmatrix}$ & $9$ & $12$ &no & no \\ \hline
 3&     $\begin{pmatrix}            2 & -1 & 1\\
      -1    &    4 & 1-\overline{\omega}\\
       1    &  1-w   &     4  \end{pmatrix}$ & $9$ & $4$  &no & no \\ \hline
   4&   $\begin{pmatrix}        3 & 1 & -1-\overline{\omega}\\
       1    &    3 & -1\\
    -1-w    &   -1     &   5   \end{pmatrix}$& $11$ & $4$ &no & no\\ \hline
  5&     $\begin{pmatrix}               3 & -1 & -1-\overline{\omega}\\
      -1    &    3 & 0\\
    -1-w     &   0       & 5    \end{pmatrix}$ & $11$ & $4$ &no & no \\
     \hline
   \end{tabular}
   \caption{\label{tab:resultg3} Example~\ref{sec:exg3}.}
   \end{table}
 \end{center}
}
\end{example}

\begin{example}
Let $q=131$. As previously, the existence of a defect-$0$ curve of genus $3$ over $\FF_q$ leads to consider indecomposable unimodular positive definite $\OO$-lattices $L_i$ of rank $3$, where $\OO$ has discriminant $-40$. The class number of $\OO$ is $2$ and we find 12 $L_i$, out of which 6 are not free and the largest $\ell$ we have to consider is $19$.  We get $11$  defect-$0$ curves of genus $3$ over $\FF_q$ up to $\FF_q$-isomorphism, for instance
\begin{eqnarray*}
x^4 &+&  72 x^3 y + 111 x^3 z + 55 x^2 y^2 + 99 x^2 y z + 47 x^2 z^2 + 8 x y^3 + 
    95 x y^2 z \\ & +&  74 x y z^2 + 30 x z^3 + 39 y^4 + 53 y^3 z + 58 y^2 z^2 +
    40 y z^3 + 59 z^4 =0
    \end{eqnarray*}
   which has an automorphism group of order $2$.
\end{example}

\begin{example}
Let $q=97$. As previously, the existence of a defect-$0$ curve of genus $3$ over $\FF_q$ leads to consider indecomposable unimodular positive definite $R$-lattices of rank $3$, where $R$ has discriminant $-27$ and therefore is not the maximal order of $\textrm{Frac}(R)$. Our algorithms finds $4$ indecomposable unimodular positive definite $R$-lattices and there is one lattice which is not projective, namely $R^2 \oplus \OO$. This leads to $4$ indecomposable principally polarized abelian threefolds over $\FF_q$ isogenous to $E^3$ where $E/\FF_q : y^2= x^3 + 92 x + 10$. 
For three of them, Serre's obstruction is trivial, so we get exactly three defect-$0$ curves of genus $3$ over $\FF_q$ up to $\FF_q$-isomorphism for instance
\begin{eqnarray*}
 x^4 &+&  63 x^3 y + 28 x^3 z + 10 x^2 y^2 + 81 x^2 y z + 43 x^2 z^2 + 89 x y^3 
        + 10 x y^2 z + 70 x y z^2 + 45 x z^3 \\
        &+& 24 y^4 
        +  55 y^3 z + 77 y^2 z^2 + 
        35 y z^3 + 54 z^4=0
\end{eqnarray*}
with an automorphism group of order $6$. 
\end{example}

%% file: sections/examplesg4.tex
 \subsection{Curves of genus 4} \label{sec:g4}
Jacobians of curves of genus $4$ are not dense in the moduli space $\Agfour$. They form a codimension-$1$ variety which we shall characterize thanks to the Igusa modular form $J$ of level $1$ and weight $8$. The modular form  $J$ is  defined over $\CC$ as a homogeneous polynomial of degree $16$ in the theta constants with integer coefficients, see for instance \cite[p.538]{igusa-schottly} or in \cite{igusa-quad} (with the choice of characteristics from \cite{chua}).  It is therefore an element of $\Gamma(\Agfour(\ZZ),\Hodge^8)$ and its values can be computed using Algorithm~\ref{algo:modular}. We will also need the following result below. In \cite{gerritzen}, the first term in the Fourier expansion of $J$ is computed and its constant coefficient is $-2^{16}$. This means that the Siegel modular form $J$ does not vanish identically on $\Agfour \otimes k$ for any algebraically closed field $k$ of characteristic different from $2$. 

Igusa proves that the Igusa modular form is related to the classical Schottky modular form by $$J= \frac{1}{2^6 \cdot 3^2 \cdot 5 \cdot 7} \cdot \left(\left(\sum \vartheta[\epsilon](\tau)^8\right)^2 - 2^{4} \sum \vartheta[\epsilon](\tau)^{16}\right)$$
the sums being over all even characteristics. Hence, over $\CC$, this form is zero precisely on the locus of principally polarized abelian varieties of dimension $4$ which are decomposable or a Jacobian. Following the same lines as \cite[Prop.2.3]{Rit}, this can be extended to any field of characteristic different from $2$.

\begin{theorem} \label{th:igusa}
Let $(A,\pol)$ be an indecomposable principally polarized abelian variety of dimension $4$ over an algebraically closed field $k$ of characteristic different from $2$ and $w_A$ a basis of regular differentials. Then $J(A,\pol,w_A)=0$ if and only if $(A,\pol)$ is the Jacobian of a curve of genus $4$ over $k$.
\end{theorem}
\begin{proof}
  Let $\Agfour$ be the moduli stack of principally polarized abelian schemes of relative dimension $4$ and let us denote by $\mathcal{T}$  the Torelli locus (the image of the moduli stack of genus $4$ curves of compact type). Following \cite[p.554]{moonenoort}, it is a reduced and closed substack of $\Agfour$. Moreover for any algebraically closed field $k$, $\mathcal{T}(k)$  coincides with the disjoint union of the set of Jacobians of genus $4$ curves and the set of decomposable principally polarized abelian varieties of dimension $4$ defined over $k$. 

  Over $\CC$,  $\mathcal{T}(\CC) = (J=0)_{red}(\CC)$.
This shows that $\mathcal{T} \otimes \QQ = (J=0)_{red} \otimes \QQ$. Taking  the schematic closure 
 over
$\ZZ[\frac{1}{2}]$ we get $\mathcal{T} \supset \overline{\mathcal{T} \otimes
\QQ} =\overline{(J=0)_{red} \otimes \QQ} \subset (J=0)_{red}$ in $\Agfour$
We need to prove that the two inclusions are equalities, i.e. that none of the loci $\mathcal{T}$ or $(J=0)_{red}$ has a vertical component. For $(J=0)_{red}$ this is the case since the
modular form $J \in \Gamma(\Agfour \otimes
\ZZ[\frac{1}{2}],\Hodge^{\otimes 8})$ is primitive  and the fibers of $\Agfour$ are irreducible (see for instance the proof of \cite[Lemma~3.2, p.~163]{faltingschai}). Similarly, for $\mathcal{T}$, this is true because we can lift any genus $4$ curve in a special fiber to characteristic $0$.

From this we deduce that $J(A,\pol,w_A) =0$ if and only if $(A,\pol) \in \mathcal{T} \otimes k$. Since we have assumed that the polarization $\pol$ is indecomposable, this is the case if and only if $(A,\pol)$ is a Jacobian.
\end{proof}

As we only need to check if the value of $J$ is zero or not, we can work with any affine lift of the theta null point. However, if it is zero and $(A,\pol)$ is therefore a Jacobian over the algebraic closure, there is currently no way to ensure that it is also a Jacobian over the ground field.

\begin{example} \label{ex:g4}
Let us consider the case of defect-$0$ genus $4$ curves $C$ over $\FF_{59}$. The Jacobian of $C$ would be isogenous to $E^4$ with $E$ an elliptic curve with $\End(E)$ of discriminant $-11$. There are three indecomposable principally polarized abelian varieties in the class of $E^4$. We can check (using for the three of them the value $\ell=3$) that for none of them the Igusa form is $0$. Hence there is no defect-$0$ curve of genus $4$ over $\FF_{59}$ as it is confirmed in the manYPoints tables \cite{manypoints} or \cite[Th.1.1]{zaytsev}.
\end{example}

It would be more interesting to look at one unknown entry of these tables, like for instance $q=89$. However in this case the discriminant of the associated elliptic curve is $32$ and our algorithms are not efficient enough to work it out yet.